\documentclass[12pt,pdftex,aos,noinfoline,addressasfootnote]{imsart}

\RequirePackage[OT1]{fontenc}
\usepackage{amsthm,amsmath,amsfonts,natbib,mathtools,custom_math,amssymb}
\RequirePackage{hypernat}
\usepackage[ruled,section]{algorithm}
\usepackage[noend]{algorithmic}
\usepackage{graphicx}
\usepackage[hscale=0.7,vscale=0.8]{geometry}
\usepackage{subfigure}
\usepackage{graphicx}
\usepackage{hyperref}
\usepackage{yhmath}

\DeclareMathOperator{\diag}{diag}

\def\interleave{|\kern-.25ex|\kern-.25ex|}
\def\interleavesub{|\kern-.15ex|\kern-.15ex|}

\newcommand{\nNorm}[1]{\left|\kern-.25ex\left|\kern-.25ex\left| {#1}\right|\kern-.25ex\right|\kern-.25ex\right|}

\numberwithin{equation}{section}
\theoremstyle{plain}
\newtheorem{theorem}{Theorem}[section]
\newtheorem{stheorem}{Theorem}
\newtheorem{corollary}{Corollary}[section]
\newtheorem{proposition}{Proposition}[section]
\newtheorem{lemma}{Lemma}[section]
\newtheoremstyle{remark}{\topsep}{\topsep}%
     {\normalfont}
     {}           
     {\bfseries}  
     {.}          
     {.5em}       
     {\thmname{#1}\thmnumber{ #2}\thmnote{ #3}}
\theoremstyle{remark}
\newtheorem{remark}{Remark}[section]
\newtheorem{example}{Example}[section]

\newtheorem{definition}{Definition}[section]

\mathchardef\mh="2D

\def\comma{\unskip,~}

\long\def\comment#1{}
\def\reals{{\mathbb R}}
\def\R{\reals}
\def\P{{\mathbb P}}
\def\E{{\mathbb E}}

\def\supp{\mathop{\text{supp}\kern.2ex}}
\def\argmin{\mathop{\text{\rm arg\,min}}}

\let\tilde\widetilde

\let\hat\widehat
\let\tilde\widetilde
\def\given{{\,|\,}}

\def\1{{(1)}}
\def\2{{(2)}}

\long\def\comment#1{}

\def\diag{\mathop{\rm diag}}

\def\threebars{\mbox{$|\kern-.25ex|\kern-.25ex|$}}

\def\coloneqq{:=}
\def\mathbf#1{\mbox{\boldmath $#1$}} 
\def\bar#1{\overline{#1}}

\arxiv{1411.1805}

\begin{document}

\begin{frontmatter}
\title{Faithful Variable Screening for\\ High-Dimensional Convex Regression}
\runtitle{Faithful Variable Screening for Convex Regression}

\begin{aug}
\vskip10pt
\author{\fnms{Min} \snm{Xu${}^*$}\ead[label=e1]{minx@cs.cmu.edu}}
\comma 
\author{\fnms{Minhua} \snm{Chen${}^\dag$}\ead[label=e2]{minhua@galton.uchicago.edu}}
\and
\author{\fnms{John}
  \snm{Lafferty${}^\dag$}\ead[label=e3]{lafferty@galton.uchicago.edu}}
\address{${}^*$Machine Learning Department, Carnegie Mellon University,
  Pittsburgh PA 15213}
\address{${}^\dag$Department of Statistics, University of Chicago, Chicago IL 60637}
\end{aug}

\begin{abstract}
  We study the problem of variable selection in convex nonparametric
  regression.  Under the assumption that the true regression function
  is convex and sparse, we develop a screening procedure to select a
  subset of variables that contains the relevant variables. Our
  approach is a two-stage quadratic programming method that estimates
  a sum of one-dimensional convex functions, followed by
  one-dimensional concave regression fits on the residuals.  In
  contrast to previous methods for sparse additive models, the
  optimization is finite dimensional and requires no tuning parameters
  for smoothness. Under appropriate assumptions, we prove that the
  procedure is faithful in the population setting, yielding no false
  negatives.  We give a finite sample statistical analysis, and
  introduce algorithms for efficiently carrying out the required
  quadratic programs. The approach leads to computational and
  statistical advantages over fitting a full model, and provides an
  effective, practical approach to variable screening in convex
  regression.
\end{abstract}


\vskip20pt 
\end{frontmatter}

\maketitle

\vskip10pt


\section{Introduction}

Shape restrictions such as monotonicity, convexity, and concavity
provide a natural way of limiting the complexity of many statistical
estimation problems.  Shape-constrained estimation is not as well
understood as more traditional nonparametric estimation involving
smoothness constraints.  For instance, the minimax rate of convergence
for multivariate convex regression has yet to be rigorously
established in full generality.  Even the one-dimensional case is
challenging, and has been of recent interest \citep{guntusen:13}.

In this paper we study the problem of variable selection in
multivariate convex regression.  Assuming that the regression function
is convex and sparse, our goal is to identify the relevant variables.
We show that it suffices to estimate a sum of one-dimensional convex
functions, leading to significant computational and statistical
advantages.  This is in contrast to general nonparametric regression,
where fitting an additive model can result in false negatives.  Our
approach is based on a two-stage quadratic programming procedure.  In
the first stage, we fit an convex additive model, imposing a sparsity
penalty.  In the second stage, we fit a concave function on the
residual for each variable.  As we show, this non-intuitive second
stage is in general necessary.  Our first result is that this
procedure is faithful in the population setting, meaning that it
results in no false negatives, under mild assumptions on the density
of the covariates.  Our second result is a finite sample statistical
analysis of the procedure, where we upper bound the statistical rate
of variable screening consistency.  An additional contribution is to show how the
required quadratic programs can be formulated to be more scalable.  We
give simulations to illustrate our method, showing that it performs in
a manner that is consistent with our analysis.

Estimation of convex functions arises naturally in several
applications.  Examples include geometric programming \citep{Boyd04},
computed tomography \citep{Prince:90}, target reconstruction
\citep{Lele:92}, image analysis \citep{Golden:06} and circuit design
\citep{Hannah:12}.  Other applications include queuing theory
\citep{Chen:01} and economics, where it is of interest to estimate
concave utility functions \citep{Pratt:68}.  See \cite{Lim:12} for
other applications.  
Beyond cases where the assumption of convexity is
natural, the convexity assumption can be attractive as a
tractable, nonparamametric relaxation of the linear model.  

Recently, there has been increased research activity on shape-constrained estimation. \cite{guntusen:13} analyze univariate convex regression and show surprisingly that the risk of the MLE is adaptive to the complexity of the true function. \cite{seijo2011nonparametric} and \cite{Lim:12} study maximum likelihood estimation of multivariate convex regression and independently establish its consistency. \cite{Cule:10} and \cite{kim2014global} analyze log-concave density estimation and prove consistency of the MLE; the latter further show that log-concave density estimation has minimax risk lower bounded by $n^{-2/(d+1)}$ for $d \geq 2$, refuting a common notion that the condition of convexity is equivalent, in estimation difficulty, to the condition of having two bounded derivatives. Additive shape-constrained estimation has also been studied; \cite{pya2014shape} propose a penalized B-spline estimator while \cite{chen2014generalised} show the consistency of the MLE. To the best of our knowledge however, there has been no work on variable selection and and estimation of high-dimensional convex functions.

Variable selection in general nonparametric regression or function
estimation is a notoriously difficult problem. \citet{lafferty2008rodeo} develop a greedy procedure for
adjusting bandwidths in a local linear regression estimator,
and show that the procedure achieves the minimax rate
as if the relevant variables were isolated in advance.
But the method only provably scales to dimensions $p$ that 
grow logarithmically in the sample size $n$, i.e., $p = O(\log n)$.  This
is in contrast to the high dimensional scaling behavior
known to hold for sparsity selection in linear models
using $\ell_1$ penalization, where $n$
is logarithmic in the dimension $p$. \citet{bertin:08}
develop an optimization-based approach in
the nonparametric setting, applying the lasso
in a local linear model at each test point.  Here again,
however, the method only scales as $p = O(\log n)$,
the low-dimensional regime.
An approximation theory approach to the same
problem is presented in \cite{devore:11}, 
using techniques based on hierarchical hashing schemes,
similar to those used for ``junta'' problems \citep{mossel:04}.
Here it is shown that the sample complexity scales as $n > \log p$ 
if one adaptively selects the points on
which the high-dimensional function is evaluated.

\citet{dalalyan:12} show that the exponential scaling $n=O(\log p)$ is
achievable if the underlying function is assumed to be smooth with
respect to a Fourier basis. They also give support for the intrinsic
difficulty of variable selection in nonparametric regression, giving
lower bounds showing that consistent variable selection is not
possible if $n < \log p$ or if $n < \exp s$, where $s$ is the number
of relevant variables.  Variable selection over kernel classes is
studied by \citet{Kolch:10}.

Perhaps more closely related to the present work is the framework
studied by \cite{Raskutti:12} for sparse additive models, where sparse
regression is considered under an additive assumption, with each
component function belonging to an RKHS.  An advantage of working over
an RKHS is that nonparametric regression with a sparsity-inducing
regularization penalty can be formulated as a finite dimensional
convex cone optimization.  On the other hand, smoothing parameters for
the component Hilbert spaces must be chosen, leading to extra tuning
parameters that are difficult to select in practice. There has also been
work on estimating sparse additive models over a spline basis, for 
instance the work of \cite{huang2010variable}, but these approaches too 
require the tuning of smoothing parameters. 

While nonparametric, the convex regression problem is naturally
formulated using finite dimensional convex optimization, with no
additional tuning parameters. The convex additive model can be used
for convenience, without assuming it to actually hold, for the purpose
of variable selection. As we show, our method scales to high
dimensions, with a dependence on the intrinsic dimension $s$ that
scales polynomially, rather than exponentially as in the general case
analyzed in \cite{dalalyan:12}.

In the following section we give a high-level summary of our technical
results, including additive faithfulness, variable selection 
consistency, and high dimensional scaling.  In
Section~\ref{sec:additivefaithful} we give a detailed account
of our method and the conditions under which we can guarantee
consistent variable selection.  In Section~\ref{sec:optimization}
we show how the required quadratic programs can be reformulated
to be more efficient and scalable.  In Section~\ref{sec:finitesample}
we give the details of our finite sample analysis, showing
that a sample size growing as $n = O\big(\textrm{poly}(s) \log p\big)$
is sufficient for variable selection.  In Section~\ref{sec:thesims}
we report the results of simulations that illustrate our methods
and theory.  The full proofs are
given in a technical appendix.



\def\x{\mathbf{x}}

\section{Overview of Results}

In this section we provide a high-level description of our technical
results.  The full technical details, the precise statement of the
results, and their detailed proofs are provided in following sections.

Our main contribution is an analysis of an additive approximation for identifying
relevant variables in convex regression.  
We prove a result that shows when and how the additive approximation
can be used without introducing false negatives in the population
setting.  In addition, we develop algorithms for the efficient implementation of
the quadratic programs required by the procedure.  

We first establish some notation, to be used throughout the  paper.
If $\mathbf{x}$ is a vector, we use $\mathbf{x}_{-k}$ to denote the
vector with the $k$-th coordinate removed. If $\mathbf{v} \in \R^n$, then
$v_{(1)}$ denotes the smallest coordinate of $\mathbf{v}$ in
magnitude, and $v_{(j)}$ denotes the $j$-th smallest; $\mathbf{1}_n \in \R^n$
is the all ones vector. If $X \in \R^p$ is a random variable and $S \subset
\{1,...,p\}$, then $X_S$ is the subvector of $X$ restricted to
the coordinates in $S$. Given $n$ samples $X^{(1)},...,X^{(n)}$, we use
$\bar{X}$ to denote the sample mean. Given a random variable
$X_k$ and a scalar $x_k$, we use $\E[\,\cdot \given x_k]$ as a shorthand
for $\E[\, \cdot \given X_k = x_k]$.

\subsection{Faithful screening}

The starting point for our approach is the observation that least squares
nonparametric estimation under convexity constraints is equivalent to
a finite dimensional quadratic program.  Specifically, the infinite
dimensional optimization 
\begin{align}
\begin{split}
\text{minimize} & \quad \sum_{i=1}^n (Y_i - f(\x_i))^2 \\
\text{subject to} &  \quad f:\reals^p\rightarrow\reals\ \text{is
  convex}
\end{split}
\end{align}
is equivalent to the finite dimensional quadratic
program 
\begin{align}
\label{eq:convreg}
\begin{split}
\text{minimize}_{f, \beta} & \;\; \sum_{i=1}^n (Y_i - f_i)^2 \\
\text{subject to} & \;\; f_j \geq f_i + \beta_i^T (\x_j-\x_i),\; \text{for
    all $i,j$}.
\end{split}
\end{align}
Here $f_i$ is the estimated function value $f(\x_i)$, and the vectors
$\beta_i \in \reals^d$ represent supporting hyperplanes to the
epigraph of $f$.  See \cite{Boyd04}, Section 6.5.5.
Importantly, this finite dimensional quadratic program does
not have tuning parameters for smoothing the function. 

This formulation of convex regression is subject to the curse
of dimensionality.  Moreover, attempting to select variables
by regularizing the subgradient vectors $\beta_i$ with
a group sparsity penality is not effective.  Intuitively, the reason
is that all $p$ components of the subgradient $\beta_i$
appear in every convexity constraint 
$ f_j \geq f_i + \beta_i^T (\x_j-\x_i)$; small changes to the 
subgradients may not violate the constraints.  Experimentally,
we find that regularization with a group sparsity penality
will make the subgradients of irrelevant variables 
small, but may not zero them out completely.

This motivates us to consider an additive approximation. 
Under a convex additive model, each component of the subgradient
appears only in the convexity constraint for
the corresponding variable:
\begin{equation}
f_{ki'} \geq f_{ki} + \beta_{ki}(x_{ki'}-x_{ki}) 
\end{equation}
where $f_{ki} = f_{k}(x_{ki})$ and $\beta_{ki}$ is the
subgradient at point $x_{ki}$.   As we show, this leads to
an effective variable selection procedure.  The
shape constraints play an essential role.
For general regression, using an additive approximation for variable
selection may make errors.  In particular, the nonlinearities in the
regression function may result in an additive component being wrongly
zeroed out.  We show that this cannot happen for convex regression
under appropriate conditions.

We say that a differentiable function $f$ depends on variable $x_k$ if
$\partial_{x_k} f \neq 0$ with probability greater than zero.  An additive approximation is given by
\begin{equation}
\{f_k^*\}, \mu^* \coloneqq \argmin_{f_1,\ldots, f_p, \mu} \Bigl\{ 
             \E \Bigl( f(X) - \mu - \sum_{k=1}^p f_k(X_k)\Bigr)^2 
         \,:\, \E f_k(X_k) = 0 \Bigr\}.
\end{equation}
We say that $f$ is \textit{additively faithful} in case $f^*_k = 0$
implies that $f$ does not depend on coordinate $k$.
Additive faithfulness is a desirable property
since it implies that an additive approximation may allow us to 
screen out irrelevant variables.

Our first result shows that convex multivariate functions are
additively faithful
under the following assumption on the distribution of the data.
\begin{definition}
  Let $p(\mathbf{x})$ be a density supported on $[0,1]^p$.  Then $p$
  satisfies the \emph{boundary flatness condition} if for all $j$, and
  for all $\mathbf{x}_{-j}$,
\[
\frac{\partial p(\mathbf{x}_{-j} \given x_j)}{\partial x_j}  =  
\frac{\partial^2 p(\mathbf{x}_{-j} \given x_j)}{\partial x_j^2} = 0
\quad \trm{at $x_j = 0$ and $x_j = 1$}.
\]
\end{definition}

As discussed in Section~\ref{sec:additivefaithful}, this is a relatively weak condition. 
Our first result is that this condition suffices in the population
setting of convex regression.

\begin{stheorem}
  Let $p$ be a positive density supported on $C=[0,1]^p$ that
  satisfies the boundary flatness property. If $f$ is convex and twice
  differentiable, then $f$ is additively faithful under $p$.
\end{stheorem}


Intuitively, an additive approximation zeroes out variable $k$ when, fixing $x_k$, every
``slice'' of $f$ integrates to zero. We prove this result
by showing that ``slices'' of convex
functions that integrate to zero cannot be ``glued together'' while
still maintaining convexity.

While this shows that convex functions are additively faithful, it is difficult to
estimate the optimal additive functions.  The difficulty
is that $f^*_k$ need not be
a convex function, as we show through a counterexample
in Section~\ref{sec:additivefaithful}. It may be possible to estimate $f^*_k$ with smoothing parameters, but, for the purpose of variable screening, it is sufficient in fact to approximate $f^*_k$ by a \emph{convex} additive model. 

Our next result states that a convex additive fit, combined with a series of univariate concave fits, is faithful. We abuse notation in Theorem~\ref{thm:summary_acdc_population} and let the notation $f^*_k$ represent convex additive components.

\begin{stheorem}
\label{thm:summary_acdc_population}
Suppose $p(\mathbf{x})$ is a positive
density on $C=[0,1]^p$ that satisfies the boundary flatness
condition. Suppose that $f$ is convex and twice-differentiable.
and that $\partial_{x_k} f$, $\partial_{x_k} p( \mathbf{x}_{-k}
\given x_k )$, and $\partial_{x_k}^2 p( \mathbf{x}_{-k} \given x_k)$
are all continuous as functions on $C$.
Define
\begin{equation}
\{ f^*_k \}_{k=1}^p,\mu^* = \arg\min_{\{f_k\},\mu} \Big \{
\E\Bigl( f(X) - \mu - \sum_{k=1}^s f_k(X_k) \Bigr)^2 \,:\, f_k \in
\mathcal{C}^1, \, \E f_k(X_k) = 0 \Big \}
\end{equation} where $\mathcal{C}^1$ is the set of univariate convex
functions, and, with respective to $f^*_k$'s from above, define
\begin{equation}
g^*_k = \arg\min_{g_k} \Big\{ \E\Bigl( f(X) - \mu^* - 
\sum_{k' \neq k} f^*_{k'}(X_{k'}) - g_k \Bigr)^2 \,:\, g_k \in \mh
\mathcal{C}^1, \E g_k(X_k) = 0 \Big\},
\end{equation}
with $\mh{}\mathcal{C}^1$ denoting the set of univariate concave
functions.  Then $f^*_k = 0$ and $g^*_k = 0$ implies that $f$ does not
depend on $x_k$, i.e., $\partial_{x_k} f(\mathbf{x}) = 0$ with
probability one.
\end{stheorem}

This result naturally suggests a two-stage screening
procedure for variable selection. In the first stage we fit a sparse convex
additive model $\{\hat f_k\}$.  In the second stage we
fit a concave function $\hat g_k$ to the residual for each variable
having a zero convex component $\hat f_k$.  If both $\hat f_k = 0$ and
$\hat g_k = 0$, we can safely discard variable $x_k$.  
As a shorthand, we refer to this two-stage procedure as AC/DC.  In 
the AC stage we fit an additive convex model.  In the DC 
stage we fit decoupled concave functions on the residuals.  The
decoupled nature of the DC stage allows all of the fits to
be carried out in parallel. The entire process involves no smoothing parameters.
Our next results concern the required optimizations, and their finite
sample statistical performance.

\subsection{Optimization}

In Section~\ref{sec:optimization} we present optimization algorithms for the additive convex regression stage.
The convex constraints for the additive functions, analogous to 
the multivariate constraints \eqref{eq:convreg},
are  that each component $f_{k}(\cdot)$ 
can be represented by its supporting hyperplanes, i.e.,
\begin{equation}
      f_{ki'} \geq f_{ki} + \beta_{ki}(x_{ki'}-x_{ki}) \quad \text{for
        all $i,i'$}
\end{equation}
where $f_{ki}\coloneqq f_{k}(x_{ki})$ and $\beta_{ki}$ is the
subgradient at point $x_{ki}$. While this apparently requires $O(n^2
p)$ equations to impose the supporting hyperplane constraints, 
in fact, only $O(np)$ constraints suffice.  This is because univariate convex functions are
characterized by the condition that the subgradient, which is a scalar, must
increase monotonically. This observation leads to a reduced quadratic
program with $O(np)$ variables and $O(np)$ constraints. 

Directly applying a QP solver to this optimization is still computationally
expensive for relatively large
$n$ and $p$.  We thus develop a block
coordinate descent method, where in each step we solve a sparse
quadratic program involving $O(n)$ variables and $O(n)$ constraints.  This 
is efficiently solved using optimization packages 
such as {\sc mosek}.  The details of these optimizations
are given in Section~\ref{sec:optimization}.

\subsection{Finite sample analysis}

In Section~\ref{sec:finitesample} 
we analyze the finite sample variable selection consistency of convex
additive modeling, without making the assumption that the true
regression function $f_0$ is additive.  Our analysis first establishes
a sufficient deterministic condition for variable selection 
consistency, and then considers a stochastic setting.
Our proof technique decomposes the KKT conditions for the optimization
in a manner that is similar to the now standard \emph{primal-dual
  witness} method~\citep{wainwright2009sharp}. 

We prove separate results that allow us to analyze false negative
rates and false positive rates.  To control false positives,
we analyze scaling conditions on the regularization parameter
$\lambda_n$ for 
group sparsity needed to zero out irrelevant variables
$k \in S^c$, where $S\subset \{1,\ldots, p\}$ is the set of
variables selected by the AC/DC algorithm in the population setting.
To control false negatives, we analyze the restricted regression
where the variables in $S^c$ are zeroed out, following the primal-dual
strategy.  

Each of our theorems uses a subset of the following assumptions:
\begin{packed_enum}
\item[A1:] $X_S, X_{S^c}$ are independent. 
\item[A2:] $f_0$ is convex and twice-differentiable. 
\item[A3:] $\|f_0\|_\infty \leq sB$ and $\|f^*_k \| \leq B$ for all $k$.
\item[A4:] The noise is mean-zero sub-Gaussian, independent of $X$.
\end{packed_enum}
In Assumption A3, $f^*=\sum_k f^*_k$ denotes the optimal additive projection of $f_0$ in the population setting.

Our analysis involves parameters $\alpha_+$ and $\alpha_-$,
which are measures of the signal strength of the weakest variable:
\begin{align*}
\alpha_+ &= \inf_{f \in \mathcal{C}^p \,:\, \textrm{supp}(f)\subsetneq \textrm{supp}(f^*)} 
       \Big\{ \mathbb{E} \big( f_0(X) - f(X) \big)^2 - 
        \mathbb{E} \big( f_0(X) - f^*(X) \big)^2  \Big\}\\
\alpha_- &=   \min_{k \in S \,:\, g^*_k \neq 0}
      \Big\{ \mathbb{E} \big( f_0(X) - f^*(X) \big)^2 - 
    \mathbb{E} \big( f_0(X) - f^*(X) - g^*_k(X_k) \big)^2 \Big\}.
\end{align*}

 Intuitively, if $\alpha_+$ is small, then it is easier to make a
false omission in the additive convex stage of the procedure. If
$\alpha_-$ is small, then it is easier to make a false omission in
the decoupled concave stage of the procedure.

We make strong assumptions on the covariates in A1 in order to make
very weak assumptions on the true regression function $f_0$ in
A2; in particular, we do not assume that $f_0$ is additive. 
Relaxing this condition is an important direction for future work.
We also include an extra
boundedness constraint to use new bracketing number
results \citep{kim2014global}.

Our main result is the following.
\begin{stheorem}
Suppose assumptions A1-A4 hold. Let $\{\hat{f}_i\}$ be any AC solution and
let $\{\hat{g}_k\}$ be any DC solution, both estimated with 
regularization parameter $\lambda$ scaling as
$\lambda = \Theta \Big( sB \sqrt{\frac{1}{n} \log^2 np} \Big)$. 
Suppose in addition that
\begin{gather}
\alpha_f/\tilde{\sigma} \geq c B^2 \sqrt{\frac{s^5}{n^{4/5}} \log^2
  np}\\
\alpha_g^2/\tilde{\sigma} \geq c B^4 \sqrt{\frac{s^5}{n^{4/5}}
  \log^2 2np}.
\end{gather} 
where $\tilde{\sigma} \equiv \max(\sigma, B)$ and $c$ is a constant dependent only on $b, c_1$.

Then, for sufficiently large $n$, with probability at least $1-\frac{1}{n}$:
\begin{align*}
\hat{f}_k \neq 0 \trm{ or } \hat{g}_k \neq 0 &\trm{ for all } k \in S\\
\hat{f}_k = 0 \trm{ and } \hat{g}_k = 0 & \trm{ for all } k \notin S.
\end{align*}


\end{stheorem}

This shows that variable selection consistency is achievable under
exponential scaling of the ambient dimension, $p = O(\exp(n^c))$
for $c<1$, as for linear models. The cost of nonparametric estimation is
reflected in the scaling with respect to $s=|S|$, which can grow only
as $o(n^{4/25})$.

We remark that \citet{dalalyan:12} show that, even with the product distribution,
 under traditional smoothness
constraints, variable selection is achievable only if $n > O(e^s)$. 
Here we demonstrate that convexity yields the scaling $n =
O(\textrm{poly}(s))$.

\section{Additive Faithfulness}
\label{sec:additivefaithful}

For general regression, an additive approximation may result in a
relevant variable being incorrectly marked as irrelevant. Such
mistakes are inherent to the approximation and may persist even in
the population setting.  In this section we give
examples of this phenomenon, and then show how the convexity
assumption
changes the behavior of the additive approximation. We begin
with a lemma that characterizes the components of the additive approximation under mild conditions.


\begin{lemma}
\label{lem:general_int_reduction}
Let $F$ be a distribution on $C=[0,1]^p$ with a positive density
function $p$. Let $f: C \rightarrow \R$ be an integrable function,
and define 
\begin{align*}
f^*_1,...,&f^*_p, \mu^* \coloneqq  \\
&\arg\min \left\{ \E \Bigl( f(X) - \mu - \sum_{k=1}^p f_k(X_k)\Bigr)^2 \,:\,
\E f_k(X_k) = 0,\; \forall k=1,\ldots, p \right\}.
\end{align*}
Then $\mu^* = \E f(X)$, 
\begin{equation}
\label{eq:backfit}
f^*_k(x_k) = \E\Bigl[ f(X) - \sum_{k' \neq k} f^*_{k'}(X_{k'}) \given
x_k\Bigr] - \E f(X), 
\end{equation}
and this solution is unique.
\end{lemma}

Lemma~\ref{lem:general_int_reduction} follows from the stationarity
conditions of the optimal solution.   This result is known, and
criterion \eqref{eq:backfit} is used in the backfitting
algorithm for fitting additive models.   We include 
a proof as our results build on it.

\begin{proof}
  Let $f^*_1,...,f^*_p, \mu^*$ be the minimizers as defined.  We first
  show that the optimal $\mu$ is $\mu^* = \E f(X)$ for any $f_1, ...,
  f_k$ such that $\E f_k(X_k) = 0$. This follows from the stationarity
  condition, which states that $\mu^* = \E[ f(X) - \sum_k f_k(X_k)] =
  \E[ f(X) ]$. Uniqueness is apparent because the second derivative is
  strictly larger than zero and strong convexity is guaranteed.

  We now turn our attention toward the $f^*_k$s.  It must be that
  $f^*_k$ minimizes 
\begin{equation}
\E\Bigl[ \big( f(X) - \mu^* - \sum_{k' \neq k}
  f^*_{k'} (X_{k'}) - f_k (X_k) \big)^2\Bigr]
\end{equation}
subject to $\E f_k(X_k) = 0$.
Fixing $x_k$, we will show that the value 
\begin{equation}
\E[ f(X) - \sum_{k' \neq k}
f_{k'}(X_{k'}) \given x_k] - \mu^*
\end{equation} 
uniquely minimizes
\begin{equation}
\min_{ f_k(x_k) } \int_{\mathbf{x}_{-k}} p(\mathbf{x}) 
         \Big( f(\mathbf{x}) - \sum_{k' \neq k} f^*_{k'} (x_{k'}) - f_k (x_k) -\mu^*\Big)^2 
                 d \mathbf{x}_{-k}.
\end{equation}
The first-order optimality condition gives us
\begin{align}
\int_{\mathbf{x}_{-k}} p(\mathbf{x}) f_k(x_k) d \mathbf{x}_{-k} &= 
  \int_{\mathbf{x}_{-k}} p(\mathbf{x}) 
      ( f(\mathbf{x})-\sum_{k' \neq k} f^*_{k'}(x_{k'})-\mu^*) d \mathbf{x}_{-k} \\  
p(x_k) f_k(x_k) &= \int_{\mathbf{x}_{-k}} p(x_k)
     p(\mathbf{x}_{-k} \given x_k ) 
     ( f(\mathbf{x}) - \sum_{k' \neq k} f^*_{k'} (x_{k'})-\mu^*) 
              d \mathbf{x}_{-k} \\
f_k(x_k) &= \int_{\mathbf{x}_{-k}} 
       p(\mathbf{x}_{-k} \given x_k ) 
     (f(\mathbf{x}) - \sum_{k'\neq k} f_{k'} (x_{k'})  -\mu^*) d \mathbf{x}_{-k} 
 \end{align}
The square error objective is strongly convex,
and the second derivative with respect to $f_k(x_k)$ is $2 p(x_k)$, which is always positive under the assumption that $p$ is positive. Therefore, the solution $f^*_k(x_k) = \E[ f(X) \given x_k ] - \E f(X)$ is unique.
Noting that $\E[ f(X) -\sum_{k'\neq k} f_{k'}(X_{k'}) | x_k] - \E f(X)$ has mean zero as a function of $x_k$
completes the proof.
\end{proof}

In the case that the distribution in
Lemma~\ref{lem:general_int_reduction} is a product distribution, 
the additive components take on a simple form.

\begin{corollary}
\label{cor:product_int_reduction}
Let $F$ be a product distribution on $C=[0,1]^p$ with density function
$p$ which is positive on $C$. 
Let $\mu^*, f^*_k(x_k)$ be defined as in Lemma~\ref{lem:general_int_reduction}.
Then $\mu^* = \E f(X)$ and $f^*_k(x_k) = \E[ f(X) \given x_k] - \E f(X)$ and this solution is unique.
\end{corollary}

In particular, if $F$ is the uniform distribution,
then $f^*_k(x_k) = \displaystyle\int f(x_k, \mathbf{x}_{-k})
d\mathbf{x}_{-k}$.

\begin{example} Using Corollary~\ref{cor:product_int_reduction}, we
  give two examples of \emph{additively unfaithfulness} under the
  uniform distribution---where relevant variables are
  erroneously marked as irrelevant under an additive
  approximation. First, consider the following function:
\begin{equation}
f(x_1, x_2) = \sin( 2\pi x_1) \sin( 2 \pi x_2)\quad
\trm{(egg carton)} 
\end{equation}
defined for $(x_1, x_2) \in [0,1]^2$.  Then
$\displaystyle\int_{x_2} f(x_1, x_2) d x_2 = 0$ and
$\displaystyle\int_{x_1} f(x_1, x_2) d x_1 = 0$ for each $x_1$ and $x_2$. An additive approximation
would set $f_1 = 0$ and $f_2 = 0$.  Next, consider the function
\begin{equation}
f(x_1, x_2) = x_1 x_2 \quad \trm{(tilting slope)} 
\end{equation}
defined for $x_1 \in [-1,1],\; x_2 \in [0,1]$.  In this case
$\displaystyle\int_{x_1} f(x_1, x_2) d x_1 = 0$ for each $x_2$; therefore, we expect $f_2 = 0$ under the additive approximation. This function, for every fixed $x_2$, is a zero-intercept linear function of $x_1$ with slope $x_2$.
\end{example}

\begin{figure*}[htp]
\vskip-10pt
	\centering
	\subfigure[egg carton]{
		\centering
		{\includegraphics[width=0.4\textwidth]{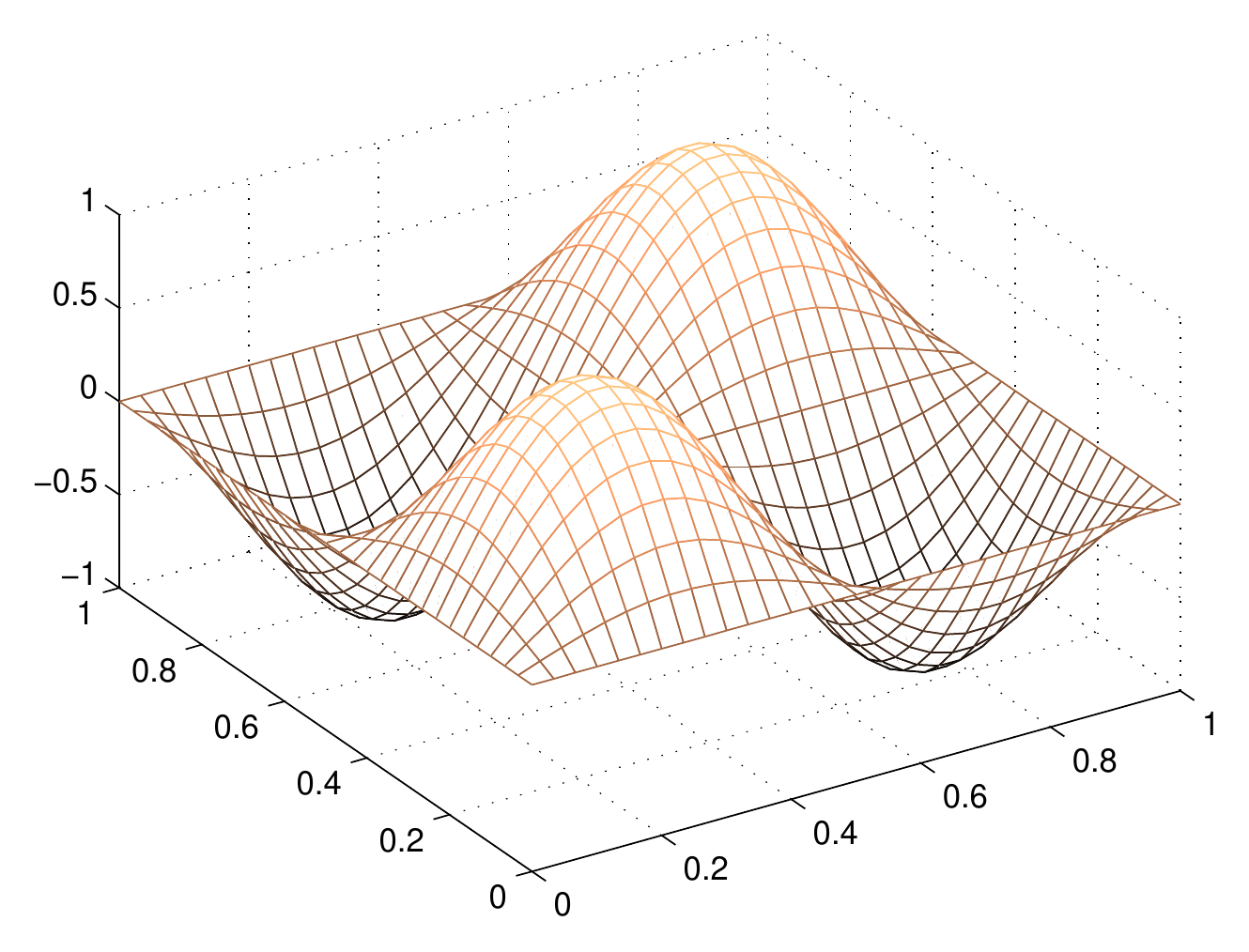}}
	}
	\subfigure[tilting slope]{
		\centering
		{\includegraphics[width=0.4\textwidth]{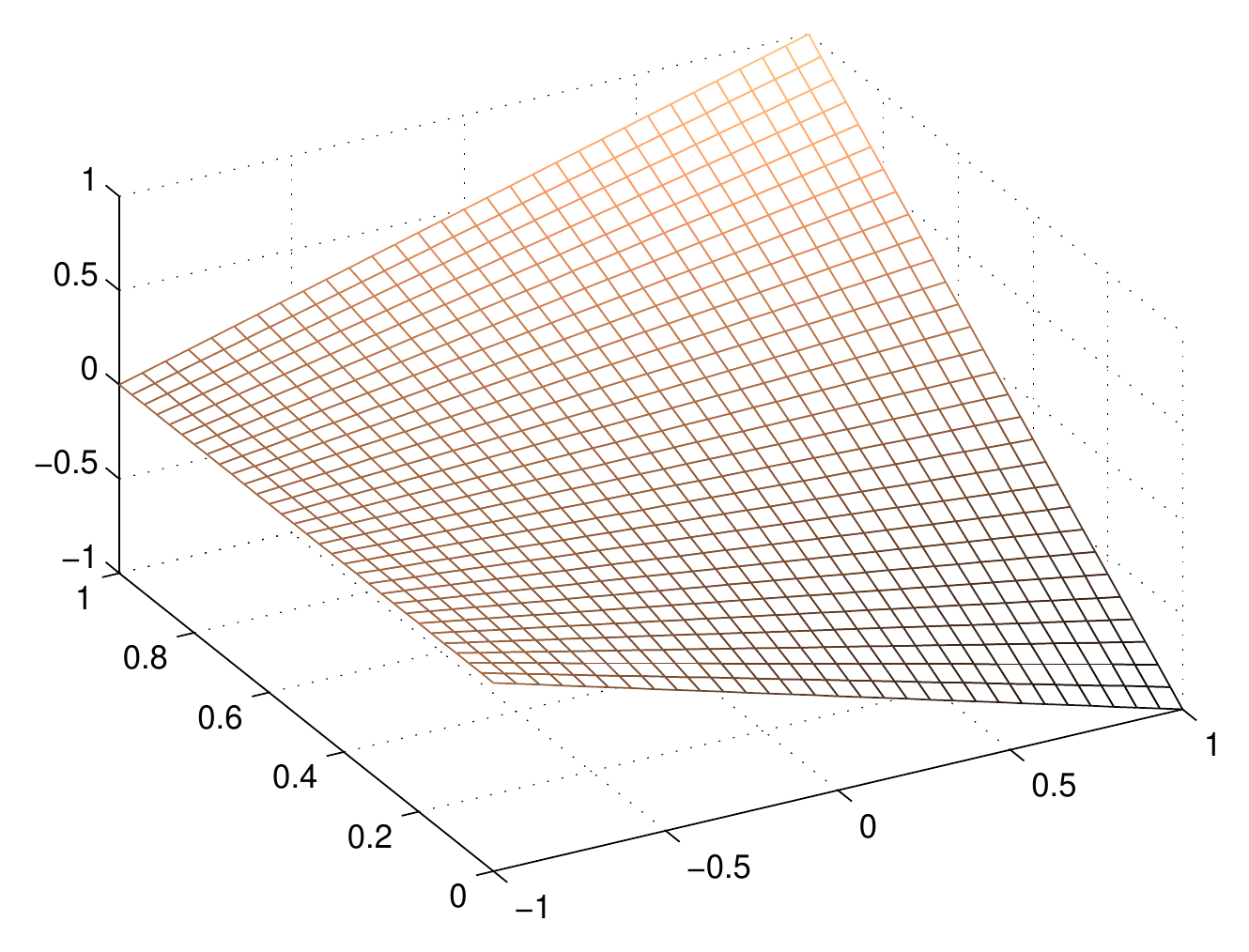}}
	}
\caption{Two additively unfaithful functions. Relevant variables are
  zeroed out under an additive approximation because every ``slice''
  of the function integrates to zero.}
\vskip-10pt
\end{figure*}

In order to exploit additive models in variable selection, it is important to understand when the
additive approximation accurately captures all of the relevant variables.
We call this property \textit{additive faithfulness}. We first formalize the intuitive notion that a multivariate function $f$ \emph{does not depends on} a coordinate $x_k$.

\begin{definition}
Let $C = [0,1]^p$ and let $f: C \rightarrow \R$. We say that $f$ \textit{does not depends on coordinate $k$} if for all $\mathbf{x}_{-k}$, $f(x_k, \mathbf{x}_{-k})$ is a constant as a function of $x_k$. If $f$ is differentiable, then $f$ does not depend on $k$ if $\partial_{x_k} f(x_k, \mathbf{x}_{-k})$ is 0 for all $\mathbf{x}_{-k}$.

In addition, suppose we have a distribution over $C$ and the additive approximation
\begin{equation}
\label{eqn:unconstrained_additive}
f_k^*, \mu^* \coloneqq \argmin_{f_1,\ldots, f_p, \mu} \Bigl\{ 
             \E \Bigl[\Bigl( f(X) - \sum_{k=1}^p f_k(X_k) -\mu \Bigr)^2 \Bigr]
         \,:\, \E f_k(X_k) = 0 \Bigr\}.
\end{equation}
We say that $f$ is \textit{additively faithful} under $F$ if
$f^*_k = 0$ implies that $f$ does not depend on coordinate $k$.
\end{definition}


Additive faithfulness is an attractive property because it implies
that, in the population setting, the additive approximation yields
consistent variable selection.

\subsection{Additive Faithfulness of Convex Functions}

We now show that under a general class of distributions which we
characterize below, convex multivariate functions are additively
faithful.

\begin{definition}
\label{defn:boundary-point}
A density $p(\mathbf{x})$ be a density supported on $[0,1]^p$ satisfies
the \emph{boundary flatness condition} if, for all $j$, and for all $\mathbf{x}_{-j}$:

\begin{equation}
\frac{\partial p(\mathbf{x}_{-j} \given x_j)}{\partial x_j}  =  
\frac{\partial^2 p(\mathbf{x}_{-j} \given x_j)}{\partial x_j^2} = 0
\quad \trm{at } x_k = 0, x_k = 1
\end{equation}

\end{definition}

The boundary flatness condition is a weak condition. For instance, it is
satisfied when the density is flat at the boundary of support---more
precisely, when the joint density satisfies the property that
$\frac{\partial p(x_j,\mathbf{x}_{-j})}{\partial x_j} =
\frac{\partial^2 p(x_j, \mathbf{x}_{-j})}{\partial x_j^2} = 0$ at boundary
points $x_j = 0, x_j=1$. The boundary flatness property is also
trivially satisfied when $p$ is a product density.

The following theorem is the main result of this section.

\begin{theorem}
\label{thm:convex_faithful}
Let $p$ be a positive density supported on $C=[0,1]^p$ that satisfies
the boundary flatness property. 
If $f$ is convex and twice differentiable, then $f$ is additively faithful under $p$.
\end{theorem}


We pause to give some intuition before we presenting the full proof.
Suppose that the underlying distribution has a product density.
Then we know from Lemma~\ref{lem:general_int_reduction} that the
additive approximation zeroes out $k$ when, fixing $x_k$, every
``slice'' of $f$ integrates to zero. We prove
Theorem~\ref{thm:convex_faithful} by showing that ``slices'' of convex
functions that integrate to zero cannot be ``glued together'' while
still maintaining convexity.

\begin{proof}

Fixing $k$ and using the result of
Lemma~\ref{lem:general_int_reduction}, 
we need only show that for all $x_k$, $ \E[ f(X) - \sum_{k'}
f_{k'}(X_{k'}) \given x_k] - \E f(X) = 0 $ 
implies that $f$ does not depend on coordinate $k$, i.e., 
$\partial_{x_k} f(\mathbf{x}) = 0$ for all $\mathbf{x}$.

Let us use the shorthand notation that $r(\mathbf{x}_{-k}) = \sum_{k'
  \neq k} f_{k'}(x_{k'})$ and assume without loss of generality that
$\mu^* = E[f(X)] = 0$. We then assume that for all $x_k$,
\begin{equation}
 \E[ f(X) - r(X_{-k})  \given x_k] \equiv 
 \int_{\mathbf{x}_{-k}}  p(\mathbf{x}_{-k} \given x_k ) 
 \big(f(\mathbf{x}) - r(\mathbf{x}_{-k}) \big) = 0.
\end{equation}
We let $p'(\mathbf{x}_{-k} \given x_k)$ denote 
$\frac{\partial p(\mathbf{\scriptstyle x}_{-k} \given x_k)}{\partial x_k}$ and 
$p''(\mathbf{x}_{-k} \given x_k)$ denote 
$\frac{\partial^2 p(\mathbf{\scriptstyle x}_{-k} \given x_k)}{\partial x_k^2}$ and
likewise for $f'(x_k, \mathbf{x}_{-k})$ and $f''(x_k,
\mathbf{x}_{-k})$. We then differentiate under the integral, valid
because all functions are bounded, and obtain
\begin{gather}
\int_{\mathbf{x}_{-k}} p'(\mathbf{x}_{-k} \given x_k) 
\big( f(\mathbf{x}) - r(\mathbf{x}_{-k}) \big) + 
p(\mathbf{x}_{-k} \given x_k) f'(x_k, \mathbf{x}_{-k}) d \mathbf{x}_{-k}  = 0 
\label{eqn:integral1a} \\
\int_{\mathbf{x}_{-k}} p''(\mathbf{x}_{-k} \given x_k) 
\big( f(\mathbf{x}) - r(\mathbf{x}_{-k}) \big) 
    + 2 p'(\mathbf{x}_{-k} \given x_k) f'(x_k, \mathbf{x}_{-k}) +
p(\mathbf{x}_{-k} \given x_k) f''(x_k, \mathbf{x}_{-k}) d\mathbf{x}_{-k}  = 0 .
\end{gather}

By the boundary flatness condition, we have that $p''(\mathbf{x}_{-k}
\given x_k)$ and $p'(\mathbf{x}_{-k} \given x_k)$ are zero at $x_k =
x_k^0 \equiv 0$. The integral equations then reduce to the following:
\begin{align}
& \int_{\mathbf{x}_{-k}} p(\mathbf{x}_{-k} \given x^0_k) f'(x^0_k, \mathbf{x}_{-k}) d \mathbf{x}_{-k}= 0 \label{eqn:integral1b} \\
& \int_{\mathbf{x}_{-k}} p(\mathbf{x}_{-k} \given x^0_k) f''(x^0_k, \mathbf{x}_{-k}) d\mathbf{x}_{-k} = 0.
\end{align}
Because $f$ is convex, $f(x_k, \mathbf{x}_{-k})$ must be a convex
function of 
$x_k$ for all $\mathbf{x}_{-k}$. Therefore, for all $\mathbf{x}_{-k}$,
$f''(x^0_k, \mathbf{x}_{-k}) \geq 0$. Since $p(\mathbf{x}_{-k} \given
x^0_k) > 0$ by the assumption that $p$ is a positive density, 
we have that $\forall \mathbf{x}_{-k}, f''(x^0_k, \mathbf{x}_{-k}) = 0$ necessarily.

The Hessian of $f$ at $(x^0_k, \mathbf{x}_{-k})$ then has a zero at
the $k$-th main diagonal entry. A positive semidefinite matrix with a
zero on the $k$-th main diagonal entry must have only zeros on the
$k$-th row and column; see proposition 7.1.10 of
\citet{HJ90}.  Thus, at all $\mathbf{x}_{-k}$, the
gradient of $f'(x^0_k, \mathbf{x}_{-k})$ with respect to
$\mathbf{x}_{-k}$ must be zero.
Therefore, $f'(x_k^0, \mathbf{x}_{-k})$ must be constant for all
$\mathbf{x}_{-k}$. By equation~\ref{eqn:integral1b}, we conclude 
that $f'(x_k^0, \mathbf{x}_{-k}) = 0$ for all $\mathbf{x}_{-k}$. We
can use the same reasoning for the case where $x_k = x_k^1$ and deduce
that $f'(x^1_k, \mathbf{x}_{-k}) = 0$ for all $\mathbf{x}_{-k}$. 

Because $f(x_k, \mathbf{x}_{-k})$ as a function of $x_k$ is convex, it must be that, for all $x_k \in (0,1)$ and for all $\mathbf{x}_{-k}$,
\begin{equation}
0 = f'(x_k^0, \mathbf{x}_{-k}) \leq f'(x_k, \mathbf{x}_{-k}) \leq 
    f'(x_k^1, \mathbf{x}_{-k}) = 0
\end{equation}
Therefore $f$ does not depend on $x_k$.


\end{proof}

Theorem~\ref{thm:convex_faithful} plays an important role in our
finite sample analysis, where we show that the additive
approximation is variable selection consistent (or ``sparsistent''), even when the true function is not
additive.

\begin{remark}
  We assume twice differentiability in
  Theorems~\ref{thm:convex_faithful} to simplify the proof.  We
  expect, however, that this this smoothness condition is not
  necessary---every convex function can be approximated arbitrarily
  well by a smooth convex function.
\end{remark}

\begin{remark} 
  We have not found natural conditions under which the opposite
  direction of additive faithfulness holds---conditions implying that if $f$ does not
  depend on coordinate $k$, then $f_k^*$ will be zero in the additive
  approximation.  Suppose, for example, that $f$ is only a
  function of $X_1, X_2$, and that $(X_1, X_2, X_3)$ follows a
  degenerate 3-dimensional distribution where $X_3 = f(X_1, X_2) -
  f^*(X_1) - f^*_2(X_2)$.  In this case $X_3$ exactly captures the
  additive approximation error.  The best additive
  approximation of $f$ would have a component $f^*_3(x_3) = x_3$ even
  though $f$ does not depend on $x_3$.
\end{remark}

\begin{remark}
In Theorem~\ref{thm:convex_faithful}, we do not assume a parametric form for the additive components. The additive approximations may not be faithful if we use parametric components. For example, suppose we approximate a convex function $f(X)$ by a linear form $X \beta$. The optimal linear function in the population setting is $\beta^* = \Sigma^{-1} \trm{Cov}(X, f(X))$. Suppose the $X$'s are independent and follow a symmetric distribution and suppose $f(\mathbf{x}) = x_1^2 - \E[X_1^2]$, then $\beta^*_1 = \E[ X_1 f(X)] = \E[ X_1^3 - X_1 \E[X_1^2]] = 0$.
\end{remark}

\begin{remark}
It is possible to get a similar result for distributions with
unbounded support, by using a limit condition $\lim_{|x_k| \rightarrow
  \infty} \frac{\partial p(\mathbf{\scriptstyle x}_{-k} \given x_k)}{\partial x_k}
= 0$.  Such a limit condition however is not obeyed by many common distributions such as the multivariate Gaussian distribution. The next example shows that certain convex functions are not additive faithful under certain multivariate Gaussian distributions.
\end{remark}

\begin{example}
\label{examp:gaussian_counterexample}
Consider a two dimensional quadratic function $f( \mathbf{x}) = \mathbf{x}^\tran H \mathbf{x} + c$ where $H = \begin{pmatrix} H_{11} & H_{12} \\ H_{12} & H_{22}\end{pmatrix}$ is positive definite and a Gaussian distribution $X \sim N(0, \Sigma)$ where $\Sigma = \begin{pmatrix}1 & \alpha \\ \alpha & 1 \end{pmatrix}$.
As we show in Section~\ref{sec:gaussian_example} of the Appendix, the additive approximation has the
following closed form.
\begin{align*}
f^*_1(x_1) &= \left( \frac{T_1 - T_2 \alpha^2}{1 - \alpha^4} \right) x_1^2 + c_1\\
f^*_2(x_2) &= \left( \frac{T_2 - T_1 \alpha^2}{1 - \alpha^4} \right) x_2^2 + c_2
\end{align*}
Where $T_1 = H_{11} + 2H_{12} \alpha + H_{22} \alpha^2$, $T_2 = H_{22} + 2H_{12} \alpha + H_{11} \alpha^2$, $c_1, c_2$ are constants such that $f^*_1$ and $f^*_2$ both have mean zero. Let $H = \begin{pmatrix} 1.6 & 2 \\ 2 & 5\end{pmatrix}$, then it is easy to check that if $\alpha = - \frac{1}{2}$, then $f^*_1 = 0$
and additive faithfulness is violated, if $\alpha > \frac{1}{2}$, then $f^*_1$ is a concave function. We take the setting where $\alpha=-0.5$, compute the optimal additive functions via numerical simulation, and show the results in Figure~\ref{fig:gaussian_example}--$f^*_1$ is zero as expected.
\end{example}

Although the Gaussian distribution does not satisfy the boundary flatness condition, it is possible to approximate the Gaussian distribution arbitrarily well with distributions that do satisfy the boundary flatness conditions.

\begin{example} 
\label{examp:boundaryflat_example}
Let $\Sigma$ be as in Example~\ref{examp:gaussian_counterexample} with $\alpha = -0.5$ so that $f^*_1 = 0$. Consider a mixture $\lambda U[-(b+\epsilon), b+\epsilon]^2 + (1-\lambda) N_b(0, \Sigma)$ where $N_b(0,\Sigma)$ is the density of a \emph{truncated} bivariate Gaussian bounded in $[-b, b]^2$ and $U[-(b+\epsilon), b+\epsilon]^2$ is the uniform distribution over a square. The uniform distribution is supported over a slightly larger square to satisfy the boundary flatness conditions.

When $b$ is large, $\epsilon$ is small, and $\lambda$ is small, the mixture closely approximates the Gaussian distribution but is still additively faithful for convex functions. Figure~\ref{fig:boundaryflat_example} shows the optimal additive components under the mixture distribution, computed by numerical integration with $b=5, \epsilon=0.3, \lambda=0.0001$. True to our theory, $f^*_1$, which is zero under the Gaussian distribution, is nonzero under the mixture approximation to the Gaussian distribution. We note that the magnitude $\E f^*_1(X_1)^2$, although non-zero, is very small, consistent with the fact that the mixture distribution closely approximates the Gaussian distribution.
\end{example}

\begin{figure*}[htp]
\vskip-10pt
	\centering
	\subfigure[Gaussian distribution]{
          \label{fig:gaussian_example}
		\centering
		{\includegraphics[width=0.4\textwidth]{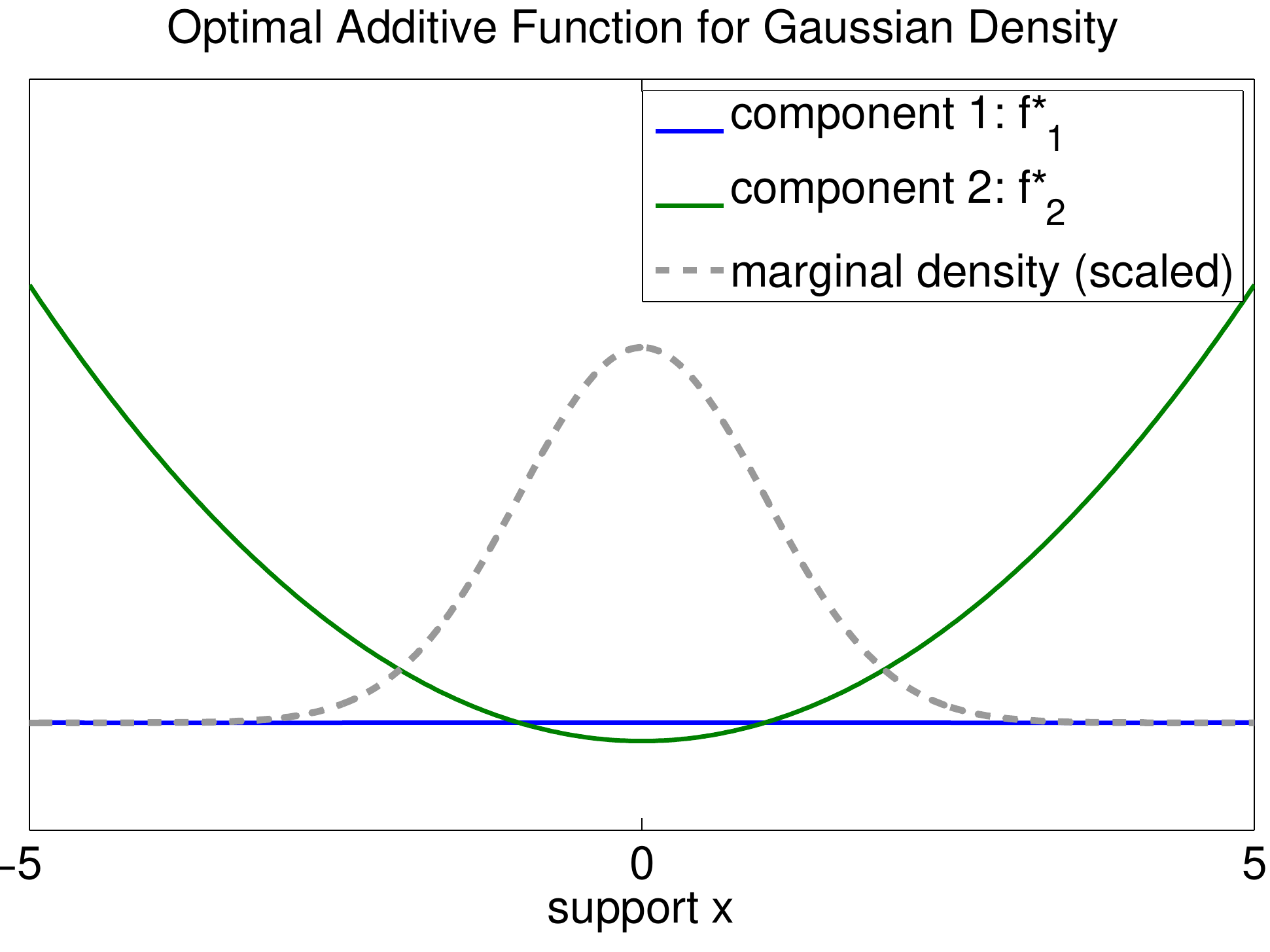}}
	}
	\subfigure[Mixture approximation]{
          \label{fig:boundaryflat_example}
		\centering
		{\includegraphics[width=0.4\textwidth]{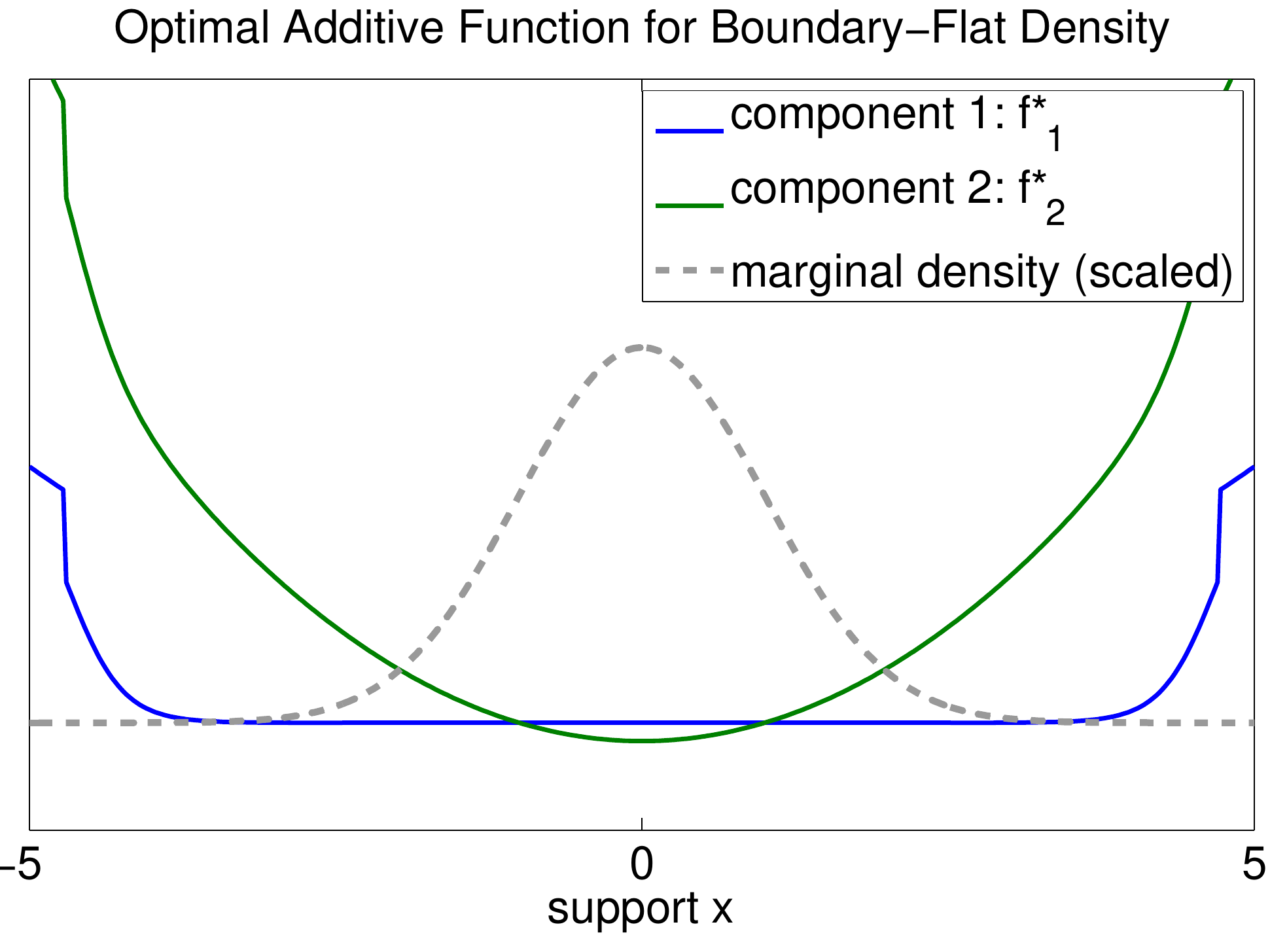}}
	}
\caption{Optimal additive projection of the quadratic function described in Example~\ref{examp:gaussian_counterexample} under both the Gaussian distribution described in Example~\ref{examp:gaussian_counterexample} and under the approximately Gaussian mixture distribution described in Example~\ref{examp:boundaryflat_example}. For the mixture approximation, we used $b=5, \epsilon=0.3, \lambda=0.0001$ where the parameters are defined in Example~\ref{examp:boundaryflat_example}. This example shows the effect and the importance of the boundary flatness conditions.}
\vskip-10pt
\end{figure*}

\subsection{Convex Additive Models}

Although convex functions are additively faithful---under appropriate conditions---it is difficult to
estimate the optimal additive functions $f^*_k$s as defined in
equation~\eqref{eqn:unconstrained_additive}.  The reason is that $f^*_k$ need not be
a convex function, as example~\ref{examp:gaussian_counterexample}
shows. It may be possible to estimate $f^*_k$ via smoothing, but we prefer an approach that is free of smoothing parameters. 
Since the true regression function $f$ is convex, we approximate the additive model with a \emph{convex} additive model. We abuse notation and, for the rest of the paper, use the notation $f^*_k$ to represent convex additive fits:
\begin{equation}
\label{eqn:convex_additive}
\{ f^*_k \}_{k=1}^p = \arg\min \Big \{ 
    \E\Bigl( f(X) - \sum_{k=1}^p f_k(X_k) \Bigr)^2 \,:\, f_k \in \mathcal{C}^1, \, \E f_k(X_k) = 0 \Big \}
\end{equation}
where $\mathcal{C}^1$ is the set of univariate convex functions. 
The convex functions $\{f^*_k\}$ are not additively faithful by
themselves, i.e., it could be that the true function $f$ depends on variable $k$ but $f^*_k = 0$. However, faithfulness can be restored by coupling the $f^*_k$'s with a set of univariate concave fits on the \emph{residual} $f - f^*$:
\begin{equation}
\label{eqn:concave_postprocess}
g^*_k = \arg\min \Big\{
   \E\Bigl( f(X) - \sum_{k' \neq k} f^*_{k'}(X_{k'}) - g_k \Bigr)^2
    \,:\, g_k \in \mh \mathcal{C}^1, \E g_k(X_k) = 0 
  \Big\}.
\end{equation}

\begin{theorem}
\label{thm:acdc_faithful}
Suppose $p(\mathbf{x})$ is a positive density on $C=[0,1]^p$ that
satisfies the boundary flatness condition. Suppose that $f$ is convex
and twice differentiable, and that $\partial_{x_k} p( \mathbf{x}_{-k} \given x_k )$, and
$\partial_{x_k}^2 p( \mathbf{x}_{-k} \given x_k)$ are all continuous
as functions of $x_k$.  Let $f^*_k$ and $g^*_k$ be as defined in equations
\eqref{eqn:convex_additive} and \eqref{eqn:concave_postprocess}, then the $f^*_k$'s and the $g^*_k$'s are unique. Furthermore, 
$f^*_k = 0$ and $g^*_k = 0$ implies that $\partial_{x_k} f(\mathbf{x})
= 0$, that is, $f$ does not depend on~$x_k$. 
\end{theorem}

Before we can prove the theorem, we need a lemma that generalizes Theorem~\ref{thm:convex_faithful}.

\begin{lemma}
\label{cor:faithfulness_extension}
Suppose $p(x)$ is a positive density on $C=[0,1]^p$ satisfying the boundary flatness condition.
Let $f(\mathbf{x})$ be a convex twice differentiable function on $C$. Let $\phi(\mathbf{x}_{-k})$ be any function that does not depend on $x_k$. Then, we have that the unconstrained univariate function
 \begin{align}
h^*_k = \arg\min_{f_k} \mathbb{E} \Bigl[\bigl( f(X) 
           - \phi(X_{-k}) - h_k(X_k) \bigr)^2\Bigr]
\end{align}
is given by $h^*_k(x_k) = \mathbb{E}\bigl[ f(X) - \phi(X_{-k}) \,|\, x_k\bigr]$,
and $ h^*_k = 0$ implies that $\partial_{x_k} f(\mathbf{x}) = 0$.
\end{lemma}

\begin{proof}
  In the proof of Theorem~\ref{thm:convex_faithful}, the only property
  of $r(\mathbf{x}_{-k})$ we used was the fact that $\partial_{x_k}
  r(\mathbf{x}_{-k}) = 0$. Therefore, the proof here is identical to
  that of Theorem~\ref{thm:convex_faithful} except that we replace $r(\mathbf{x}_{-k})$ with $\phi(\mathbf{x}_{-k})$.
\end{proof}

\begin{proof}[Proof of theorem~\ref{thm:acdc_faithful}]

Fix $k$. Let $f^*_k$ and $g^*_k$ be defined as in
equation~\ref{eqn:convex_additive} and
equation~\ref{eqn:concave_postprocess}. Let $\phi(\mathbf{x}_{-k}) \equiv \sum_{k' \neq k} f^*_{k'}(x_{k'})$.

Then we have that
\begin{align}
f^*_k &= \arg\min_{f_k} \Big\{
   \E\big( f(X) - \sum_{k' \neq k} f^*_{k'}(X_{k'}) - f_k \big)^2 
    \,:\, f_k \in  \mathcal{C}^1,\, \E f_k(X_k) = 0 
  \Big\} \\
g^*_k &= \arg\min_{g_k} \Big\{
   \E\big( f(X) - \sum_{k' \neq k} f^*_{k'}(X_{k'}) - g_k \big)^2 
    \,:\, g_k \in \mh \mathcal{C}^1,\, \E g_k(X_k) = 0 
  \Big\}
\end{align}

Let us suppose that $f^*_k = g^*_k = 0$. It must be then that
\begin{align*}
\argmin_{c \in \R} \E \left( f(X) - \phi(X_{-k}) - c (X_k^2 - m_k^2) \right)^2 = 0
\end{align*}
where $m_k^2 \equiv \E X_k^2$; this is because $c(x_k^2 - m_k^2)$ is either convex or concave in $x_k$ and it is centered, i.e. $\E[ X_k^2 - m_k^2] = 0$. Since the optimum has a closed form $c^* = \frac{\E\big[(f(X) -\phi(X_{-k}))(X_k^2 - m^2_k)\big]}{\E X_k^2}$, we deduce that 
\begin{align*}
\E \big[ (f(X) - \phi(X_{-k})) &(X_k^2 - m_k^2) \big] \\
   &= \E[ (f(X) - \phi(X_{-k})) X^2_k] = 
 \E[ \E[ f(X) - \phi(X_{-k}) \given X_k] X^2_k] = 0
\end{align*}

We denote $h^*_k(x_k) = \E[ f(X) - \phi(X_{-k}) \given x_k]$.

Under the derivative continuity conditions in the theorem, we apply Lemma~\ref{lem:acdc_derivative_bound} in the appendix and know that $h^*_k(x_k)$ is twice-differentiable and has a second derivative bounded away from $-\infty$. Therefore, for a large enough positive scalar $\alpha$, $h^*_k(x_k) + \alpha(x_k^2 - m_k^2)$ has a non-negative second derivative and is thus convex.

Because we assumed $f^* = g^* = 0$, it must be that
\[
\argmin_{c \in \R} 
\E\Big( f(X) - \phi(X_{-k}) - c\big( h^*_k(X_k) + \alpha(X_k^2 - m_k^2)\big) \Big)^2 = 0
\]
This is because $c\big( h^*_k(x_k) + \alpha(x_k^2 - m_k^2) \big)$ is convex for $c \geq 0$ and concave for $c \leq 0$ and it is a centered function.

Again, $c^* = \frac{\mathbb{E}[(f(X)-\phi(X_{-k}))\big( 
           h^*_k(X_k) + \alpha (X_k^2 - m_k^2) \big)]}{\mathbb{E}
       \big( h^*_k(X_k) + \alpha (X_k^2 - m_k^2) \big)^2} = 0$, so
\begin{align*}
\mathbb{E}[ (f(X)-\phi(X_{-k})) \big( h^*_k(X_k) + \alpha (X_k^2-m_k^2) \big) ] &= 
\mathbb{E}[ (f(X) - \phi(X_{-k})) h^*_k(X_k) ] \\
& = \mathbb{E}\Big[ \mathbb{E}[ f(X) - \phi(X_{-k}) \given X_k]  h^*_k(X_k) \Big] \\
& = \mathbb{E} h^*_k(X_k)^2 = 0
\end{align*}
where the first equality follows because $\E[ (f(X) - \phi(X_{-k})) (X_k^2 - m_k^2)] = 0$. Therefore, we get that $h^*_k = 0$. Now we use
Lemma~\ref{cor:faithfulness_extension} with $\phi(\mathbf{x}_{-k}) =
f(\mathbf{x}) - \sum_{k' \neq k} f^*_{k'} (x_{k'})$ and conclude that
$f^*_k = 0$ and $g^*_k = 0$ together imply that $f$ does not depend on $x_k$.\\

Now we turn to uniqueness. Suppose for sake of contradiction that $f^*$ and $f'$ are optimal solutions to (\ref{eqn:convex_additive}) and $\E (f'-f^*)^2 > 0$.  $f^* + \lambda ( f' - f^*)$ for any $\lambda \in [0,1]$ must then also be an optimal solution by convexity of the objective and constraint. However, the second derivative of the objective $\E( f - f^* - \lambda(f' - f^*))^2$ with respect to $\lambda$ is $2 \E( f' - f^*)^2 > 0$. The objective is thus strongly convex and $f^*, f'$ cannot both be optimal. The uniqueness of $g^*$ is proved similarly.
\end{proof}

\def\C{\mathcal{C}}

\subsection{Estimation Procedure}
\label{sec:acdc}

Theorem~\ref{thm:acdc_faithful} naturally suggests 
a two-stage screening procedure for variable selection in the population setting. In the first stage, we fit a convex additive model. 
\begin{equation}
\label{eqn:scam2_pop}
f^*_1, ..., f^*_p = \argmin_{f_1,...,f_p \in \C^1_0} 
   \E \Big( f(X)  - \sum_{k=1}^p f_k(X_k) \Big)^2 
\end{equation}
where we denote $\C^1_0$ ($\mh{}\C^1_0$) as the set of one-dimensional convex (resp. concave) functions with population mean zero. In the second stage, for every variable marked as irrelevant in the first stage, we fit a univariate \emph{concave} function separately on the residual for that variable.
 for each $k$ such that $ f^*_k = 0$:
\begin{equation}
\label{eqn:dc2}
g^*_k = \argmin_{g_k \in \mh{}\C^1_0} 
   \E \Big( f(X) - \sum_{k'} f^*_{k'}(X_{k'}) 
    - g_k(X_{k})\Big)^2 
\end{equation}
We screen out $S^C$, any variable $k$ that is zero after the second stage, and output $S$.
\begin{equation}
\label{eqn:acdc_vars_pop}
S^c = \bigl\{k : f^*_k =
0 \; \mathrm{and}\; g^*_k =0\bigr\}.
\end{equation}

We refer to this procedure as AC/DC (additive
convex/decoupled concave). Theorem~\ref{thm:acdc_faithful} guarantees that the true set of relevant variables $S_0$ must be a subset of $S$.


It is straightforward to construct a finite sample variable screening procedure, which we describe in Figure~\ref{fig:backfitting:algo}.
We use an $\ell_\infty/\ell_1$ penalty in equation~\eqref{eqn:scam2}
and an $\ell_\infty$ penalty in equation~\eqref{eqn:dc2} to encourage
sparsity.  Other penalties can also produce
sparse estimates, such as a penalty on the derivative of each of the
component functions.  The $\|\cdot\|_\infty$ norm is convenient for both
theoretical analysis and implementation.

The optimization in \eqref{eqn:scam2} appears to be infinite
dimensional, but it is equivalent to a finite dimensional quadratic
program.  In the following section, we give the details
of this optimization, and show how it can be reformulated
to be more computationally efficient.


\begin{figure}[t]
{\sc AC/DC Algorithm for Variable Selection in Convex Regression\hfill}
\vskip5pt
\begin{center}
\hrule
\vskip7pt
\normalsize
\begin{enumerate}
\item[] \textit{Input}:  $(\mathbf{x}_1, y_1), ..., (\mathbf{x}_n, y_n)$, regularization parameter $\lambda$.
\vskip5pt
\item[] \textit{AC Stage}:  Estimate a sparse additive convex model:
\begin{equation}
\label{eqn:scam2}
\hat{f}_1, ..., \hat{f}_p, \hat\mu = \argmin_{f_1,...,f_p \in \C^1_0} 
   \frac{1}{n} \sum_{i=1}^n \Big(y_i - \mu-\sum_{k=1}^p f_k(x_{ik}) \Big)^2 
       + \lambda \sum_{k=1}^p \| f_k \|_\infty
\end{equation}
\vskip5pt
\item[] \textit{DC Stage}:  Estimate concave functions
 for each $k$ such that $\| \hat{f}_k \|_\infty = 0$:
\begin{equation}
\label{eqn:dc22}
\hat{g}_k = \argmin_{g_k \in \mh{}\C^1_0} 
   \frac{1}{n} \sum_{i=1}^n \Big( y_i - \hat \mu - \sum_{k'} \hat{f}_{k'}(x_{ik'}) 
    - g_k(x_{ik})\Big)^2 
      + \lambda \| g_k \|_\infty
\end{equation}
\item[] \textit{Output}: Component functions $\{\hat f_k\}$ and 
relevant variables $\hat S$ where
\begin{equation}
\hat S^c = \bigl\{k : \| \hat{f}_k \| =
0 \; \mathrm{and}\; \|\hat{g}_k \|=0\bigr\}.
\end{equation}
\end{enumerate}
\vskip3pt
\hrule
\end{center}
\vskip0pt
\caption{The AC/DC algorithm for variable selection in convex
  regression.  The AC stage fits a sparse additive convex regression
  model, using a quadratic program that imposes an group sparsity
  penalty for each component function.  The DC stage fits
  decoupled concave functions on the residuals, for each 
  component that is zeroed out in the AC stage.}
\label{fig:backfitting:algo}
\end{figure}

 

\def\uds#1{#1}
\def\perm#1{\pi_k(#1)}

\section{Optimization}
\label{sec:optimization}

We now describe in detail the optimization algorithm for the additive
convex regression stage.  The second decoupled concave regression stage
follows a very similar procedure.

Let $\bds{x}_{i}\in\mathbb{R}^{p}$ be the covariate, let $y_{i}$ be
the response and let $\epsilon_{i}$ be the mean zero noise. The
regression function $f(\cdot)$ we estimate is the sum of
univariate functions $f_{k}(\cdot)$ in each variable dimension and a scalar
offset $\mu$.  We impose additional constraints that each
function $f_{k}(\cdot)$ is convex, which can be
represented by its supporting hyperplanes, i.e.,
\begin{equation}\label{hyper}
      f_{i'k} \geq f_{ik} + \beta_{ik}(x_{i'k}-x_{ik}) \quad
      \textrm{for all $i,i' = 1,\ldots, n$,}
\end{equation}
where $f_{ik}\coloneqq f_{k}(x_{ik})$ is the function value and $\beta_{ik}$ is a
subgradient at point $x_{ik}$. This ostensibly requires $O(n^2 p)$ constraints to
impose the supporting hyperplane constraints.
In fact, only $O(np)$
constraints suffice, since univariate convex functions are
characterized by the condition that the subgradient, which is a scalar, must
increase monotonically. This observation leads to the  optimization
\begin{equation}
\begin{split}
       \min_{\{f_k,\beta_k\},\mu} & \;\; \frac{1}{2n}\sum_{i=1}^{n}
                     \Bigl( y_{i}-\mu - \sum_{k=1}^{p}f_{ik}\Bigr)^{2} 
                         + \lambda\sum_{k=1}^{p}\|f_k\|_{\infty} \\
       \textrm{subject to} &\;\; \textrm{for all $k=1,\ldots, p$:}\\
       & \;\; f_{\perm{i+1} k} = f_{\perm{i} k} +
       \beta_{\perm{i} k}(x_{\perm{i+1} k}-x_{\perm{i} k}),\;\textrm{for $i=1,\ldots, n-1$}\\
       & \;\; \sum_{i=1}^{n}f_{ik}=0,\\
       & \;\; \beta_{\perm{i+1} k} \geq \beta_{\perm{i} k}\;\textrm{for $i=1,\ldots, n-1$}.
\end{split}
\label{np}
\end{equation}
Here $f_k$ denotes the vector $f_k = (f_{1k}, f_{2_k}\ldots, f_{nk})^T\in\reals^n$
and $\{\perm{1},\perm{2},\ldots,\perm{n}\}$ are the indices in the sorted ordering
of the values of coordinate $k$:
\begin{equation}
x_{\perm{1} k} \leq{} x_{\perm{2} k} \leq \cdots \leq{} x_{\perm{n} k}.
\end{equation}

We can solve for $\mu$ explicitly as  
$\mu = \frac{1}{n} \sum_{i=1}^n y_i = \bar{y}$.  This follows from the
KKT conditions
and the constraints $\sum_i f_{ki} = 0$.


The sparse convex additive model optimization in \eqref{np} is a quadratic program with
$O(np)$ variables and $O(np)$ constraints. 
Directly applying a QP solver for $f$ and $\beta$
is computationally expensive for relatively large
$n$ and $p$. However, notice that variables in different feature
dimensions are only coupled in the squared error term
$(y_{i}-\mu - \sum_{k=1}^{p}f_{ik})^{2}$. Hence, we can apply the block
coordinate descent method, where in each step we solve the following
QP subproblem for $\{f_k, \beta_k\}$ with the
other variables fixed. In matrix notation, the optimization is
\begin{align}
\begin{split}
\min_{ f_k, \beta_k, \gamma_k} \;\;& \frac{1}{2n} \| \uds{r}_k - \uds{f}_k \|_2^2 
     + \lambda \gamma_k \label{opt:1d_compact} \\
 \textrm{such that } & P_k \uds{f}_k = \diag(P_k \bds{x}_k)  \uds{\beta}_k \\
   & D_k \uds{\beta}_k \leq 0 \\
   & -\gamma_k \mathbf{1}_n \leq \uds{f}_k \leq \gamma_k \mathbf{1}_n   \\
   & \mathbf{1}_n^\tran \uds{f}_k = 0 
\end{split}
\end{align}
where $\uds{\beta}_k \in \R^{n-1}$ is the vector $\uds{\beta}_k =
(\beta_{1k}, \ldots, \beta_{(n-1)k})^T$, and
$\uds{r}_{k} \in \R^n$ is the residual vector $\uds{r}_{k} = (y_i -
\hat\mu - \sum_{k' \neq k} f_{ik'})^T$.
In addition, 
$P_k \in \R^{(n-1) \times n}$ is a permutation matrix where the $i$-th
row  is all zeros except for the value $-1$ in position $\perm{i}$ and
the value $1$ in
position $\perm{i+1}$, and $D_k \in \R^{(n-2) \times (n-1)}$ is another
permutation matrix  where the $i$-th row is all zeros except for a
value $1$  in position $\perm{i}$ and a value $-1$ in position $\perm{i+1}$.  We denote by
$\diag( v )$ the diagonal matrix with diagonal entries $v$.
The extra variable $\gamma_{k}$ is introduced to impose the
regularization penalty involving the $\ell_{\infty}$ norm.  


This QP
subproblem involves $O(n)$ variables, $O(n)$ constraints and a sparse
structure, which can be solved efficiently using optimization
packages. In our experiments we use {\sc mosek} (\href{http://www.mosek.com/}{www.mosek.com}).  We cycle through
all covariates $k$ from $1$ to $p$ multiple times until convergence.
Empirically, we observe that the algorithm converges in only a few
cycles. We also implemented an ADMM solver for \eqref{np}
\citep{Boyd:admm}, but found
that it is not as efficient as this blockwise QP solver.

After optimization, the function estimate for an input vector $\bds{x}$ is, according to \eqref{hyper},
\begin{equation}
\begin{split}
      \hat f(\bds{x}) & = \sum_{k=1}^{p} \hat f_k(x_{k})+ \hat \mu 
= \sum_{k=1}^{p}\max_{i} \Bigl\{\hat f_{ik}+ \hat \beta_{ik}(x_{k}-x_{ik})\Bigr\} +
      \hat \mu.
\end{split}
\end{equation} 

The univariate concave function estimation required in the DC stage is a straightforward
modification of optimization~\eqref{opt:1d_compact}. It is only
necessary to modify the linear inequality constraints so that the subgradients are
non-increasing: $\beta_{\perm{i+1}k} \leq \beta_{\perm{i}k}$.

\subsection{Alternative Formulation}
Optimization \eqref{np} can be reformulated in terms of the second
derivatives. The alternative formulation replaces the order constraints
$\beta_{\perm{i+1}k} \geq \beta_{\perm{i}k}$ with positivity constraints, which
simplifies the analysis.  



Define $d_{\perm{i}k}$ as the second
derivative: $d_{\perm{1}k} = \beta_{\perm{1}k}$, and $d_{\perm{i}k} = \beta_{\perm{i}k} -
\beta_{\perm{i-1}k}$ for $i > 1$. The convexity constraint is equivalent to the
constraint that $d_{\perm{i}k} \geq 0$ for all $i > 1$.

It is easy to verify that $\beta_{\perm{i}k} = \sum_{j \leq i} d_{\perm{j}k}$ and 
\begin{align*}
f_k(x_{\perm{i}k}) = & f_k(x_{\perm{i-1}k}) + \beta_{\perm{i-1}k}(x_{\perm{i}k} - x_{\perm{i-1}k}) \\
 =& f_k(x_{\perm{1}k}) + \sum_{j < i} \beta_{\perm{j}k} (x_{\perm{j}k} - x_{\perm{j-1}k}) \\
 =& f_k(x_{\perm{1}k}) + \sum_{j < i} \sum_{j' \leq j} d_{\perm{j'}k} (x_{\perm{j}k} - x_{\perm{j-1}k})\\
 =& f_k(x_{\perm{1}k}) + \sum_{j' < i} d_{\perm{j'}k} \sum_{i > j \geq j'} (x_{\perm{j}k} - x_{\perm{j-1}k}) \\
 =& f_k(x_{\perm{1}k}) + \sum_{j' < i} d_{\perm{j'}k} (x_{\perm{i}k} - x_{\perm{j'}k}).
\end{align*}
We can write this more compactly in matrix notation as
\begin{align}
\nonumber
\left[ \begin{array}{c}
f_k(x_{\perm{1}k}) \\
f_k(x_{\perm{2}k}) \\
\vdots \\
f_k(x_{\perm{n}k})
\end{array} \right] &=
\left[ \begin{array}{ccc}
    (x_{k1} - x_{\perm{1}k})_+ & \cdots & (x_{k1} - x_{\perm{n-1}k})_+ \\
    \cdots & & \\
    (x_{kn} - x_{\perm{1}k})_+ & \cdots & (x_{kn} - x_{\perm{n-1}k})_+ 
\end{array} \right]
\left[ \begin{array}{c}
    d_{\perm{1}k} \\
    \cdots \\
    d_{\perm{n-1}k}
\end{array} \right] + \mu_k \\[10pt]
& \equiv \Delta_k d_k + \mu_k
\end{align}
where $\Delta_k$ is a $n\times n-1$ matrix such that $\Delta_k(i,j) =
(x_{\perm{i}k} - x_{\perm{j}k})_+$, $d_k = (d_{\perm{1}k} ,\ldots,
d_{\perm{n-1}k})$, and $\mu_k = f_k(x_{\perm{1}k}) \mathbf{1}_n$.
Because $f_k$ has to be centered, $\mu_k = - \frac{1}{n}
\mathbf{1}_n^\tran \Delta_k d_k$, and therefore
\begin{equation}
\Delta_k d_k + \mu_k \mathbf{1}_n = 
   \Delta_k d_k - \frac{1}{n} \mathbf{1}_n \mathbf{1}_n^\tran \Delta_k d_k = 
   \bar{\Delta}_k d_k 
\label{eq:deltabar_defn}
\end{equation}
where $\bar{\Delta}_k \equiv \Delta_k - \frac{1}{n} \mathbf{1}_n \mathbf{1}_n^\tran \Delta_k$ is $\Delta_k$ with the mean of each column subtracted.

The above derivations prove the following proposition, which states that \eqref{np} has an alternative formulation.

\begin{proposition} 
\label{prop:alt_opt_form}
Let $\{\hat{f}_k, \hat{\beta}_k\}_{k=1,...,p}$ be an optimal solution to \eqref{np} and suppose $\bar{Y} = 0$. Define vectors $\hat{d}_k \in \R^{n-1}$ such that $\hat{d}_{\pi_k(1)k} = \hat{\beta}_{\pi_k(1)k}$ and $\hat{d}_{\pi_k(i)k} = \hat{\beta}_{\pi_k(i)k} - \hat{\beta}_{\pi_k(i-1)k}$ for $i > 1$. Then $\hat{f}_k = \bar{\Delta}_k \hat{d}_k$ and $\hat{d}_k$ is an optimal solution to the following optimization:
\begin{align}
\min_{\{d_k \in \R^{n-1}\}_{k=1,...p}}\;\; & \frac{1}{2n} 
       \Bigl\| Y - \sum_{k=1}^p 
              \bar{\Delta}_k d_k \Bigr\|_2^2 
               + \lambda_n \sum_{k=1}^p \|\bar{\Delta}_k d_k \|_\infty   
     \label{opt:alternate_opt} \\
\trm{such that}\;\;  & d_{\perm{2}k}, \ldots , d_{\perm{n-1}k} \geq 0  	
               \qquad \trm{(convexity).} \nonumber 
\end{align}
Likewise, suppose $\{\hat{d}_k\}_{k=1,...p}$ is a solution to \eqref{opt:alternate_opt}, define $\hat{\beta}_{\pi_k(i)k} = \sum_{j\leq i} \hat{d}_{\pi_k(j)k}$ and $\hat{f}_k = \bar{\Delta}_k \hat{d}_k$. Then $\{\hat{f}_k, \hat{\beta}_k\}_{k=1,...,p}$ is an optimal solution to \eqref{np}. $\bar{\Delta}$ is the $n$ by $n-1$ matrix defined by \eqref{eq:deltabar_defn}.

\end{proposition}

The decoupled concave postprocessing stage optimization is again
similar. Specifically, suppose $\hat{d}_k$ is the output of
optimization~\eqref{opt:alternate_opt}, and define the residual vector
\begin{equation}
\hat{r} = Y -
\sum_{k=1}^p \bar{\Delta}_k \hat{d}_k.
\label{eq:residual}
\end{equation}  
Then  for all $k$ such that $\hat{d}_k = 0$, the DC stage optimization is
formulated as
\begin{align}
  \min_{c_k}\;\; & 
      \frac{1}{2n} \Bigl \| \hat{r} - \Delta_k c_k \Bigr \|_2^2
      + \lambda_n \| \Delta_k c_k \|_\infty 
      \label{opt:alternate_opt_concave}\\
 \trm{such that }\;\; & c_{\perm{2}k}, \ldots, c_{\perm{n-1}k} \leq 0 \qquad \trm{(concavity).} \nonumber
\end{align}

We can use either the off-centered $\Delta_k$ matrix or the centered
$\bar{\Delta}_k$ matrix because the concave estimations are decoupled
and hence are not subject to non-identifiability under additive constants.


\section{Analysis of Variable Screening Consistency}
\label{sec:finitesample}

Our goal is to show that variable screening consistency. That is, as $n,p \rightarrow \infty$, $\mathbb{P}( \hat{S} = S)$ approaches 0 where $\hat{S}$ is the set of variables outputted by AC/DC in the finite sample setting (Algorithm~\ref{fig:backfitting:algo}) and $S$ is the set of variables outputted in the population setting~\eqref{eqn:acdc_vars_pop}.

We divide our analysis into two parts. We first establish a sufficient
deterministic condition for consistency of the sparsity pattern
screening procedure.  We then consider the stochastic setting and argue that the 
deterministic conditions hold with high probability. Note that in all of our results 
and analysis, we let $c,
C$ represent absolute constants; the actual values of $c,C$ may change from line to line. We derived two equivalent optimizations for AC/DC: \eqref{np} outputs $\hat{f}_k, \hat{g}_k$ and \eqref{opt:alternate_opt} outputs the second derivatives $\hat{d}_k$. Their equivalence is established in Proposition~\ref{prop:alt_opt_form} and we use both $\hat{d}_k$ and $\hat{f}_k$ in our analysis.

In our analysis, we assume that an upper bound $B$ to $\| \hat{f}_k \|_\infty$ 
is imposed in the optimization procedure, where $B \geq \| f^*_k \|_\infty$. This $B$-boundedness constraint is added so that we may use the convex function bracketing results from~\cite{kim2014global} to establish uniform convergence between the empirical risk and the population risk. We emphasize that this constraint is not needed in practice and we do not use it for any of our simulations.

\subsection{Deterministic Setting}

We analyze Optimization~\ref{opt:alternate_opt} and construct an additive convex solution $\{\hat{d}_k\}_{k=1,\ldots,p}$
that is zero for $k \in S^c$, where $S$ is the set of relevant
variables, and show that it satisfies the KKT
conditions for optimality of optimization~\eqref{opt:alternate_opt}. We
define $\hat{d}_k$ for $k \in S$ to be a solution to the restricted
regression (defined below). We also show that $\hat{c}_k =
0$ satisfies the optimality condition of
optimization~\eqref{opt:alternate_opt_concave} for all $k \in S^c$.

\begin{definition}
\label{def:restricted_regression}
We define the \emph{restricted regression} problem 
\[
\min_{d_k} \frac{1}{n} \Big\| Y - \sum_{k \in S} \bar{\Delta}_k d_k \Big\|_2^2 + 
   \lambda_n \sum_{k \in S} \| \bar{\Delta}_k d_k \|_\infty \quad \trm{such that} \, d_{k,1}, \ldots, d_{k,n-1} \geq 0
\]
where we restrict the indices $k$ in
optimization \eqref{opt:alternate_opt} to lie in some set $S$ which contains the true
relevant variables.
\end{definition}

\begin{theorem}[Deterministic setting]
\label{thm:deterministic}
Let $\{\hat{d}_k \}_{k \in S}$ be a minimizer of the restricted regression as defined above.
Let $\hat{r} \coloneqq Y - \sum_{k \in S} \bar{\Delta}_k \hat{d}_k$ be the restricted regression residual. 
Suppose for all $k\in S^c$, for all $i=1,\ldots,n$, $\lambda_n > | \frac{1}{2n}
\hat{r}^\tran \mathbf{1}_{(i:n)}|$ where $\mathbf{1}_{(i:n)}$ is 1 on
the coordinates of the 
$i$-th largest to the $n$-th largest entries of $X_k$ and 0
elsewhere.  Then the following two statements hold.
\begin{enumerate}
\item Let $\hat{d}_k = 0$ for $k \in S^c$.  Then
  \{$\hat{d}_k\}_{k=1,\ldots,p}$ is an optimal solution to
  optimization~\eqref{opt:alternate_opt}. Furthermore, any solution to
  the optimization program \eqref{opt:alternate_opt} must be zero on
  $S^c$.
\item For all $k \in S^c$, the solution $\hat{c}_k$ to optimization~\eqref{opt:alternate_opt_concave} must be zero.
\end{enumerate}

\end{theorem}

This result holds regardless of whether or not we impose the boundedness conditions in optimization~\eqref{opt:alternate_opt} and~\eqref{opt:alternate_opt_concave}.
The full proof of Theorem~\ref{thm:deterministic} is in Section~\ref{sec:deterministic_proof} of the Appendix. We allow $S$ in Theorem~\ref{thm:deterministic} to be any set containing the relevant variables; in Lasso analysis, $S$ is taken to be the set of relevant variables; we will take $S$ to be the set of variables chosen by the additive convex and decoupled concave procedure in the population setting, which is guaranteed to contain the relevant variables because of additive faithfulness.

Theorem~\ref{thm:deterministic} allows us to separately analyze the false negative
rates and false positive rates. To control false positives,
we analyze the condition on $\lambda_n$ for $k \in S^c$. To control
false negatives, we need only analyze the restricted regression with only $|S|$ variables.

The proof of Theorem~\ref{thm:deterministic} analyses the KKT
conditions of optimization~\eqref{opt:alternate_opt}.  This parallels
the now standard \emph{primal-dual witness}
technique~\citep{wainwright2009sharp}. However, we cannot derive analogous
\emph{mutual incoherence} conditions because the estimation is
nonparametric---even the low dimensional restricted regression has
$s(n-1)$ variables. The details of the proof are given in
Section~\ref{sec:deterministic_proof} of the Appendix.

\subsection{Probabilistic Setting}

In the probabilistic setting we treat the covariates as random.  We
adopt the following standard setup:

\begin{enumerate}
\item The data $X^{(1)},\ldots, X^{(n)} \sim F$ are iid from
a distribution $F$ with a density $p$ that is supported and strictly positive on $\mathcal{X}=[-b,b]^p$. 
\item The response is $Y = f_0(X) + W$ where $W$ is
  independent, zero-mean noise; thus $Y^{(i)} = f_0(X^{(i)}) + W^{(i)}$.
\item The regression function $f_0$ satisfies
$f_0(X) = f_0(X_{S_0})$, where $S_0 = \{1,\ldots,s_0\}$ is the set of
relevant variables.
\end{enumerate}

Let $\mathcal{C}^1$ denote the set of univariate convex functions
supported on $[-b,b]$, 
and let  $\mathcal{C}_1^{p}$ denote the set of convex additive functions
$\mathcal{C}_1^p \equiv \{ f \,:\, f = \sum_{k=1}^p f_k, \,
   f_k \in \mathcal{C}^1 \} $.  
Let $f^*(\mathbf{x}) = \sum_{k=1}^p f^*_k(x_k)$ be the population risk
minimizer in $\mathcal{C}_1^p$, 
\begin{equation}
f^* = \arg\min_{f \in \mathcal{C}_1^p} \E\big(f_0(X) - f^*(X)
\big)^2.
\end{equation}
$f^*$ is the unique minimizer by Theorem~\ref{thm:acdc_faithful}. Similarly, we define $\mh \mathcal{C}^1$ as the set of univariate concave functions supported on $[-b, b]$ and define
\begin{equation}
g^*_k = \arg\min_{g_k \in \mh \mathcal{C}^1} \E \big( f_0(X) - f^*(X)
- g_k(X_k) \big)^2.
\end{equation}
The $\hat{g}_k$'s are unique minimizers as well. We let $S = \{ k = 1,\ldots,p \,:\, f^*_k \neq 0 \trm{ or } g^*_k \neq 0\}$ and let $s = |S|$. By additive faithfulness (Theorem~\ref{thm:acdc_faithful}), it must be that $S_0 \subset S$ and thus $s \geq s_0$. In some cases, such as when $X_{S_0}, X_{S^c_0}$ are independent, we have $S = S_0$.
Each of our theorems will use a subset of the following assumptions:
\begin{packed_enum}
\item[A1:] $X_S, X_{S^c}$ are independent. 
\item[A2:] $f_0$ is convex and twice-differentiable. 
\item[A3:] $\|f_0\|_\infty \leq sB$ and $\| f^*_k \|_\infty \leq B$ for all $k$.
\item[A4:] $W$ is mean-zero sub-Gaussian, independent of $X$, with scale $\sigma$; i.e., for all $t \in \R$, $\E e^{t \epsilon} \leq e^{\sigma^2 t^2 / 2}$.
\end{packed_enum}
By assumption A1, $f^*_k$ is must be zero for $k\notin S$.
We define $\alpha_{+}, \alpha_{-}$ as a measure of the signal strength of the weakest variable:
\begin{align}
\alpha_{+} &= \inf_{f \in \mathcal{C}_1^p \,:\, \textrm{supp}(f) \subsetneq \textrm{supp}(f^*)} 
       \Big\{ \mathbb{E} \big( f_0(X) - f(X) \big)^2 - 
        \mathbb{E} \big( f_0(X) - f^*(X) \big)^2  \Big\} \label{eqn:signal_level_defn} \\
\alpha_{-} &=   \min_{k \in S \,:\, g^*_k \neq 0}
      \Big\{ \mathbb{E} \big( f_0(X) - f^*(X) \big)^2 - 
    \mathbb{E} \big( f_0(X) - f^*(X) - g^*_k(X_k) \big)^2 \Big\} \nonumber
\end{align}
$\alpha_+$ is a lower bound on the excess risk incurred by any additive convex function whose support is strictly smaller than $f^*$; $\alpha_+ > 0$ since $f^*$ is the unique risk minimizer. Likewise, $\alpha_-$ lower bounds the excess risk of any decoupled concave fit of the residual $f_0 - f^*$ that is strictly more sparse than the optimal decoupled concave fit $\{\hat{g}_k^*\}$; $\alpha_- > 0$ by the uniqueness of $\{g^*_k\}$ as well. These quantities play the role of the absolute value of the smallest nonzero coefficient in the true linear model in lasso theory.  Intuitively, if
$\alpha_{+}$ is small, then it is easier to make a false omission in the
additive convex stage of the procedure. If $\alpha_{-}$ is small, then
it is easier to make a false omission in the decoupled concave stage
of the procedure.

\begin{remark}
  We make strong assumptions on the covariates in A1 in order to make
  weak assumptions on the true regression function $f_0$ in
  A2. In particular, we do not assume that $f_0$ is additive. In
  important direction for future work is to weaken assumption A1.
  Our simulation experiments indicate that the procedure can be
  effective even when the relevant and irrelevant variables are correlated.
\end{remark}

\begin{theorem}[Controlling false positives]
\label{thm:false_positive}
Suppose assumptions A1-A4 hold. Define $\tilde{\sigma} \equiv \max(\sigma, B)$,
and suppose 
\begin{equation}
\lambda_n \geq 8 s \tilde{\sigma}  \sqrt{ \frac{\log^2 np}{n}}.
\end{equation}  
Then with probability at least $ 1 - \frac{12}{np}$, for all $k \in
S^c$, and for all $i'=1,\ldots, n$,
\begin{equation}
\Big| \frac{1}{2n}\hat{r}^\tran \mathbf{1}_{(i':n)_k} \Big| < \lambda_n
\end{equation}
and for all $k \in S^c$, both the AC solution $\hat{f}_k$ from optimization~\eqref{opt:alternate_opt} and the DC solution $\hat{g}_k$ from optimization~\eqref{opt:alternate_opt_concave} are zero. 
\end{theorem}

The proof of Theorem~\ref{thm:false_positive} exploits independence of
$\hat{r}$ and $X_k$ from A1; when $\hat{r}$ and $X_k$ are independent,
$\hat{r}^\tran \mathbf{1}_{(i':n)}$ is the sum of $n - i' +1$ random
coordinates of $\hat{r}$.  We can then use concentration of
measure results for sampling without replacement to argue that $|
\frac{1}{n} \hat{r}^\tran\mathbf{1}_{(i':n)}|$ is small with high
probability. The result of Theorem~\ref{thm:deterministic} is then
used. The full proof of Theorem~\ref{thm:false_positive} is in
Section~\ref{sec:false_positive_proof} of the Appendix.

\begin{theorem}[Controlling false negatives]
\label{thm:false_negative}
Suppose assumptions A1-A4 hold. Let $\hat{f}$ be any AC solution to
the restricted regression with $B$-boundedness constraint, and let
$\hat{g}_k$ be any DC solution to the restricted regression with
$B$-boundedness constraint. Let $\tilde{\sigma}$ denote $\max(\sigma,
B)$.  Suppose 
\begin{equation}
\lambda_n \leq 9 s \tilde{\sigma} \sqrt{\frac{1}{n} \log^2 np}
\end{equation}
and that $n$ is sufficiently large so that 
\begin{equation}
\frac{n^{4/5}}{\log np} \geq B^4 \tilde{\sigma}^2 s^5.
\end{equation}
Assume that the signal-to-noise ratio satisfies
\begin{align}
\frac{\alpha_{+}}{\tilde{\sigma}} & \geq c B^2
\sqrt{\frac{s^5}{n^{4/5}} \log^2 np}\\
\frac{\alpha_{-}^2}{\tilde{\sigma}} &\geq c B^2
\sqrt{\frac{s^5}{n^{4/5}} \log^2 np}
\end{align}
where $c$ is a constant.  Then with probability at least $1 -
\frac{C}{n}$ for some constant $C$, 
$\hat{f}_k \neq 0$ or $\hat{g}_k \neq 0$ 
for all $k \in S$.
\end{theorem}

This is a finite sample version of
Theorem~\ref{thm:convex_faithful}. We need stronger assumptions in
Theorem~\ref{thm:false_negative} to use our additive faithfulness
result, Theorem~\ref{thm:convex_faithful}. The full proof of Theorem~\ref{thm:false_negative} is in Section~\ref{sec:false_negative_proof} of the appendix.

Combining Theorems~\ref{thm:false_positive} and
~\ref{thm:false_negative} 
we obtain the following result.
\begin{corollary}
  Suppose the assumptions of Theorem~\ref{thm:false_positive} and
  Theorem~\ref{thm:false_negative} hold.  
Then with probability at least $1-\frac{C}{n}$
\begin{align}
\hat{f}_k \neq 0 \trm{ or }\; \hat{g}_k \neq 0 &\trm{ for all } k \in S\\
\hat{f}_k = 0 \trm{ and }\; \hat{g}_k = 0 & \trm{ for all } k \notin S
\end{align}
for some constant $C$.

\end{corollary}
The above corollary implies that consistent variable selection is
achievable at the same exponential scaling of the ambient dimension
scaling attained with parametric models, $p = O(\exp(n^c))$ for $c<1$.
The cost of nonparametric modeling through shape constraints is
reflected in the scaling with respect to the number of relevant
variables, which can scale as $s = o(n^{4/25})$.

\begin{remark}
  \citet{dalalyan:12} have shown that under tradtional smoothness
  constraints, even with a product distribution, variable selection is
  achievable only if $n > O(e^{s_0})$. It is interesting to observe that
  because of additive faithfulness, the convexity assumption enables a
  much better scaling of $n = O(\textrm{poly}(s_0))$, demonstrating that
  geometric constraints can be quite different from the previously
  studied smoothness conditions.
\end{remark}


 

\def\x{x}
\def\Q{Q}
\def\bds#1{#1}
\def\tts#1{\texttt{\small #1}}

\section{Experiments}
\label{sec:thesims}
We first illustrate our methods using a simulation of the model
\begin{equation}\nonumber
         Y_i = \x_{iS}^{\top}\Q\x_{iS} + \epsilon_i \quad (i=1,2,\ldots,n).
\end{equation}
Here $\x_{i}$ denotes data sample $i$ drawn from $\mathcal{N}(\bds{0}, \Sigma)$, 
$\x_{iS}$ is a subset of $\x_i$ with dimension $|S|=5$, where $S$
represents the relevant variable set, and 
$\epsilon_i$ is additive noise drawn from $\mathcal{N}(0,1)$. 
The matrix $\Q$ is a symmetric positive definite matrix of dimension $|S|\times{}|S|$. 
Note that if $\Q$ is diagonal, then the true function is convex
additive; 
otherwise the true function is convex but not additive.
For all simulations in this section, we set $\lambda=4\sqrt{{\log(np)}/{n}}$. 

In the first simulation, we vary the ambient dimension $p$. We set $Q$ as one on the diagonal and $1/2$ on the off-diagonal with $0.5$ probability, and choose $n=100, 200,\ldots,1000$ and $p=64,128,256$ and $512$. We use independent design by setting $\Sigma = \textrm{I}_p$.
For each $(n,p)$ combination, we generate $100$ independent data
sets. For each data set we use the AC/DC algorithm and mark feature $j$ as irrelevant if both the AC estimate $\| \hat{f}_j \|_\infty$ and the DC estimate $\| \hat{g}_k \|_\infty$ are smaller than $10^{-6}$.
We plot the probability of exact
support recovery over the $100$ data sets in Figure \ref{Support}(a).  We
observe that the algorithm performs consistent variable selection even if the dimensionality is large. To give the reader a
sense of the running speed, for a 
data set with $n=1000$ and $p=512$, the code runs in roughly two 
minutes on a machine with 2.3 GHz Intel Core i5 CPU and 4 GB memory.

In the second simulation, we vary the sparsity of the $Q$ matrix, that is, we vary the extent to which the true function is non-additive. We generate four $\Q$ matrices
plotted in Figure \ref{Support}(b), where the diagonal elements are all one and
the off-diagonal elements are $\frac{1}{2}$ with probability $\alpha$
($\alpha=0,0.2,0.5,1$ for the four cases). We fix $p=128$ and choose
$n=100,200,\ldots,1000$. We use independent design by seting $\Sigma = \textrm{I}_p$. We again run the AC/DC optimization on $100$
independently generated data sets and plot the probability of recovery
in Figure \ref{Support}(c). The results demonstrate that AC/DC performs
consistent variable selection even if the true function is not additive (but
still convex).

In the third simulation, we use correlated design and vary the correlation. We let
$\x_i$ be drawn from $\mathcal{N}(\bds{0},\bds{\Sigma})$
instead of $\mathcal{N}(\bds{0},\bds{I}_{p})$, with
$\Sigma_{ij}=\nu^{|i-j|}$. We use the non-additive $\Q$, same as in the second experiment, with $\alpha=0.5$ and fix $p=128$.  The recovery curves for $\nu=0.2, 0.4,
0.6, 0.8$ are shown in Figure \ref{Support}(d). As can be seen, for
design of moderate correlation, AC/DC can still select relevant
variables well.

We next use the Boston housing data rather than simulated data. This data set
contains 13 covariates, 506 samples and one response variable
indicating housing values in suburbs of Boston. The data and detailed description
can be found on the UCI Machine Learning Repository website\footnote{\url{http://archive.ics.uci.edu/ml/datasets/Housing}}. 

We first use all $n=506$ samples (with standardization) in the AC/DC algorithm,
using a set of candidate regularization parameters $\{\lambda^{(t)}\}$
ranging from $\lambda^{(1)} = 0$ (no regularization) to $2$. For each $\lambda^{(t)}$
we obtain a function value matrix $\bds{h}^{(t)}$ with $p=13$
columns. The non-zero columns in this matrix indicate the variables selected
using $\lambda^{(t)}$.  We plot $\|\bds{h}^{(t)}\|_{\infty}$ and
the column-wise mean of $\bds{h}^{(t)}$ versus the normalized norm
$\frac{\|\bds{h}^{(t)}\|_{\infty,1}}{\|\bds{h}^{(1)}\|_{\infty,1}}$
in Figures \ref{Boston}(a) and \ref{Boston}(b).  For comparison, we
plot the LASSO/LARS result in a similar way in Figure \ref{Boston}(c).
From the figures we observe that the first three variables selected by
AC/DC and LASSO are the same: \tts{LSTAT}, \tts{RM} and \tts{PTRATIO},
consistent with previous findings~\citep{SpAM:07}.  The fourth
variable selected by AC/DC is \tts{INDUS} (with $\lambda^{(t)}=0.7$).
We then refit AC/DC with only these four variables without
regularization, and plot the estimated additive functions in Figure
\ref{Boston}(e). When refitting, we constrain a component to be convex if it is non-zero in the AC stage and concave if it is non-zero in the DC stage. As can be seen, these functions contain clear
nonlinear effects which cannot be captured by LASSO. The shapes of
these functions, including the concave shape of the \tts{PTRATIO} function, are in agreement with those obtained by
SpAM~\citep{SpAM:07}. 

Next, in order to quantitatively study the predictive performance, we
run 3 times 5-fold cross validation, following the same procedure
described above---training, variable selection and refitting.  A plot
of the mean and standard deviation of the predictive mean squared
error (MSE) is shown in Figure \ref{Boston}(d). Since for AC/DC the same
regularization level $\lambda^{(t)}$ may lead to a slightly different number of selected
variables in different folds and runs, the values on the $x$-axis
for AC/DC are not necessarily integers. The figure clearly shows that AC/DC has a 
lower predictive MSE than LASSO.  We also compared the performance of
AC/DC with that of Additive Forward Regression (AFR) presented
in~\cite{Xi:09}, and found that they are similar.  The main advantages
of AC/DC compared with AFR and SpAM are that there are no smoothing
parameters required, and the optimization is formulated
as a convex program, guaranteeing a global optimum.


\begin{figure*}[!t]
\begin{center}
\begin{tabular}{cccc}
\hskip-10pt
\includegraphics[width=.26\textwidth]{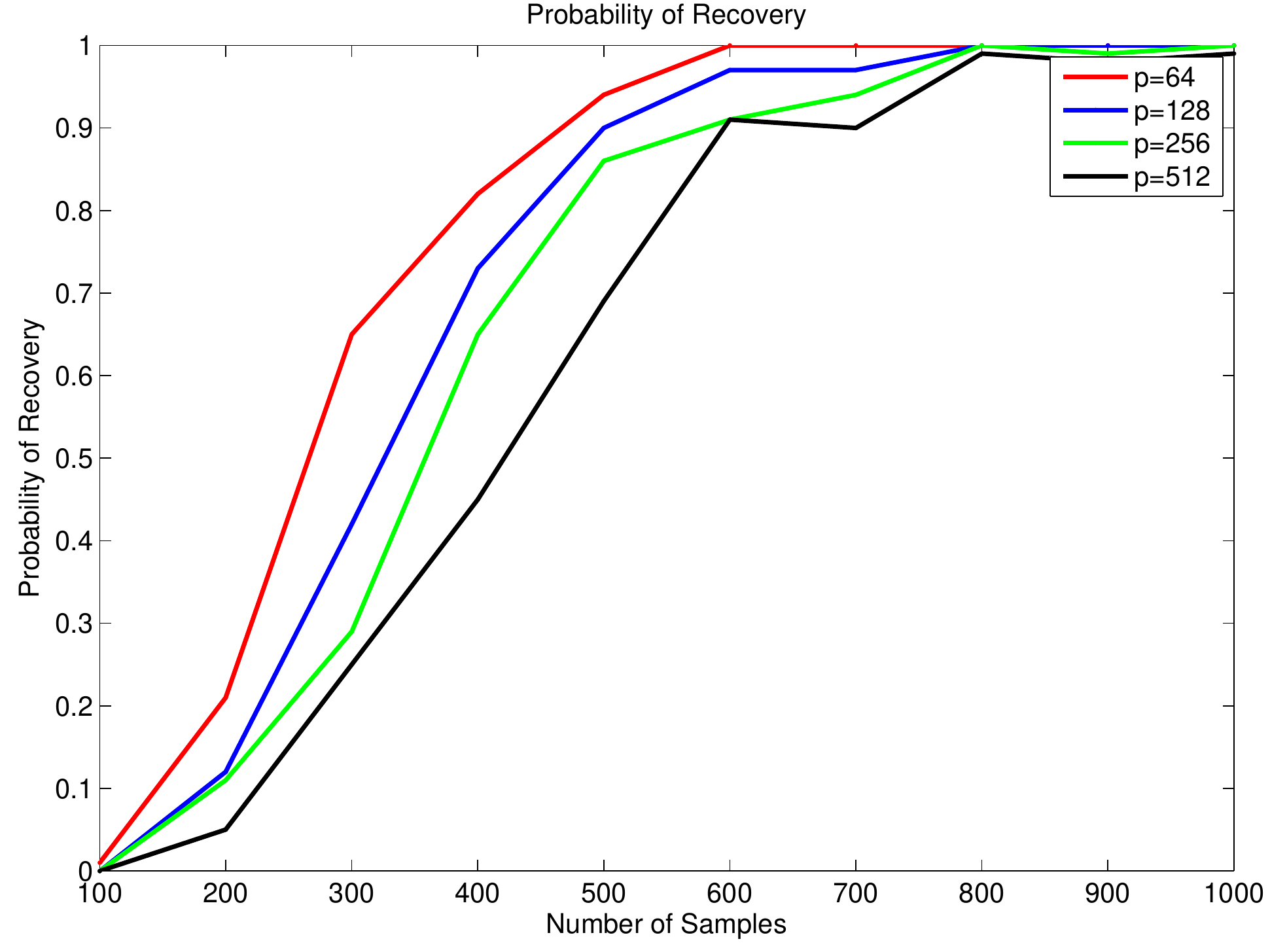} &
\hskip-10pt
\includegraphics[width=.26\textwidth]{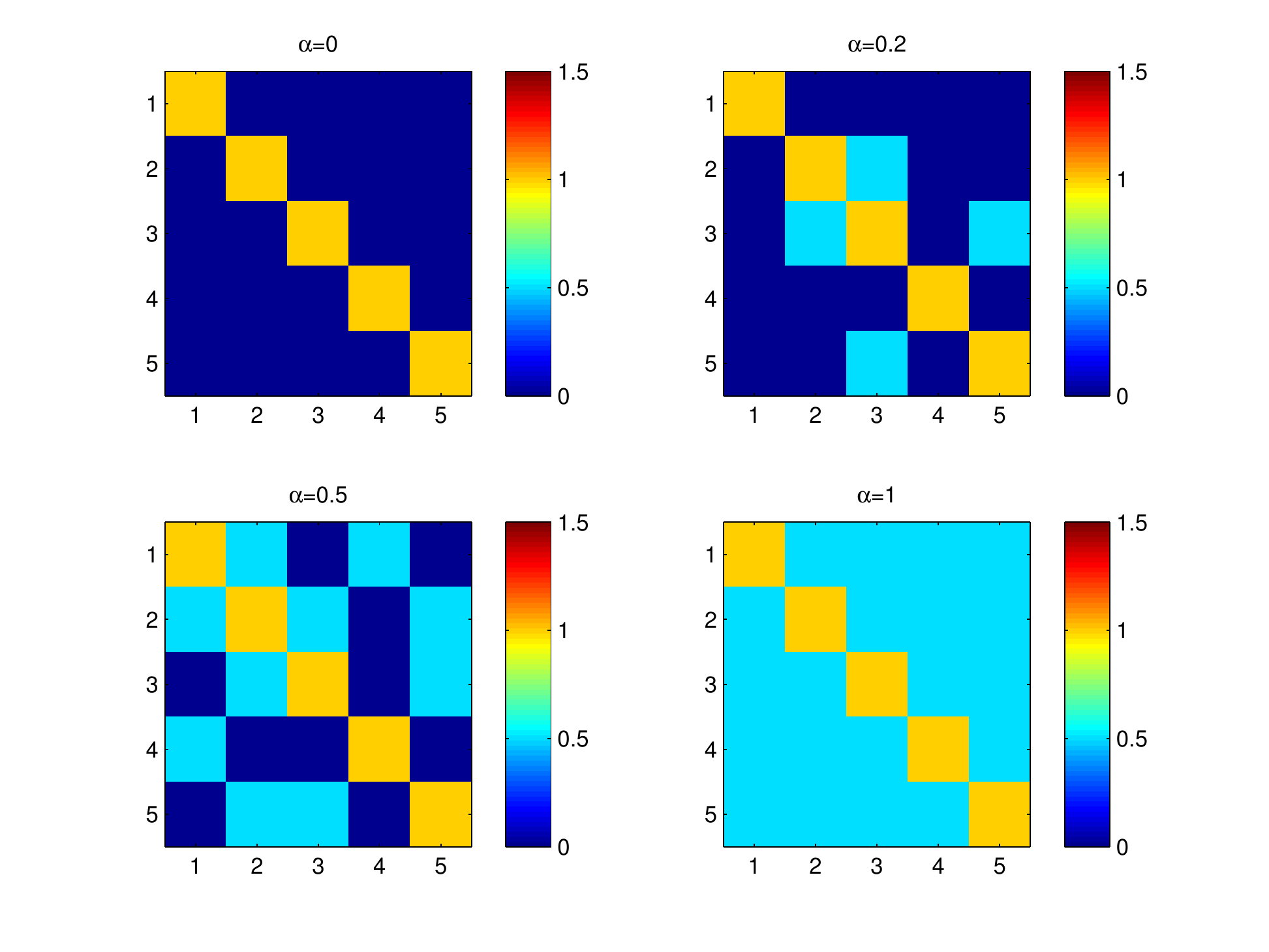} &
\hskip-10pt
\includegraphics[width=.26\textwidth]{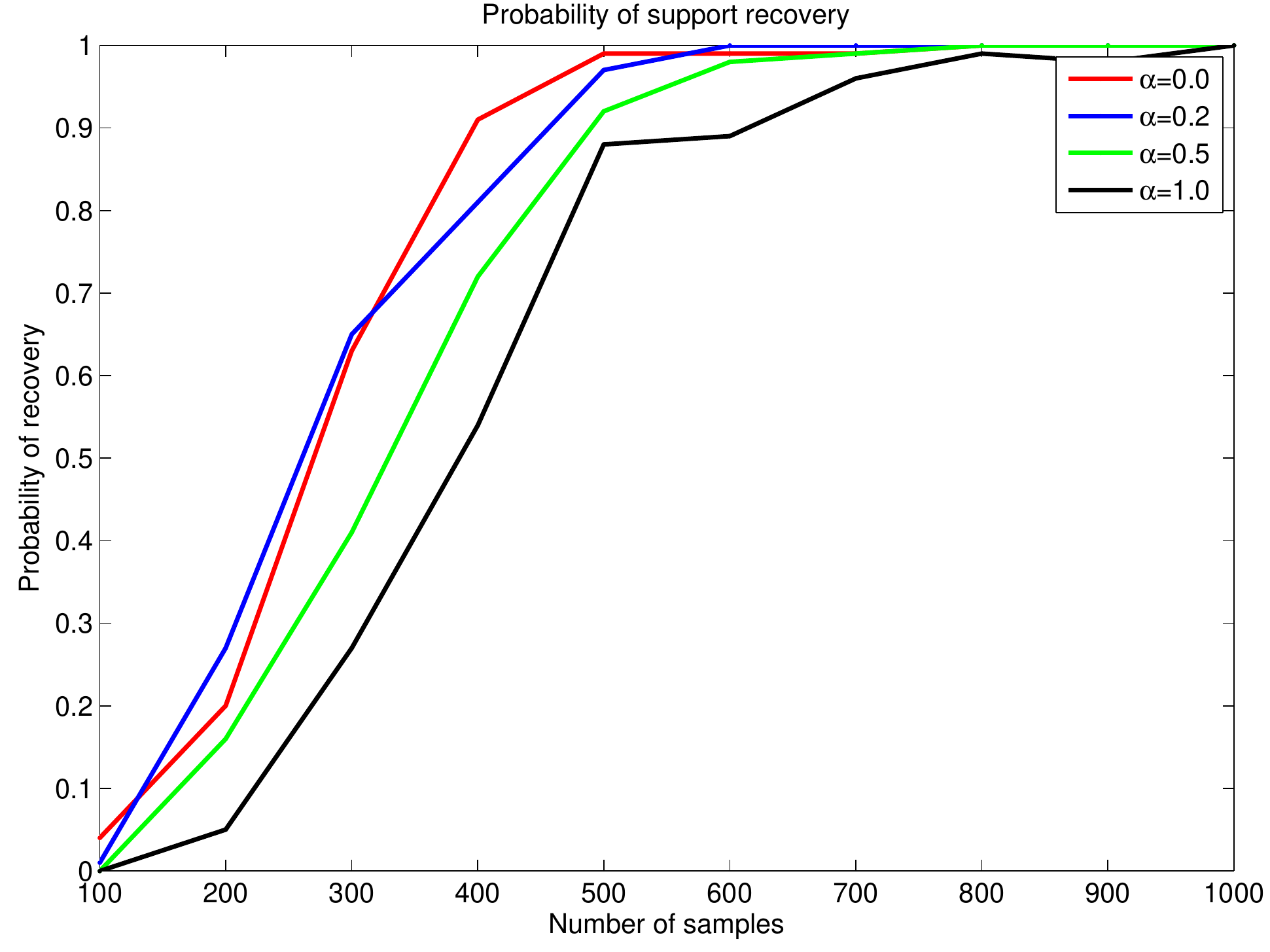} &
\hskip-10pt
\includegraphics[width=.26\textwidth]{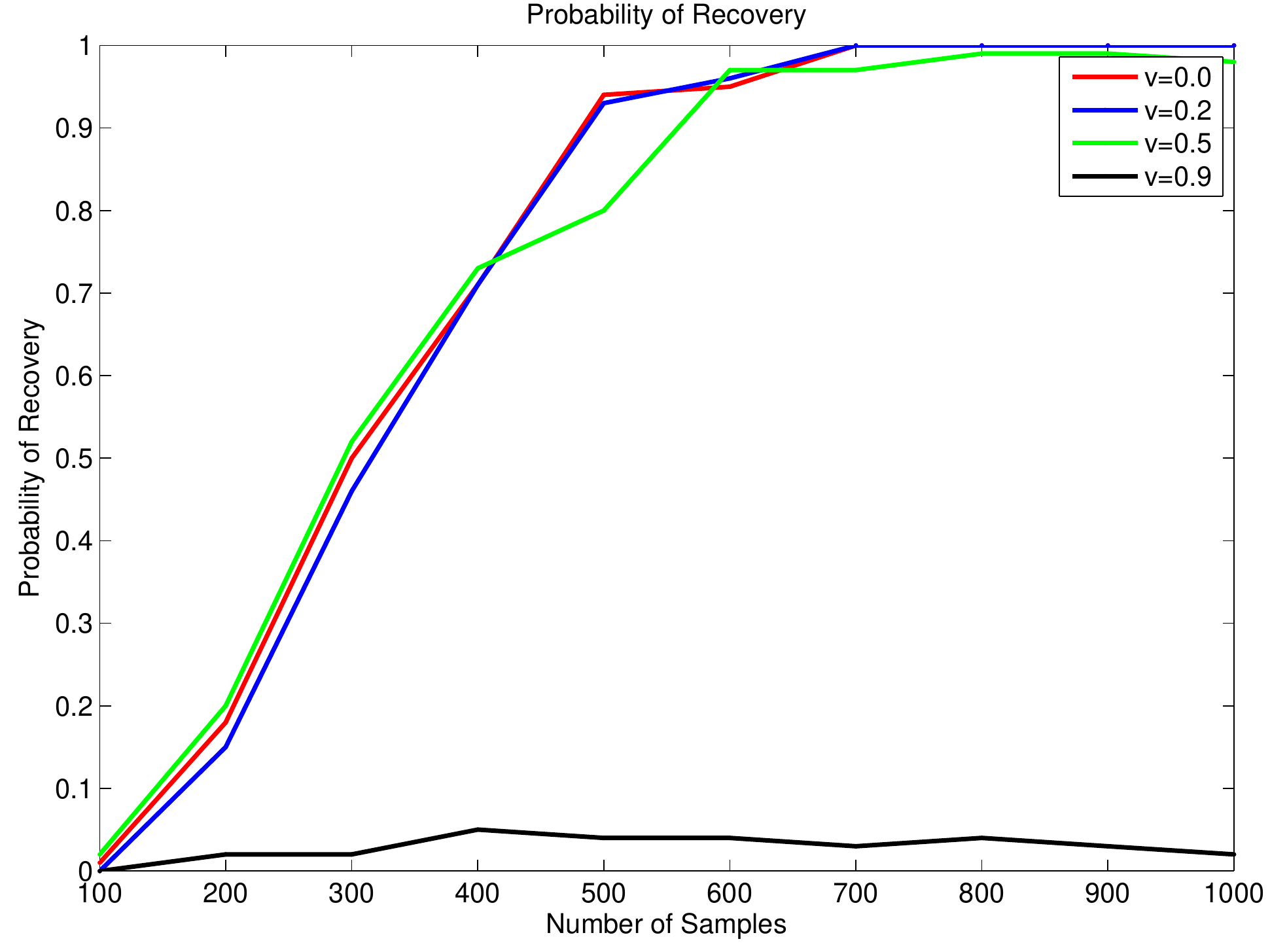}  \\
\hskip-10pt (a) additive model & 
\hskip-10pt (b) four $\Q$ matrices &
\hskip-10pt (c) non-additive models & 
\hskip-10pt (d) correlated design
\end{tabular}
\end{center}
\caption{Support recovery results where the additive assumption is
  correct (a), incorrect (b), (c), and with correlated design (d).}\label{Support}
\vskip10pt

\begin{center}
\begin{tabular}{cc}
  \includegraphics[width=.38\textwidth]{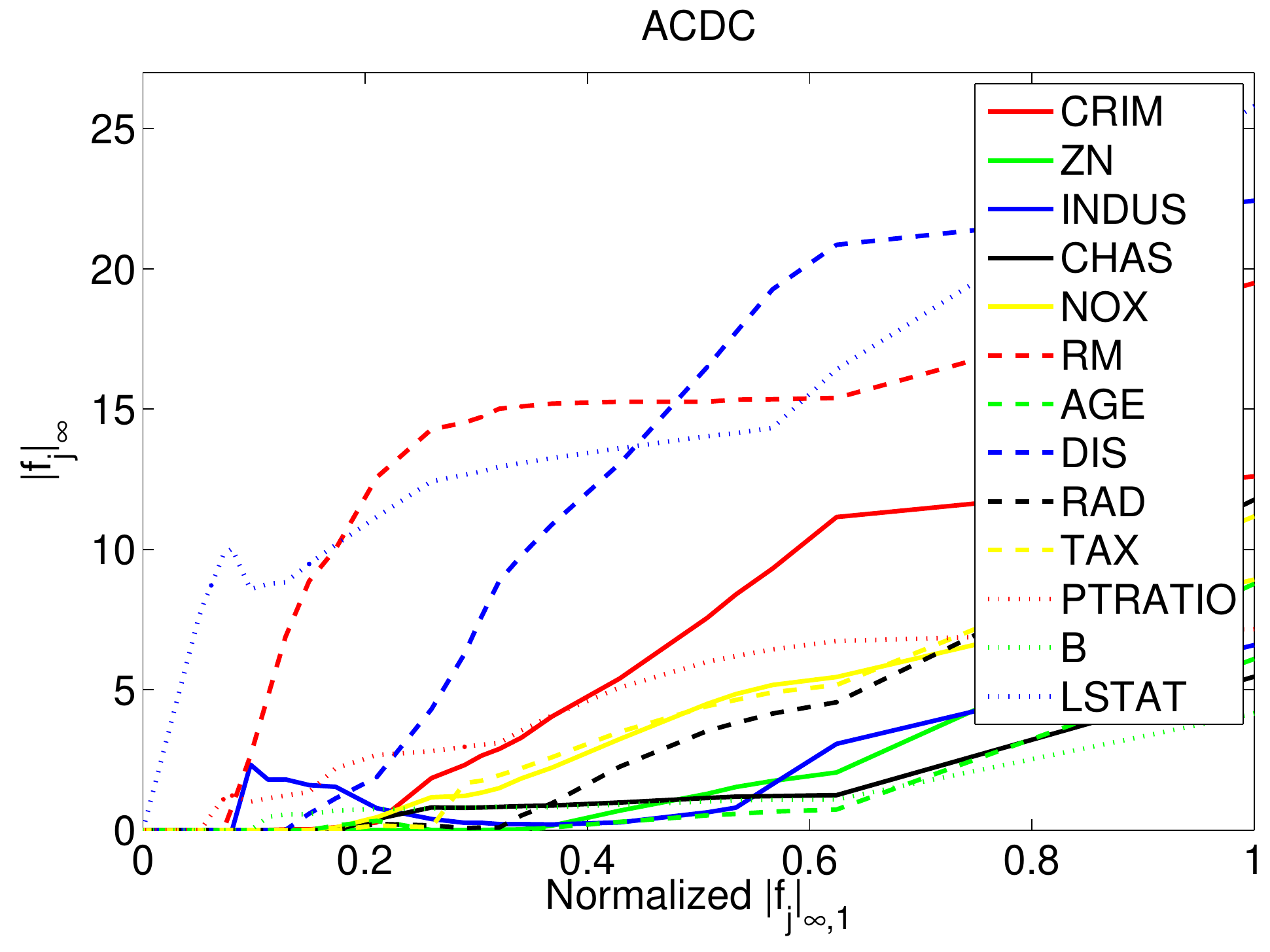} &
  \includegraphics[width=.38\textwidth]{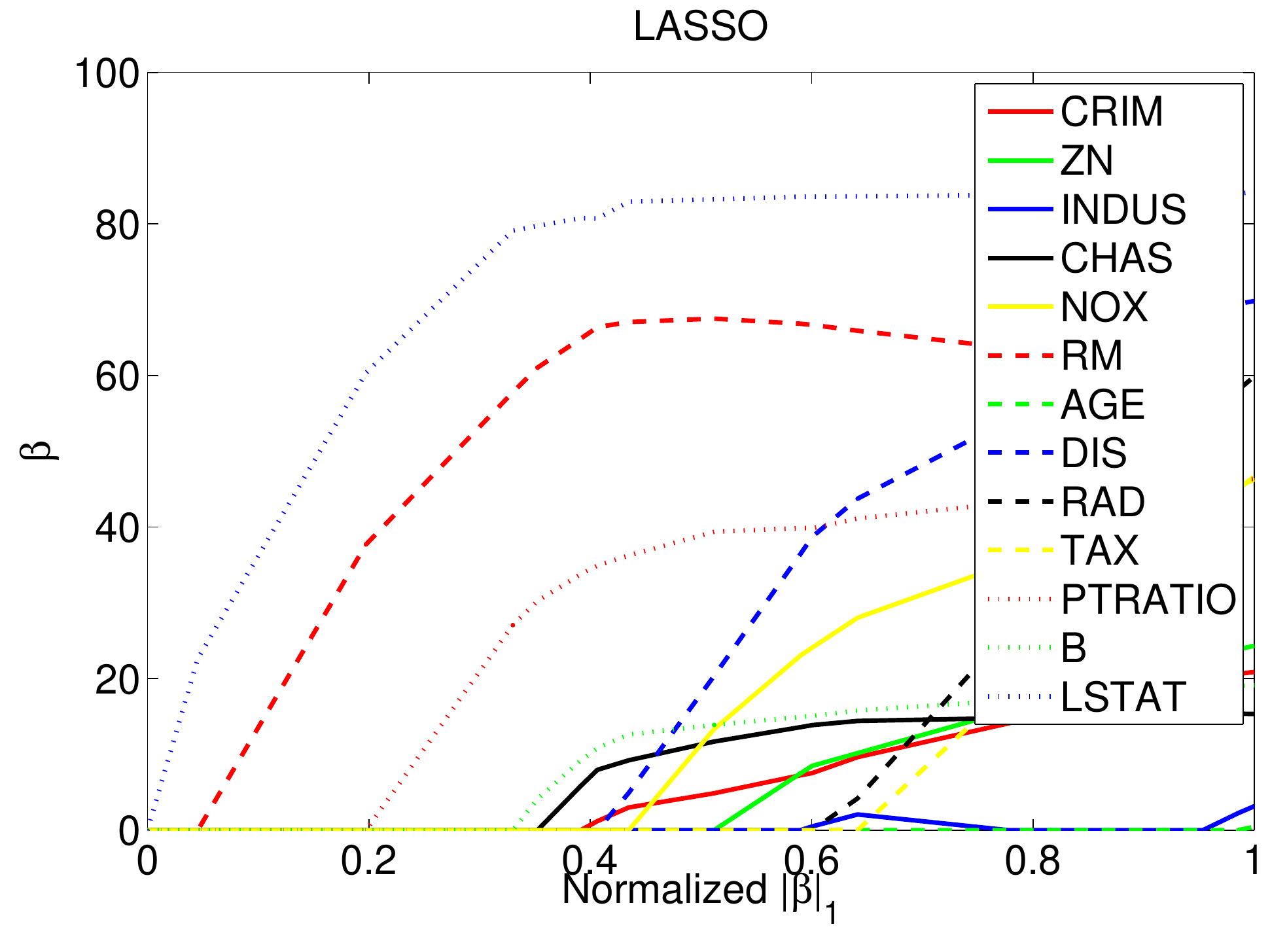} 
\\
AC/DC $\|f_k\|_\infty$ paths & 
LASSO $|\beta_k|$ paths \\
  \includegraphics[width=.38\textwidth]{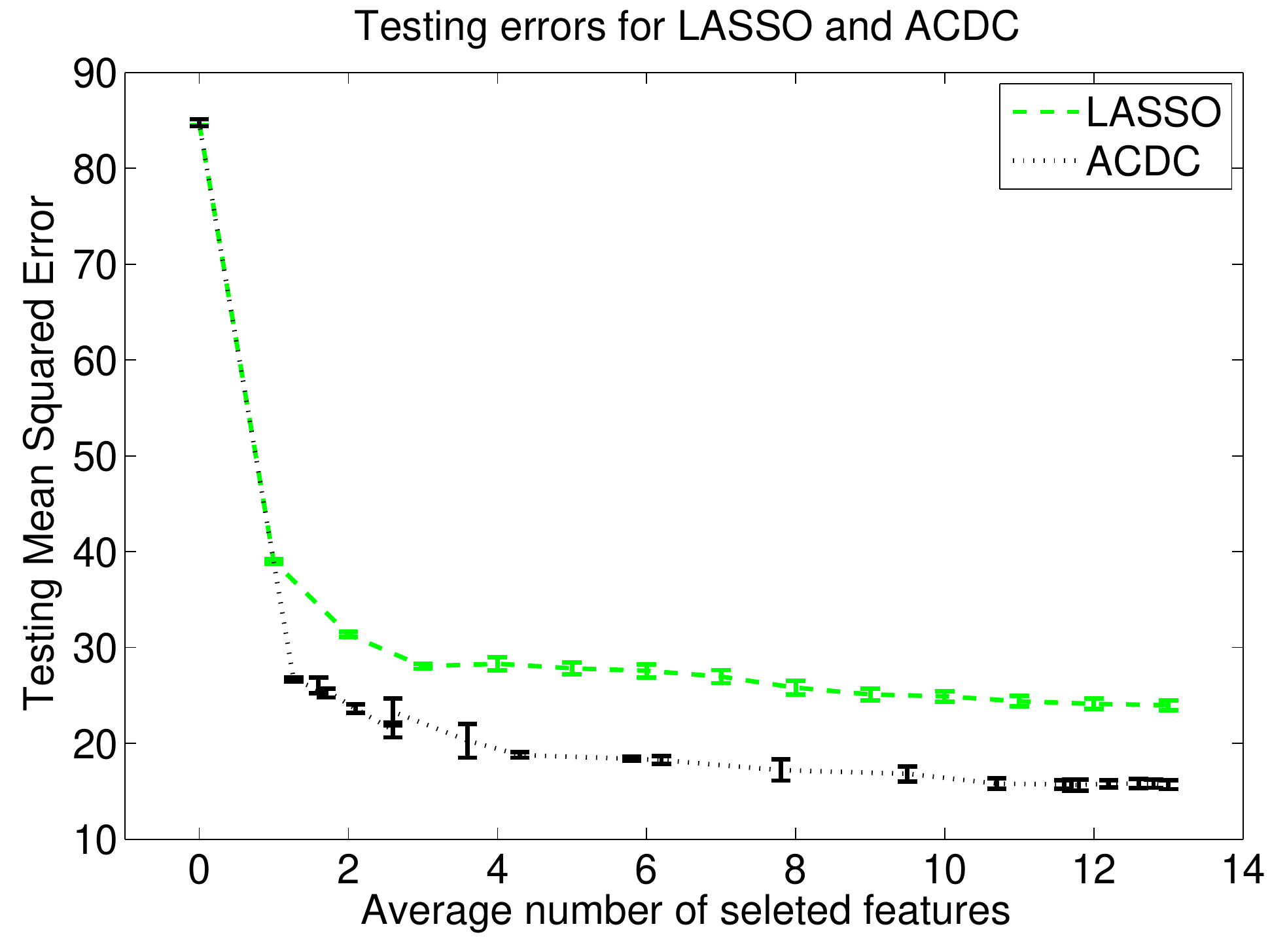} &
  \includegraphics[width=.45\textwidth]{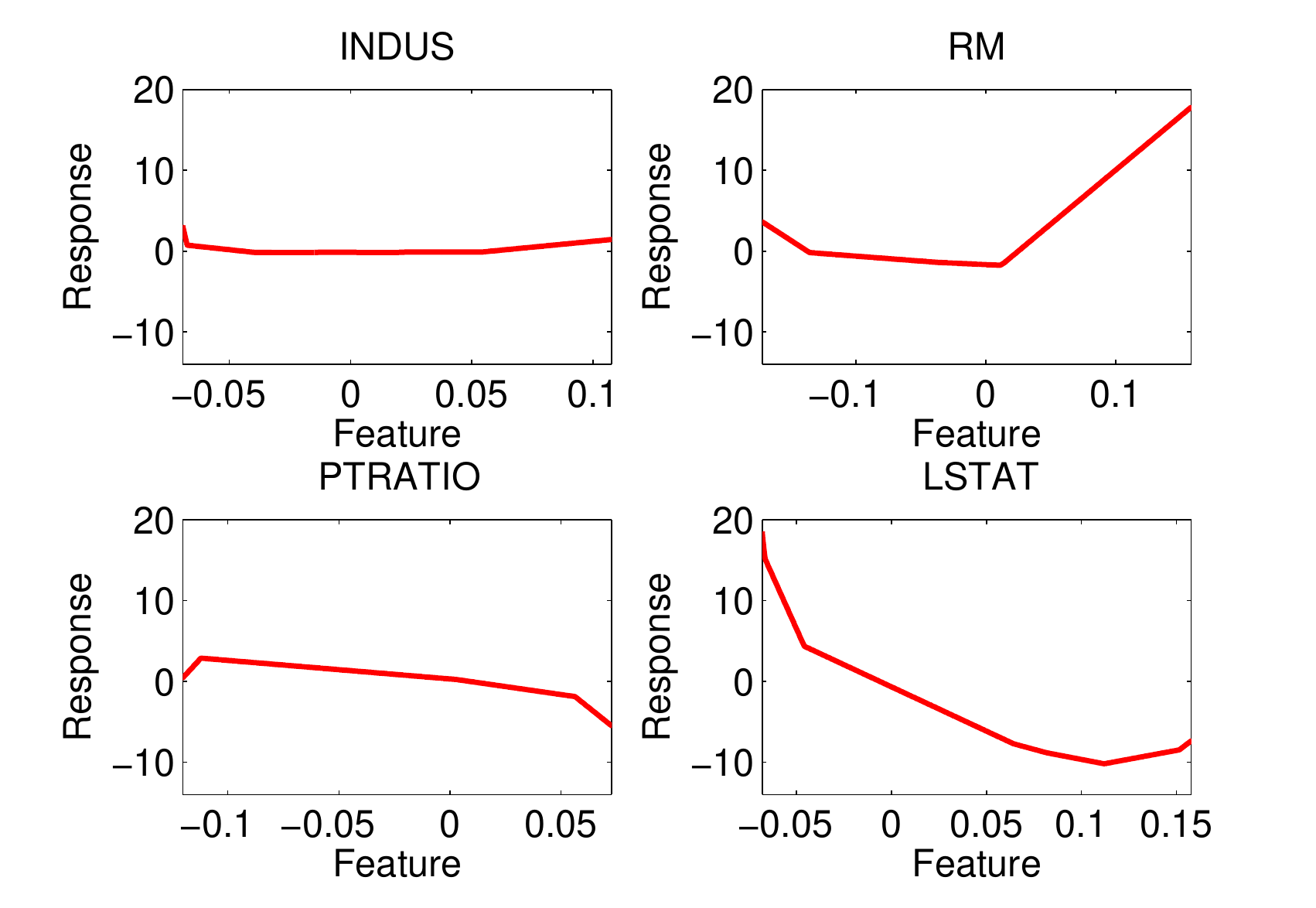}
\\
predictive MSE & estimated functions from AC/DC
\end{tabular}
\end{center}
\caption{Results on Boston housing data, showing regularization paths,
 MSE and fitted functions.}\label{Boston}
\end{figure*}

 

\section{Discussion}

We have introduced a framework for estimating high dimensional but
sparse convex functions.  Because of the special properties of
convexity, variable selection for convex functions enjoys additive
faithfulness---it suffices to carry out variable selection over an
additive model, in spite of the approximation error this introduces.
Sparse convex additive models can be optimized using block coordinate
quadratic programming, which we have found to be effective and
scalable.  We established variable selection consistency results,
allowing exponential scaling in the ambient dimension.  We expect
that the technical assumptions we have used in these analyses can be
weakened; this is one direction for future work.  Another interesting
direction for building on this work is to allow for additive models that
are a combination of convex and concave components.  If the
convexity/concavity of each component function is known, this again
yields a convex program.  The challenge is to develop a method to
automatically detect the concavity or convexity pattern of the
variables.


\section*{Acknowledgements}
Research supported in part by NSF grants IIS-1116730, 
AFOSR grant FA9550-09-1-0373, ONR grant
N000141210762, and an Amazon AWS in Education Machine Learning
Research grant.

\clearpage
\bibliographystyle{imsart-nameyear}
\bibliography{local}

\clearpage

\section{Supplement:  Proofs of Technical Results}

 \subsection{Proof of the Deterministic Condition for Sparsistency}
 \label{sec:deterministic_proof}
 
We restate Theorem~\ref{thm:deterministic} first for convenience. 
The following holds regardless of whether we impose the $B$-boundedness condition (see discussion at beginning of Section~\ref{sec:finitesample} for definition of the $B$-boundedness condition).
 
\begin{theorem} 
Let $\{\hat{d}_k \}_{k \in S}$ be a minimizer of the restricted regression, that is, the solution to optimization (\ref{opt:alternate_opt}) where we restrict $k \in S$. 
Let $\hat{r} \coloneqq Y - \sum_{k \in S} \bar{\Delta}_k \hat{d}_k$ be the restricted regression residual.

Let $\pi_k(i)$ be an reordering of $X_k$ in ascending order so that $X_{k \pi_k(n)}$ is the largest entry. Let $\mathbf{1}_{\pi_k(i:n)}$ be 1 on the coordinates $\pi_k(i),\pi_k(i+1),...,\pi_k(n)$ and 0 elsewhere. Define $\mathsf{range}_k = X_{k\pi_k(n)} - X_{k \pi_k(1)}$.

Suppose for all $k\in S^c$ and for all $i=1,...,n$, $\lambda_n \geq \mathsf{range}_k | \frac{32}{n} \hat{r}^\tran \mathbf{1}_{\pi_k(i:n)}|$, and $\max_{i=1,...,n-1} \frac{X_{k\pi_k(i+1)} - X_{k\pi_k(i)}}{\mathsf{range}_k} \geq \frac{1}{16}$, and $\mathsf{range}_k \geq 1$.

Then the following are true:
\begin{enumerate}
\item Let $\hat{d}_k = 0$ for $k \in S^c$, then \{$\hat{d}_k\}_{k=1,...,p}$ is an optimal solution to optimization~\ref{opt:alternate_opt}. Furthermore, any solution to the optimization program \ref{opt:alternate_opt} must be zero on $S^c$.
\item For all $k \in S^c$, the solution to optimization~\ref{opt:alternate_opt_concave} must be zero and unique.
\end{enumerate}
\end{theorem}

\begin{proof} 
We will omit the $B$-boundedness constraint in our proof here. It is easy to verify that the result of the theorem still holds with the constraint added in.

We begin by considering the first item in the conclusion of the theorem.
We will show that with $\{\hat{d}_k\}_{k=1,..,p}$ as constructed, we
can set the dual variables to satisfy the 
complementary slackness and stationary conditions: $\nabla_{d_k} \mathcal{L}(\hat{d})  = 0$ for all $k$. 

The Lagrangian is
\begin{equation}
\label{eqn:full_lagrange}
\mathcal{L}( \{ d_k \}, \nu) = 
  \frac{1}{2n} \Big\| 
    Y - \sum_{k=1}^p  \bar{\Delta}_k d_k  \Big\|_2^2 + 
    \lambda \sum_{k=1}^p \| \bar{\Delta}_k d_k \|_\infty -
    \sum_{k=1}^p \sum_{i=2}^{n-1} \nu_{ki} d_{ki} 
\end{equation}
with the constraint that $\nu_{ki} \geq 0$ for all $k,i$.

Because $\{\hat{d}_k\}_{k \in S}$ is by definition the optimal solution of the restricted regression, it is a consequence that stationarity holds for $k \in S$, that is, $\partial_{ \{ d_k \}_{k \in S} } \mathcal{L}(d) = 0$, and that the dual variables $\nu_k$ for $k \in S$ satisfy complementary slackness.

We now verify that stationarity holds also for $k \in S^c$. We fix one dimension $k \in S^c$ and let $\hat{r} = Y - \sum_{k' \in S} \bar{\Delta}_{k'} \hat{d}_{k'}$. 
The Lagrangian form of the optimization, in terms of just $d_k$, is
\[
\mathcal{L}(d_k, \nu_k) =
  \frac{1}{2n} \big\| Y - \sum_{k' \in S} \bar{\Delta}_{k'} d_{k'} 
  -  \bar{\Delta}_k d_k \big\|_2^2 
   + \lambda \| \bar{\Delta}_k d_k\|_\infty
  - \sum_{i=2}^{n-1} \nu_{ki} d_{ki}
\]
with the constraint that $v_{ki} \geq 0$ for $i=2,...,n-1$. 

The derivative of the Lagrangian is
\begin{align*}
\partial_{d_k} \mathcal{L}(d_k) =  -\frac{1}{n} \bar{\Delta}_k^\tran ( Y - \sum_{k'\in S} \bar{\Delta}_{k'} d_{k'}  - \bar{\Delta}_k d_k )
        + \lambda \bar{\Delta}_k^\tran \mathbf{u}
      - \left( \begin{array}{c} 0 \\ \nu_k \end{array} \right)
\end{align*}
where $\mathbf{u}$ is the subgradient of $\| \bar{\Delta}_k d_k \|_\infty$. If $\bar{\Delta}_k d_k = 0$, then $\mathbf{u}$ can be any vector whose $L_1$ norm is less than or equal to $1$. $\nu_k \geq 0$ is a vector of Lagrangian multipliers. $\nu_{k1}$ does not exist because $d_{k1}$ is not constrained to be non-negative.

We now substitute in $d_{k'} = \hat{d}_{k'}$ for $k' \in S$, $d_k = 0$ for $k \in S$, and $r = \hat{r}$ and show that the $\mathbf{u}, \nu_k$ dual variables can be set in a way to ensure that stationarity:
\begin{align*}
\partial_{d_k} \mathcal{L}(\hat{d}_k) = -\frac{1}{n} \bar{\Delta}_k^\tran\hat{r} + \lambda \bar{\Delta}_k^\tran \mathbf{u}
           - \left( \begin{array}{c} 0 \\ \nu_k \end{array} \right) = 0 
\end{align*}
where $\| \mathbf{u} \| \leq 1$ and $\nu_k \geq 0$. It clear that to show stationarity, we only need to show that $[-\frac{1}{n} \bar{\Delta}_k^\tran \hat{r} + \lambda \bar{\Delta}_k^\tran \mathbf{u}]_j = 0$ for $j=1$ and $\geq 0$ for $j=2,...,n-1$.

To ease notational burden, let us reorder the samples in ascending order so that the $i$-th sample is the $i$-th smallest sample. We will from here on write $X_{ki}$ to denote $X_{k \pi_k(i)}$.

Define $i^*$ as the largest index such that $\frac{X_{kn} - X_{ki^*}}{X_{kn} - X_{k1}} \geq 1/2$. 
We will construct $\mathbf{u} = (a - a', 0, ..., -a, 0,..., a')$ where
$a,a'$ are positive scalars, where $-a$ lies at the $i^*$-th coordinate, and where the coordinates of $\mathbf{u}$ correspond
to the new sample ordering. 

We define

\begin{align*}
\kappa &= \frac{1}{\lambda n} \big[ \Delta_k^\tran \hat{r} \big]_1 \\
a' &= \frac{X_{kn} - X_{k1}}{X_{kn} -  X_{ki^*}} \kappa + 
     \frac{X_{ki^*}-X_{k1}}{X_{kn} - X_{k i^*}} \frac{1}{8} \\
a &= \frac{X_{kn} - X_{k1}}{X_{kn} -  X_{ki^*}} \kappa + 
     \frac{X_{kn}-X_{k1}}{X_{kn} - X_{ki^*}} \frac{1}{8} \\
\end{align*}

and we verify two facts: first that the KKT stationarity is satisfied and second, that $\| \mathbf{u} \|_1 < 1$ with high probability. Our claim is proved immediately by combining these two facts.

Because $\hat{r}$ and $\mathbf{u}$ are both centered vectors, $\bar{\Delta}_k^\tran \hat{r} = \Delta_k^\tran \hat{r}$ and likewise for $\mathbf{u}$. Therefore, we need only show that for $j=1$, $\lambda \big[ \Delta_k^\tran \mathbf{u} \big]_j = \big[ \frac{1}{n} \Delta_k^\tran \hat{r} \big]_j$ and that for $j = 2,..., n-1$, $\lambda \big[ \Delta_k^\tran \mathbf{u} \big]_j \geq \big[ \frac{1}{n} \Delta_k^\tran \hat{r} \big]_j$.

With our explicitly defined form of $\mathbf{u}$, we can characterize
\begin{align}
\big[ \Delta_k^\tran \mathbf{u} \big]_j = 
  \left \{ \begin{array}{cc} 
   (-a + a')(X_{k i^*} - X_{kj})+ a'(X_{kn} - X_{k i^*})
   & \trm{ if } j \leq i^* \\
   a'(X_{kn} - X_{kj}) & \trm{ if } j \geq i^* 
     \end{array} \right.
\end{align} 

It is straightforward to check that $\big[ \lambda \Delta_k^\tran \mathbf{u} \big]_1 = \lambda \kappa = \frac{1}{n} \big[ \Delta_k^\tran \hat{r} \big]_1$.

To check that other rows of stationarity condition holds, we characterize $[\frac{1}{n} \Delta_k^\tran \hat{r}]_j$:
\begin{align*}
[\frac{1}{n} \Delta_k^\tran \hat{r}]_j &= \frac{1}{n} \sum_{i > j} (X_{ki} - X_{kj}) \hat{r}_i \\
  & = \frac{1}{n} \sum_{i > j} \sum_{j < i' \leq i} \mathsf{gap}_{i'} \hat{r}_i \\
 & = \frac{1}{n} \sum_{i' > j} \mathsf{gap}_{i'} \sum_{i \geq i'} \hat{r}_i \\
 & = \frac{1}{n} \sum_{i' > j} \mathsf{gap}_{i'} \mathbf{1}_{i':n}^\tran \hat{r} 
\end{align*}
where we denote $\mathsf{gap}_{i'} = X_{ki'} - X_{k(i'-1)}$.

We pause for a second here to give a summary of our proof strategy. We leverage two critical observations: first, any two adjacent coordinates in the vector $\frac{1}{n} \Delta_k^\tran \hat{r}$ cannot differ by too much. Second, we defined $a, a'$ such that $-a+a' = -\frac{1}{8}$ so that $\lambda \Delta_k^\tran \mathbf{u}$ is a sequence that strictly increases in the first half (for coordinates in $\{1,...,i^*\}$) and strictly decreases in the second half. 

We know $\frac{1}{n} \Delta_k^\tran \hat{r}$ and $\lambda \Delta_k^\tran \mathbf{u}$ are equal in the first coordinate. We will show that the second sequence increases faster than the first sequence which will imply that the second sequence is larger than the first in the first half of the coordinates. We will then work similarly but backwards for the second half.

Following our strategy, We first compare $[\lambda \Delta_k^\tran \mathbf{u}]_j$ and $[\frac{1}{n} \Delta_k^\tran \hat{r}]_j$ for $j=1,..., i^*-1$.

For $j = 1,..., i^*-1$, we have that
\begin{align*}
\lambda [ \Delta_k^\tran \mathbf{u}]_{j+1} - \lambda [ \Delta_k^\tran \mathbf{u}]_{j} &= \lambda (a - a') \mathsf{gap}_{j+1} \\
 &\geq -\mathsf{gap}_{j+1} \frac{1}{n} \mathbf{1}_{(j+1):n}^\tran \hat{r} \\
& = [ \frac{1}{n} \Delta_k^\tran \hat{r} ]_{j+1} - [ \frac{1}{n} \Delta_k^\tran \hat{r} ]_{j}
\end{align*}

The inequality follows because $a - a' = \frac{1}{8}$ and thus $\lambda (a - a') \geq \left| \frac{1}{n} \mathbf{1}_{(j+1):n}^\tran \hat{r} \right|$. Therefore, for all $j = 1,...,i^*$:
\begin{align*}
[ \lambda \Delta_k^\tran \mathbf{u}]_j &\geq [ \frac{1}{n} \Delta_k^\tran \hat{r} ]_j
\end{align*}

For $j \geq i^*$, we start our comparison from $j=n-1$. First, we claim that $a' > \frac{1}{32}$. To prove this claim, note that
\begin{align}
|\kappa| &= \Big|\frac{1}{\lambda n} \sum_{i' > 1} \mathsf{gap}_{i'} \mathbf{1}_{i':n}^\tran \hat{r}\Big| \leq \frac{1}{X_{kn} - X_{k1}}\frac{1}{32} \sum_{i' > 1} \mathsf{gap}_{i'} 
 = \frac{1}{32} 
\end{align}
and that
\begin{align*}
\frac{X_{kn} - X_{ki^*}}{X_{kn} - X_{k1}} &= \frac{X_{kn} - X_{k(i^*+1)} + X_{k(i^*+1)} - X_{ki^*}}{X_{kn} - X_{k1}} \leq \frac{1}{2} + \frac{1}{16}
\end{align*}
where the inequality follows because we had assumed that $\frac{X_{k(i+1)} - X_{ki}}{X_{k(n)} - X_{k(1)}} \leq \frac{1}{16}$ for all $i = 1,...,n-1$.

So, we have 
\begin{align*}
a' &= \frac{X_{kn} - X_{k1}}{X_{kn} - X_{ki^*}} \kappa 
   + \frac{ X_{ki^*} - X_{k1}}{X_{kn} - X_{ki^*}} \frac{1}{8} \\
 &= \frac{X_{kn} - X_{k1}}{X_{kn} - X_{ki^*}} \left(
  \kappa + \frac{X_{ki^*} - X_{k1}}{X_{kn} - X_{k1}} \frac{1}{8} \right) \\
&\geq \frac{X_{kn} - X_{k1}}{X_{kn} - X_{ki^*}} \left(
  -\frac{1}{32} + (\frac{1}{2} - \frac{1}{16}) \frac{1}{8} \right) \\
&\geq \frac{1}{1/2 + 1/16} \left(
  -\frac{1}{32} + (\frac{1}{2} - \frac{1}{16}) \frac{1}{8} \right) \\
&\geq \frac{1}{32}
\end{align*}
In the first inequality of the above derivation, we used the fact that $\frac{X_{ki^*} - X_{k1}}{X_{kn} - X_{k1}} \leq \frac{1}{2} - \frac{1}{16}$. In the second inequality, we used the fact that the quantity inside the parentheses is positive and $\frac{X_{kn} - X_{k1}}{X_{kn} - X_{ki^*}} \geq \frac{1}{1/2 + 1/16}$.

Now consider $j=n-1$. 
\[
[ \frac{1}{n} \Delta_k^\tran \hat{r} ]_{n-1} = \frac{1}{n} \mathsf{gap}_n \hat{r}_n 
 \leq \mathsf{gap}_n \frac{\lambda }{32} \leq \lambda \mathsf{gap}_n a' = \lambda [\Delta_k^\tran \mathbf{u}]_{n-1} 
\]

For $j = i^*, ..., n-2$, we have that 
\begin{align*}
\lambda [\Delta_k^\tran \mathbf{u} ]_j - \lambda [\Delta_k^\tran \mathbf{u}]_{j+1} &= 
 \lambda a' \mathsf{gap}_{j+1}  \\
 &\geq \mathsf{gap}_{j+1} \frac{1}{n} \mathbf{1}_{(j+1):n}^\tran \hat{r}  \\
 &\geq  [\frac{1}{n} \Delta_k^\tran \hat{r}]_j - [\frac{1}{n} \Delta_k^\tran \hat{r}]_{j+1}
\end{align*}

Therefore, for $j = i^*,...,n-2$,
\begin{align*}
\lambda [\Delta_k^\tran \mathbf{u}]_j \geq \frac{1}{n} [ \Delta_k^\tran \hat{r}]_j
\end{align*}

We conclude then that $\lambda [\Delta_k^\tran \mathbf{u} ]_j \geq [\frac{1}{n} \Delta_k^\tran \hat{r}]_j$ for all $j = 2,...,n-1$. 

We have thus verified that the stationarity equations hold and now will bound $\| \mathbf{u} \|_1$.

\begin{align*}
\| \mathbf{u} \|_1 = | a - a'| + a + a' \leq \frac{1}{8} + 2 a \leq \frac{1}{8} + 4 |\kappa| + \frac{1}{2}  \leq \frac{1}{8} + \frac{1}{8} + \frac{1}{2} < 1
\end{align*}
In the third inequality, we used the fact that $|\kappa| \leq \frac{1}{32}$.

We have thus proven that there exists one solution $\{ \hat{d}_k
\}_{k=1,...,p}$ such that $\hat{d}_k = 0$ for all $k \in
S^c$. Furthermore, we have shown that the subgradient variables
$\mathbf{u}_k$ of the solution $\{ \hat{d}_k \}$ can be chosen such
that $\| \mathbf{u}_k \|_1 < 1$ for all $k \in S^c$.  

We now prove that if $\{ \hat{d}'_k \}_{k = 1,..., p}$ is another
solution, then it must be that $\hat{d}'_k = 0$ for all $k \in S^c$ as
well.  We first claim that $\sum_{k=1}^p \bar{\Delta}_k \hat{d}_k =
\sum_{k=1}^p \bar{\Delta}_k \hat{d}'_k$. If this were not true, then a
convex combination of $\hat{d}_k, \hat{d}'_k$ would achieve a strictly
lower objective on the quadratic term. More precisely, let $\zeta \in
[0,1]$. If $\sum_{k=1}^p \bar{\Delta}_k \hat{d}'_k \neq \sum_{k=1}^p
\bar{\Delta}_k \hat{d}_k$, then $\| Y - \sum_{k=1}^p \bar{\Delta}_k
\big( \hat{d}_k + \zeta ( \hat{d}'_k - \hat{d}_k) \big) \|_2^2$ is
strongly convex as a function of $\nu$. Thus, it cannot be that
$\hat{d}_k$ and $\hat{d}'_k$ both achieve optimal objective, and we
have reached a contradiction.

Now, we look at the stationarity condition for both $\{ \hat{d}_k \}$
and $\{ \hat{d}'_k \}$. Let $\mathbf{u}_k \in \partial \|
\bar{\Delta}_k \hat{d}_k \|_\infty$ and let $\mathbf{u}'_k
\in \partial \| \bar{\Delta}_k \hat{d}'_k \|_\infty$ be the two sets
of subgradients. Let $\{ \nu_{ki} \}$ and $\{
\nu'_{ki} \}$ be the two sets of positivity dual
variables, for $k=1,..,p$ and $i=1,...n-1$.  Note that since there is no positivity constraint on
$d_{k1}$, we let $\nu_{k1} = 0$ always.

Let us define $\bar{\Delta}$, a $n \times p(n-1)$ matrix, to denote the column-wise concatenation of $\{ \bar{\Delta}_k \}_k$ and $\hat{d}$, a $p(n-1)$ dimensional vector, to denote the concatenation of $\{ \hat{d}_k \}_k$. With this notation, we can express $\sum_{k=1}^p \bar{\Delta}_k \hat{d}_k = \bar{\Delta} \hat{d}$.

Since both solutions $(\hat{d}, \mathbf{u}, \nu)$ and $(\hat{d}',
\mathbf{u}', \nu')$ must satisfy the stationarity condition, we have
that
\[
\bar{\Delta}^\tran ( Y - \bar{\Delta} \hat{d} ) 
   + \lambda \sum_{k=1}^p \bar{\Delta}_k^\tran \mathbf{u}_k - \nu = 
\bar{\Delta}^\tran ( Y - \bar{\Delta} \hat{d}' ) 
   + \lambda \sum_{k=1}^p \bar{\Delta}_k^\tran \mathbf{u}'_k - \nu' = 0.
\] 
Multiplying both sides of the above equation by $\hat{d}'$,
\[
\hat{d}'^{\tran}  \bar{\Delta}^\tran ( Y - \bar{\Delta} \hat{d} ) 
    + \lambda \sum_{k=1}^p \hat{d}'^\tran_k \bar{\Delta}_k^\tran \mathbf{u}_k - \hat{d}'^\tran \nu = \hat{d}'^{\tran}  \bar{\Delta}^\tran ( Y - \bar{\Delta} \hat{d}' ) 
    + \lambda \sum_{k=1}^p \hat{d}'^\tran_k \bar{\Delta}_k^\tran \mathbf{u}'_k - \hat{d}'^\tran \nu'.
\]
Since $\bar{\Delta} \hat{d}_k = \bar{\Delta} \hat{d}$, $\hat{d}'^\tran \nu' = 0$ (complementary slackness), and $\hat{d}'^\tran_k \bar{\Delta}_k^\tran \mathbf{u}'_k  = \| \hat{f}'_k \|_\infty$ (where $\hat{f}'_k = \bar{\Delta}_k \hat{d}'_k$), we have that
\[
\lambda \sum_{k=1}^p \hat{d}'^\tran_k \bar{\Delta}_k^\tran \mathbf{u}_k - \hat{d}'^\tran \nu = \lambda \sum_{k=1}^p \| \hat{f}'_k \|_\infty.
\]
On one hand, $\hat{d}'$ is a feasible solution so $\hat{d}'^\tran \nu \geq 0$ and so 
\[
\sum_{k=1}^p \hat{d}'^\tran_k \bar{\Delta}_k^\tran \mathbf{u}_k \geq \sum_{k=1}^p \| \hat{f}'_k \|_\infty .
\]
On the other hand, by H\"older's inequality,
\begin{align*}
\sum_{k=1}^p \hat{d}'^\tran_k \bar{\Delta}_k^\tran \mathbf{u}_k &\leq 
   \sum_{k=1}^p \| \hat{f}'_k \|_\infty \|\mathbf{u}_k \|_1 .
\end{align*}
Since $\mathbf{u}_k$ can be chosen so that $\| \mathbf{u}_k \|_1 < 1$ for all $k \in S^c$, we would get a contradiction if $\| \hat{f}'_k \|_\infty > 0$ for some $k \in S^c$. We thus conclude that $\hat{d}'$ must follow the same sparsity pattern.

The second item in the theorem concerning optimization~\ref{opt:alternate_opt_concave} is proven in exactly the same way. 
The Lagrangian of optimization~\ref{opt:alternate_opt_concave} is
\[
\mathcal{L}_{\trm{cave}}(d_k, \nu_k) = 
  \frac{1}{2n} \big\| \hat{r} - \bar{\Delta}_k d_k \big \|_2^2 + 
  \lambda \| \bar{\Delta}_k d_k \|_\infty + \sum_{k=1}^p \sum_{i=2}^{n-1} \nu_{ki} d_{ki}.
\]
with $\nu_{ki} \geq 0$.
The same reasoning applies to show that $\hat{d}_k = 0$ satisfies KKT conditions sufficient for optimality.
\end{proof}

 \subsection{Proof of False Positive Control}
 \label{sec:false_positive_proof}
 
 We note that in the following analysis the symbols $c,C$ represent
 absolute constants. We will often abuse notation and ``absorb'' new
 absolute constants into $c, C$; the actual value of $c, C$ could thus
 vary from line to line.
We first restate the theorem for convenience. 

\begin{theorem} 
Suppose assumptions A1-A5 hold. Define $\tilde{\sigma} \equiv \max(\sigma, B)$. Suppose that $p \leq O\big( \exp( c n) \big)$ and $n \geq C$ for some constants $0<c<1$ and $C$. Define $\mathsf{range}_k = X_{k\pi_k(n)} - X_{k\pi_k(1)}$.

If $\lambda_n \geq 2 (8 \cdot 32) s \tilde{\sigma}  \sqrt{ \frac{1}{n} \log^2 np}$ then, with probability at least $ 1 - \frac{24}{n}$, for all $k \in S^c$, and for all $i'=1,...,n$
\begin{align*}
& \lambda_n > \mathsf{range}_k \Big| \frac{32}{n}\hat{r}^\tran \mathbf{1}_{\pi_k(i':n)} \Big|  
\end{align*}
and $\max_{i'} \frac{X_{k\pi_k(i'+1)} - X_{k \pi_k(i')}}{\mathsf{range}_k} \leq \frac{1}{16}$ and $\mathsf{range}_k \geq 1$.

Therefore, for all $k \in S^c$, both the AC solution $\hat{f}_k$ from optimization~\ref{opt:alternate_opt}, and the DC solution $\hat{g}_k$ from optimization~\ref{opt:alternate_opt_concave} are zero. 
\end{theorem}

\begin{proof}
The key is to note that $\hat{r}$ and $\Delta_{k,j}$ are independent for all $k \in S^c,j=1,...,n$ because $\hat{r}$ is only dependent on $X_{S}$.

Fix $j$ and $i$. Then $\hat{r}^\tran \mathbf{1}_{\pi_k(i':n)}$ is the sum
of $n-i'+1$ random coordinates of $\hat{r}$. We will use
Serfling's theorem on the concentration of measure of sampling without
replacement (Corollary~\ref{cor:serfling}). We must first bound $\|
\hat{r} \|_\infty$ and $\frac{1}{n} \sum_{i=1}^n \hat{r}_i$ before we
can use Serfling's results however.

\vskip5pt
\textbf{Step 1}: {\it Bounding $\| \hat{r} \|_\infty$.} We have $\hat{r}_i = f_0(x_i) + w_i - \hat{f}(x_i)$ where
$\hat{f}(x_i) = \sum_{k \in S} \bar{\Delta}_k \hat{d}_k$ is the convex
additive function outputted by the restricted regression. Note that
both $f_0(x_i)$ and $\hat{f}(x_i)$ are bounded by $2sB$. 
Because $w_i$ is sub-Gaussian, $|w_i| \leq  \sigma \sqrt{2\log \frac{2}{\delta}}$ with probability at least $1-\delta$. By union bound across $i=1,...,n$, we have that $\| w\|_\infty \leq \sigma \sqrt{ 2 \log \frac{2n}{\delta}}$ with probability at least $1 - \delta$.

Putting these observations together,
\begin{align}
\| \hat{r} \|_\infty &\leq 2sB + \sigma \sqrt{ 2\log \frac{2n}{\delta}}) \nonumber \\
      &\leq 4 s \tilde{\sigma} \sqrt{\log \frac{2n}{\delta}} \label{eqn:stepone_rhat}
\end{align}
with probability at least $1 - \delta$, where we have defined
$\tilde{\sigma} = \max(\sigma, B)$,
and assumed that $\sqrt{\log \frac{2np}{\delta}} \geq 1$. This assumption holds under the conditions in the theorem which state that $p \leq \exp( c n)$ and $n \geq C$ for some small constant $c$ and large constant $C$.

\vskip5pt
\textbf{Step 2}: {\it Bounding $| \frac{1}{n} \hat{r}^\tran \mathbf{1}
  |$.}  We have that 
\begin{align*}
\frac{1}{n} \hat{r}^\tran \mathbf{1} &= 
    \frac{1}{n} \sum_{i=1}^n f_0(x_i) + w_i - \hat{f}(x_i) \\
  &= \frac{1}{n} \sum_{i=1}^n f_0(x_i) + w_i \quad \trm{ ($\hat{f}$ is centered)}.
\end{align*}
Since $|f_0(x_i)| \leq sB$, the first term $| \frac{1}{n} \sum_{i=1}^n
f_0(x_i)|$ is at most $sB \sqrt{\frac{2}{n} \log \frac{2}{\delta}}$
with probability at most $1-\delta$ by Hoeffding's inequality. Since
$w_i$ is sub-Gaussian, the second term $|\frac{1}{n} \sum_{i=1}^n
w_i|$ is at most $\sigma \sqrt{ \frac{2}{n} \log \frac{2}{\delta}}$
with probability at most $1-\delta$.  
Taking a union bound, we have that 
\begin{align}
| \frac{1}{n} \hat{r}^\tran \mathbf{1}| &\leq sB \sqrt{\frac{2}{n} \log \frac{4}{\delta}} +  \sigma \sqrt{\frac{2}{n} \log \frac{4}{\delta}} \nonumber \\
  &\leq 4 s \tilde{\sigma} \sqrt{\frac{1}{n} \log \frac{4}{\delta}} \label{eqn:steptwo_rhat}
\end{align}
with probability at least $1-\delta$.

\vskip5pt
\textbf{Step 3}: {\it Apply Serfling's theorem.}  
For any $k \in S^c$, Serfling's theorem states that with probability at least $1 - \delta$
\begin{align*}
\Big
|\frac{1}{n} \hat{r}^\tran \mathbf{1}_{\pi_k(i':n)}\Big| \leq
   2\| \hat{r} \|_\infty \sqrt{ \frac{1}{n} \log \frac{2}{\delta}} + 
   \Big|\frac{1}{n} \hat{r}^\tran \mathbf{1} \Big|
\end{align*}
We need Serfling's theorem to hold for all $k = 1,...,p$ and $i' =
1,...,n$. We also need the events that $\|\hat{r}\|_\infty$ and $|
\frac{1}{n} \hat{r}^\tran \mathbf{1}|$ are small to hold. Using a
union bound, with probability at least $1-\delta$, for all $k,i'$,
\begin{align*}
\Big
|\frac{1}{n} \hat{r}^\tran \mathbf{1}_{\pi_k(i':n)}\Big| &\leq
   2\| \hat{r} \|_\infty \sqrt{ \frac{1}{n} \log \frac{6np}{\delta}} + 
   \Big|\frac{1}{n} \hat{r}^\tran \mathbf{1} \Big|\\
  &\leq 4s\tilde{\sigma}\sqrt{\log \frac{6n}{\delta}} \sqrt{\frac{1}{n}\log \frac{6np}{\delta}} + 4 s \tilde{\sigma} \sqrt{\frac{1}{n}\log \frac{12}{\delta}} \\
  &\leq 8 s\tilde{\sigma} \sqrt{ \frac{1}{n} \log^2 \frac{12np}{\delta}}
\end{align*}
In the second inequality, we used equation~\eqref{eqn:stepone_rhat}
and equation~\eqref{eqn:steptwo_rhat} from steps 1 and 2
respectively. Setting $\delta = \frac{12}{n}$ gives the desired expression.

Finally, we note that $2 \geq (X_{k\pi_k(n)} - X_{k\pi_k(1)})$ since $X_k \subset [-1,1]$. This concludes the proof for the first part of the theorem. 

To prove the second and the third claims, let the interval $[-1, 1]$ be divided into $64$ non-overlapping segments each of length $1/32$. Because $X_k$ is drawn from a density with a lower bound $c_l > 0$, the probability that every segment contains some samples $X_{ki}$'s is at least $1-64 \left( 1 - \frac{1}{32} c_l \right)^n$. Let $\mathcal{E}_k$ denote the event that every segment contains some samples. 

Define $\mathsf{gap}_i = X_{k \pi_k(i+1)} - X_{k \pi_k(i)}$ for $i=1,...,n-1$ and define $\mathsf{gap}_0 = X_{k \pi_k(1)} - (-1)$ and $\mathsf{gap}_{n} = 1 - X_{k\pi_k(n)}$. 

If any $\mathsf{gap}_i \geq \frac{1}{16}$, then $\mathsf{gap}_i$ has to contain one of the segments. Therefore, under event $\mathcal{E}_k$, it must be that $\mathsf{gap}_i \leq \frac{1}{16}$ for all $i$.

Thus, we have that $\mathsf{range}_k \geq 2 - 1/8 \geq 1$ and that for all $i$,

\[
\frac{X_{k\pi_k(i+1)} - X_{k \pi_k(i)}}{\mathsf{range}_k} \geq 
\frac{ 1/16 }{ 2 - 1/8} \geq 1/16
\]

Taking an union bound for each $k \in S^c$, the probability of that all $\mathcal{E}_k$ hold is at least $1 - p 64 \left( 1 - \frac{1}{32} c_l \right)^n$.

$p64 \left( 1 - \frac{1}{32} c_l \right)^n = 64 p \exp( - c' n)$ for some positive constants $0< c' < 1$ dependent on $c_l$. Therefore, if $p \leq exp( c n)$ for some $0<c < c'$ and if $n$ is larger than some constant $C$, $64 p \exp( - c' n) \leq 64 \exp( - (c' - c'') n) \leq \frac{12}{n}$.

Taking an union bound with the event that $\lambda_n$ upper bounds the partial sums of $\hat{r}$ and we establish the claim. 

\end{proof}

 
 \subsection{Proof of False Negative Control}
 \label{sec:false_negative_proof}
 We begin by introducing some notation.
 \subsubsection{Notation} 
\label{sec:false_negative_proof_notations}
If $f : \mathbb{R}^s \rightarrow \R$, we define $\| f \|_P \equiv \E f(X)^2$. 
Given samples $X_1,...,X_n$, we denote $\| f \|_n \equiv \frac{1}{n} \sum_{i=1}^n f(X_i)^2$ and $\langle f, g \rangle_n \equiv \frac{1}{n} \sum_{i=1}^n f(X_i) g(X_i)$. 

Let $\mathcal{C}^1$ denote the set of univariate convex functions supported on $[-1,1]$. Let $\mathcal{C}^1_B \equiv \{ f \in \mathcal{C}^1 \,:\, \| f \|_\infty \leq B \}$ denote the set of $B$-bounded univariate convex functions. 
Define $\mathcal{C}^s$ as the set of convex additive functions and
$\mathcal{C}^s_B$ likewise as the set of convex additive functions
whose components are $B$-bounded:
\begin{align*}
\mathcal{C}^s &\equiv \{ f \,:\, f = \sum_{k=1}^s f_k, \,
   f_k \in \mathcal{C}^1 \} \\
\mathcal{C}^s_B &\equiv \{ f \in \mathcal{C}^s \,:\, 
f = \sum_{k=1}^s f_k, \, \| f_k \|_\infty \leq B \}.
\end{align*}
Let $f^*(x) = \sum_{k=1}^s f^*_k(x_k)$ be the population risk minimizer:
\[
f^* = \arg\min_{f \in \mathcal{C}^s} \| f_0 - f^* \|_P^2
\]
We let $sB$ be an upper bound on $\| f_0 \|_\infty$ and $B$ be an
upper bound on $\| f^*_k \|_\infty$. 
It follows that $\|f^* \|_\infty \leq s B$.

We define $\hat{f}$ as the empirical risk minimizer:
\[
\hat{f} = \arg\min \Big \{ \| y - f \|_n^2 + \lambda \sum_{k=1}^s \| f_k \|_\infty 
    \,:\, f \in \mathcal{C}^s_B,\, \mathbf{1}_n^\tran f_k = 0 \Big \}
\]
For $k \in \{1,...,s\}$, define $g^*_k$ to be the decoupled concave population risk minimizer
\[
g^*_k \equiv \argmin_{g_k \in \mh \mathcal{C}^1} \| f_0 - f^* - g_k \|_P^2 .
\]
In our proof, we will analyze $g^*_k$ for each $k$ such that $f^*_k = 0$. Likewise, we define the empirical version:
\[
\hat{g}_k \equiv \argmin \Big\{ \| f_0 - \hat{f} - g_k \|_n^2 \,:\, g_k \in \mh \mathcal{C}^1_B \,, \mathbf{1}_n^\tran g_k = 0 \Big\}.
\]
By the definition of the AC/DC procedure, $\hat{g}_k$ is defined only
for an index $k$ that has zero as the convex additive approximation.

\subsubsection{Proof}
 
By additive faithfulness of the AC/DC procedure, it is known that $f^*_k \neq 0$ or $g^*_k \neq 0$ for all $k \in S$. 
Our argument will be to show that the risk of the AC/DC estimators $\hat{f}, \hat{g}$ tends to the risk of the population optimal functions $f^*, g^*$:
\begin{align*}
\| f_0 - \hat{f} \|^2_P & = \| f_0 - f^* \|^2_P + \trm{err}_+(n) \\
\| f_0 - f^* - \hat{g}_k \|^2_P &=  \| f_0 - f^* - g^*_k \|^2_P + \trm{err}_-(n) 
       \quad \trm{for all $k \in S$ where $f^*_k = 0$},
\end{align*}
where the estimation errors $\trm{err}_+(n)$ and $\trm{err}_-(n)$ decrease with $n$ at some rate. 

Assuming this, suppose that $\hat{f}_k = 0$ and $f^*_k \neq 0$. Then when $n$ is large
enough such that $\trm{err}_+(n)$ and $\trm{err}_-(n)$ are smaller
than $\alpha_+$ and $\alpha_-$ defined in
equation~\eqref{eqn:signal_level_defn}, we reach a contradiction.
This is because the risk $\| f_0 - f^* \|_P$ of $f^*$ is
strictly larger by $\alpha_+$ than the risk of the best approximation
whose $k$-th component is constrained to be zero.
Similarly, suppose $f^*_k = 0$ and $g^*_k \neq 0$. Then when $n$ is large
enough, $\hat{g}_k$ must not be zero.

Theorem~\ref{thm:convex_consistent} and
Theorem~\ref{thm:concave_consistent} characterize $\trm{err}_+(n)$ and
$\trm{err}_-(n)$ respectively.

\begin{theorem}
\label{thm:convex_consistent}
Let $\tilde{\sigma} \equiv \max(\sigma, B)$, and let $\hat{f}$ be the
minimizer of the restricted regression with $\lambda \leq 512 s
\tilde{\sigma} \sqrt{ \frac{1}{n} \log^2 np}$.
Suppose $n \geq c_1 s \sqrt{sB}$.
Then with probability at least $1-\frac{C}{n}$,
\begin{align}
\|f_0 - \hat{f} \|_P^2 - \| f_0 - f^* \|_P^2 
&\leq c B^2 \tilde{\sigma} \sqrt{ \frac{s^5}{n^{4/5}} \log^2 np},
\end{align}
where $c_1$ is an absolute constant and $c, C$ are constants possibly dependent on $b$.
\end{theorem}

\begin{proof}
Our proof proceeds in three steps.  First, we bound
the difference of empirical risks $\|f_0 - \hat{f} \|_n^2 - \| f_0 -
f^* \|_n^2$.  Second, we bound the cross-term in the bound using
a bracketing entropy argument for convex function classes.  Finally, 
we combine the previous two steps to complete the argument.

\textbf{Step 1.} The function $\hat{f}$ minimizes the penalized
empricial risk by definition. We would thus like to say that the
penalized empirical risk of $\hat{f}$ is no larger than that of
$f^*$. We cannot do a direct comparison, however, because the empirical mean
$\frac{1}{n} \sum_i f^*_k(x_{ik})$ is close to, but not exactly
zero. We thus have to work first with the function $f^* - \bar{f}^*$.
We have that
\begin{align*}
\| y - \hat{f} \|_n^2 + \lambda \sum_{k=1}^s \| \hat{f}_k \|_\infty &\leq
  \| y - f^* + \bar{f}^* \|_n^2 + \lambda \sum_{k=1}^s \| f^*_k - \bar{f}^*_k \|_\infty 
\end{align*}
Plugging in $y = f_0 + w$, we obtain
\begin{align*}
\| f_0 + w - \hat{f} \|_n^2 + \lambda \sum_{k=1}^s \Big( \| \hat{f}_k \|_\infty - 
    \| f^*_k - \bar{f}^*_k \|_\infty \Big) &\leq \|f_0 + w - f^* + \bar{f}^* \|_n^2 \\
\| f_0 - \hat{f} \|_n^2 + 2\langle w, f_0 - \hat{f} \rangle_n 
     +\lambda \sum_{k=1}^s \Big( \| \hat{f}_k \|_\infty - \|f^*_k -\bar{f}^*_k\|_\infty \Big) 
    &\\
\leq \| f_0 - f^* + \bar{f}^* \|_n^2 + 
    2 \langle w, f_0 - f^* + \bar{f}^* \rangle \\
\|f_0 - \hat{f} \|_n^2 - \| f_0 - f^* + \bar{f}^* \|_n^2 + 
    \lambda \sum_{k=1}^s \Big( \| \hat{f}_k \|_\infty - 
 \| f^*_k - \bar{f}^*_k \|_\infty \Big) &\leq 2 \langle w, \hat{f} - f^* + \bar{f}^* \rangle.
\end{align*}
The middle term can be bounded under the assumption that $\|f^*_k -
\bar{f}^*_k \|_\infty \leq 2B$; thus,
\begin{align*}
\|f_0 - \hat{f} \|_n^2 - \| f_0 - f^* + \bar{f}^* \|_n^2 
   &\leq 2 \langle w, \hat{f} - f^* + \bar{f}^* \rangle + \lambda 2 s B .
\end{align*}
Using Lemma~\ref{lem:remove_centering}, we can remove $\bar{f}^*$ from
the lefthand side. Thus with probability at least $1 - \delta$,
\begin{align}
\label{eqn:first_step_inequality}
\|f_0 - \hat{f} \|_n^2 - \| f_0 - f^* \|_n^2 
   &\leq 2 \langle w, \hat{f} - f^* + \bar{f}^* \rangle + \lambda 2 s B + c(sB)^2 \frac{1}{n} \log \frac{2}{\delta}.
\end{align}



\textbf{Step 2.} We now upper bound the cross term $2 \langle w,
\hat{f} - f^* + \bar{f}^* \rangle$ using bracketing entropy.

Define $\mathcal{G} =\{ f - f^* + \bar{f}^* \,:\, f \in
\mathcal{C}^s_B \}$ 
as the set of convex additive functions centered around the function $f^* - \bar{f}^*$. 
By Corollary~\ref{prop:convexbracket_lp}, there is an $\epsilon$-bracketing of $\mathcal{G}$ whose size is bounded by $\log N_{[]}( 2\epsilon, \mathcal{G}, L_1(P)) \leq sK^{**} \left( \frac{2sB}{\epsilon} \right)^{1/2}$, for all $\epsilon \in (0, sB \epsilon_3]$.
Let us suppose condition~\ref{cond:simplify_covering_number} holds. Then, by Corollary~\ref{cor:convexbracket_ln}, with probability at least $1-\delta$, each bracketing pair $(h_U, h_L)$ is close in $L_1(P_n)$ norm, i.e., for all $(h_U, h_L)$, 
$\frac{1}{n} \sum_{i=1}^n | h_U(X_i) - h_L(X_i) | \leq 2 \epsilon + sB \sqrt{ \frac{sK^{**}(2sB
)^{1/2} \log \frac{1}{\delta}}{2\epsilon^{1/2} n}}$. We verify at the
end of the proof that 
condition~\ref{cond:simplify_covering_number} indeed holds.

For each $h \in \mathcal{G}$, there exists a pair $(h_U, h_L)$ such that $h_U(X_i) - h_L(X_i) \geq h(X_i) - h_L(X_i) \geq 0$. Therefore, with probability at least $1-\delta$, uniformly for all $h \in \mathcal{G}$:
$$
\frac{1}{n} \sum_{i=1}^n |h(X_i) - h_L(X_i)| \leq \frac{1}{n} \sum_{i=1}^n | h_U(X_i) - h_L(X_i)| \leq 2\epsilon +  (sB) \sqrt{ \frac{sK^{**}(2sB)^{1/2} \log \frac{1}{\delta}}{2\epsilon^{1/2} n}}.
$$
We denote $\epsilon_{n,\delta} \equiv (sB) \sqrt{
  \frac{sK^{**}(2sB)^{1/2} \log \frac{1}{\delta}}{2\epsilon^{1/2}
    n}}$. Let $\mathcal{E}_{[\,]}$ denote the event that for each $h
\in \mathcal{G}$, there exists $h_L$ in the $\epsilon$-bracketing such
that $\|h-h_L\|_{L_{P_n}} \leq 2\epsilon + \epsilon_{n, \delta}$. Then
$\mathcal{E}_{[\,]}$ has probability at most $1-\delta$ as shown.

Let $\mathcal{E}_{\|w\|_\infty}$ denote the event that $\| w \|_\infty
\leq \sigma \sqrt{ 2\log \frac{2n}{\delta}}$.  Then
$\mathcal{E}_{\|w\|_\infty}$ has probability at most $1-\delta$. We
now take an union bound over $\mathcal{E}_{\|w\|_\infty}$ and
$\mathcal{E}_{[\,]}$ and get that, with probability at most
$1-2\delta$, for all $h$
\[
|\langle w, h - h_L\rangle_n| \leq \| w \|_\infty \frac{1}{n} \sum_{i=1}^n |h(X_i) - h_L(X_i)| \leq
  \sigma \sqrt{2 \log \frac{4n}{\delta}} \left( \epsilon + \epsilon_{n,2\delta} \right).
\]
Because $w$ is a sub-Gaussian random variable, we have that for any fixed vector
$h_L = (h_L(X_1),...,h_L(X_n))$, with probability at least $1-\delta$,
$|\langle w, h_L \rangle_n | \leq \| h_L \|_n \sigma \sqrt{
  \frac{2}{n} \log \frac{2}{\delta} }$. Using another union bound, we
have that the event $\sup_{h_L} |\langle w, h_L \rangle| \leq sB
\sigma \sqrt{ \frac{2}{n}\log \frac{2 N_{[]}}{\delta}}$ has
probability at most $1-\delta$.

Putting this together, we have that
\begin{align*}
\lefteqn{|\langle w, h \rangle_n | \leq | \langle w, h_L\rangle_n| + |\langle w, h - h_L\rangle_n|}\\
\lefteqn{|\sup_{h \in \mathcal{G}} \langle w, h \rangle_n| \leq 
     | \sup_{h^L} \langle w, h_L \rangle_n | + \sigma \sqrt{2 \log \frac{2n}{\delta}} (2\epsilon + \epsilon_{n, 2\delta})} \\
   &\leq   sB \sigma \sqrt{ 2 \frac{ \log N_{[]} + \log \frac{1}{\delta}}{n}} + \sigma \sqrt{2 \log \frac{2n}{\delta}} (2\epsilon + \epsilon_{n, \delta}) \\
   &\leq  sB \sigma \sqrt{ 2 \frac{sK^{**} (2sB)^{1/2} \log \frac{1}{\delta}}{n \epsilon^{1/2}}} +
   \sigma \sqrt{ 2\log \frac{2n}{\delta}} (2\epsilon + \epsilon_{n, \delta}) \\
   &\leq sB \sigma \sqrt{ 2\frac{sK^{**} (2sB)^{1/2} \log \frac{1}{\delta}}{n \epsilon^{1/2}}} +
   2\sigma\sqrt{2 \log \frac{2n}{\delta}} \epsilon + sB \sigma \sqrt{2 \frac{sK^{**} (2sB)^{1/2}\log \frac{1}{\delta}}{n \epsilon^{1/2}} \log \frac{2n}{\delta}} \\
   &\leq 2\sigma\sqrt{2\log \frac{2n}{\delta}} \epsilon + 2 sB \sigma \sqrt{ \frac{sK^{**} (2sB)^{1/2} \log^2 \frac{2n}{\delta}}{n \epsilon^{1/2}}}.
\end{align*}
To balance the two terms, we set $\epsilon = sB \sqrt{ \frac{(s
    K^{**} (sB)^{1/2})^{4/5}}{n^{4/5}} }$. It is easy to verify that
if $n \geq c_1 s \sqrt{sB}$ for some absolute constant $c_1$, then 
$\epsilon \in (0, sB \epsilon_3]$ for some absolute constant $\epsilon_3$ as required by the bracketing number results (Corollary~\ref{cor:convexadditive_lp}). Furthermore, conditions~\eqref{cond:simplify_covering_number} also hold.

In summary, we have that probability at least $1-\delta$,
\[
|\sup_{h \in \mathcal{G}} \langle w, h \rangle | \leq c sB \sigma \sqrt{ 
   \frac{s^{6/5} B^{2/5} \log^2 \frac{Cn}{\delta}}{n^{4/5}}} \leq 
  c sB \sigma \sqrt{ 
   \frac{s (sB)^{1/2} \log^2 \frac{Cn}{\delta}}{n^{4/5}}}
\]
where we absorbed $K^{**}$ into the constant $c$ and the union bound multipliers into the constant $C$.


Plugging this result into equation~\eqref{eqn:first_step_inequality}
we get that, with probability at least $1 - 2\delta$,
\begin{align}
\|f_0 - \hat{f} \|_n^2 - \| f_0 - f^* \|_n^2 
   &\leq c sB \sigma \sqrt{ 
   \frac{s (sB)^{1/2} \log^2 \frac{Cn}{\delta}}{n^{4/5}}}
   + \lambda 2 s B + c (sB)^2 \frac{1}{n} \log \frac{2}{\delta} \nonumber\\
\|f_0 - \hat{f} \|_n^2 - \| f_0 - f^* \|_n^2 
   &\leq c sB \sigma \sqrt{ 
   \frac{s (sB)^{1/2} \log^2 \frac{Cn}{\delta}}{n^{4/5}}}
   + \lambda 2 s B \nonumber\\   
   &\leq c B \sigma 
    \sqrt{ \frac{s^4 B^{1/2}}{n^{4/5}} \log^2 \frac{Cn}{\delta}} + \lambda 2 sB
\label{eqn:second_step_inequality}
\end{align}


\textbf{Step 3.} Continuing from
equation~\eqref{eqn:second_step_inequality}, we use
Lemma~\ref{lem:uniform_convergence} and another union bound to obtain
that, with probability at least $1-3\delta$,
\begin{align}
\|f_0 - \hat{f} \|_P^2 - \| f_0 - f^* \|_P^2 
   &\leq cB^2 \sigma 
    \sqrt{ \frac{s^4}{n^{4/5}} \log^2 \frac{Cn}{\delta}}
 +\lambda 2 s B + c B^3 \sqrt{ \frac{s^5}{n^{4/5}} \log \frac{2}{\delta}}
    \nonumber \\
&\leq c B^2 \tilde{\sigma} \sqrt{ \frac{s^5}{n^{4/5}} \log^2 \frac{Cn}{\delta}} + \lambda 2 sB \nonumber
\end{align}
Substituting in $\lambda \leq 512 s \tilde{\sigma} \sqrt{\frac{1}{n}
  \log^2 np}$ and $\delta = \frac{C}{n}$ we obtain the statement of
the theorem.
\end{proof}
 

\begin{theorem}
\label{thm:concave_consistent}
Let $\hat{g}_k$ denote the minimizer of the concave postprocessing
step with $\lambda_n \leq 512 s\tilde{\sigma} \sqrt{\frac{1}{n} \log^2 np}$. Let $\tilde{\sigma} \equiv \max(\sigma, B)$.
Suppose $n$ is sufficiently large that $\frac{n^{4/5}}{\log^2 np} \geq c' B^4 \tilde{\sigma}^2 s^5$ where $c' \geq 1$ is a constant.
Then with probability at least $1- \frac{C}{n}$, for all $k=1,...,s$,
\[
\| f_0 - f^* - \hat{g}_k \|_P^2 - \| f_0 - f^* - g^*_k \|_P^2 \leq  c B^2 \tilde{\sigma}^{1/2} \sqrt[4]{ \frac{s^5}{n^{4/5}} \log^2 np}.
\]
\end{theorem}

\begin{proof}
This proof is similar to that of Theorem \ref{thm:convex_consistent}; it requires a few more steps because $\hat{g}_k$ is fitted against $f_0 - \hat{f}$ instead of $f_0 - f^*$. We start with the following decomposition:
\begin{align}
\| f_0 - f^* - \hat{g}_k \|_P^2 - \| f_0 - f^* - g^*_k \|_P^2 = & \underbrace{\| f_0 - \hat{f} - \hat{g}_k \|_P^2 - \| f_0 - \hat{f} - g^*_k \|_P^2}_{\trm{term 1}} + \nonumber \\
   & \underbrace{\| f_0 - f^* - \hat{g}_k \|_P^2 - \| f_0 - \hat{f} - \hat{g}_k \|_P^2}_{\trm{term 2}} + \nonumber \\
   & \underbrace{\| f_0 - \hat{f} - g^*_k \|_P^2 - \| f_0 - f^* - g^*_k \|_P^2}_{\trm{term 3}}  \label{eqn:concave_init_decomposition}.
\end{align}
We now bound each of the terms. The proof proceeds almost identically to that of Theorem~\ref{thm:convex_consistent}, because convex and concave functions have the same bracketing number.

\textbf{Step 1.} To bound term 1, we start from the definition of
$\hat{g}_k$ and obtain
\begin{align*}
\| y - \hat{f} - \hat{g}_k \|_n^2 + \lambda_n \| \hat{g} \|_\infty &\leq
   \| y - \hat{f} - g^*_k \|_n^2 + \lambda_n \| g^* \|_\infty \\
\| y - \hat{f} - \hat{g}_k \|_n^2 &\leq \| y - \hat{f} - g^*_k \|_n^2 + \lambda_n 2B \\[10pt]
\| f_0 - \hat{f} - \hat{g}_k + w\|_n^2 & \leq \| f_0 - \hat{f} - g^*_k + w \|_n^2 
   +\lambda_n 2 B \\
\| f_0 - \hat{f} - \hat{g}_k \|_n^2 - \|f_0 -\hat{f} - g^*_k\|_n^2 &\leq
   2 \langle w, \hat{g}_k - g^*_k \rangle_n + \lambda_n 2B.
\end{align*}
Using the same bracketing analysis as in Step 2 of the proof of Theorem~\ref{thm:convex_consistent} but setting $s=1$, we have, with probability at least $1-\delta$,
\begin{align*}
\| f_0 - \hat{f} - \hat{g}_k \|_n^2 - \|f_0 - \hat{f} - g^*_k \|_n^2 &\leq
  c B^2 \sigma \sqrt{ \frac{1}{n^{4/5}} \log \frac{C}{\delta} }+ \lambda_n 2 B.
\end{align*}
The condition $n \geq c_1 s\sqrt{sB}$ in the proof of Theorem~\ref{thm:convex_consistent} is satisfied here because we assume that $n^{4/5} \geq c_1 B^4 \tilde{\sigma}^2 s^5 \log^2 np$ in the statement of the theorem.
Using the uniform convergence result of Lemma~\ref{lem:uniform_convergence}, with probability at least $1-\delta$,
\begin{align*}
\| f_0 - \hat{f} - \hat{g}_k \|_P^2 - \|f_0 - \hat{f} - g^*_k \|_P^2 &\leq
  c B^2 \sigma \sqrt{ \frac{1}{n} \log \frac{Cn}{\delta} }+ \lambda_n 2 B +
  c B^3 \sqrt{\frac{s^5}{n^{4/5}} \log \frac{2}{\delta} } \\
 &\leq c B^2 \tilde{\sigma} \sqrt{\frac{s^5}{n^{4/5}} \log \frac{C}{\delta}}+ \lambda_n 2B
\end{align*}

Finally, plugging in $\lambda_n \leq 9 s \tilde{\sigma} \sqrt{
  \frac{1}{n} \log^2 np}$, we obtain
\begin{align*}
\| f_0 - \hat{f} - \hat{g}_k \|_P^2 - \|f_0 - \hat{f} - g^*_k \|_P^2 &
\leq c B^2 \tilde{\sigma} \sqrt{\frac{s^5}{n^{4/5}} \log \frac{C}{\delta}}+ 
    2s B \tilde{\sigma} \sqrt{\frac{1}{n} \log^2 np}\\
\| f_0 - \hat{f} - \hat{g}_k \|_P^2 - \|f_0 - \hat{f} - g^*_k \|_P^2 &
\leq c B^2 \tilde{\sigma} \sqrt{\frac{s^5}{n^{4/5}} \log^2 \frac{Cnp}{\delta}}
\end{align*}
with probability at least $1-\delta$.

\textbf{Step 2.} We now bound term 3.
\begin{align*}
\| f_0 - \hat{f} - g^*_k \|_P^2 - \| f_0 - f^* - g^*_k\|_P^2 &\leq 
    \| f_0 - \hat{f} \|_P^2 - \|f_0 - f^*\|_P^2 - 2\langle f_0 - \hat{f}, g^*_k \rangle_P
   + 2 \langle f_0 - f^*, g^*_k \rangle_P \\
 &\leq c B^2 \tilde{\sigma} \sqrt{ \frac{s^5}{n^{4/5}} \log^2 np} + 
    2 | \langle \hat{f} - f^*, g^*_k \rangle_P |  \\
 &\leq  c B^2 \tilde{\sigma} \sqrt{ \frac{s^5}{n^{4/5}} \log^2 np} +
    2 \| \hat{f} - f^* \|_P \| g^*_k \|_P \\
&\leq  c B^2 \tilde{\sigma} \sqrt{ \frac{s^5}{n^{4/5}} \log^2 np} +
   c B \sqrt{B^2 \tilde{\sigma} \sqrt{ 
                   \frac{s^5}{n^{4/5}} \log^2 np} }\\
&\leq  cB^2 \tilde{\sigma}^{1/2} \sqrt[4]{ 
                   \frac{s^5}{n^{4/5}} \log^2 np} 
\end{align*}
with probability at least $1-\frac{C}{n}$, by
Theorem~\ref{thm:convex_consistent}. To obtain the fourth inequality,
we used the fact that $\| \hat{f} - f^* \|^2 \leq \| f_0 - \hat{f}
\|_P^2 - \|f_0 - f^*\|_P^2$, which follows from the fact that $f^*$ is the
projection of $f_0$ onto the set of additive convex functions and the
set of additive convex functions is convex itself.
The last inequality holds because there is a condition in the theorem which states $n$
is large enough such that $B^2 \tilde{\sigma} \sqrt{ \frac{s^5}{n^{4/5}} \log^2 np} \leq 1$.
The same derivation and the same bound likewise holds for term 2.

\textbf{Step 3.} Collecting the results and plugging them into equation~\eqref{eqn:concave_init_decomposition}, we have, with probability at least $1-2\delta$:
\begin{align*}
\| f_0 - f^* - \hat{g}_k \|_P^2 - \|f_0 - f^* - g^*_k \|_P^2 \leq
   c B^2 \tilde{\sigma}^{1/2} 
     \sqrt[4]{ \frac{s^5}{n^{4/5}} \log^2 \frac{4np}{\delta}} 
\end{align*}
Taking a union bound across the $s$ dimensions completes the result.
\end{proof}

\subsubsection{Support Lemmas}


\begin{lemma}
\label{lem:uniform_convergence}
Suppose $n \geq c_1 s\sqrt{sB}$ for some absolute constant $c_1$. Then, with probability at least $1-\delta$:
\begin{align*}
\sup_{f \in \mathcal{C}^s_B} \Big| \| f_0 - f \|^2_n - \|f_0 - f \|^2_P\Big| \leq
   c B^3 \sqrt{ \frac{s^5}{n^{4/5}} \log \frac{2}{\delta}}
\end{align*}
where $c_1$ is some absolute constant and $c,C$ are constants possibly dependent on $b$.
\end{lemma}

\begin{proof}
Let $\mathcal{G}$ denote the off-centered set of convex functions, that is, $\mathcal{G} \equiv \mathcal{C}^s - f_0$. Note that if $h \in \mathcal{G}$, then $\| h \|_\infty = \| f_0 - f \|_\infty \leq 4 s B$.
There exists an $\epsilon$-bracketing of $\mathcal{G}$, and
by Corollary~\ref{cor:convexadditive_lp}, the bracketing has size at most $\log N_{[]}(2\epsilon, \mathcal{C}^s, L_1(P)) \leq s K^{**}\left( \frac{2sB}{\epsilon} \right)^{1/2}$. By Corollary~\ref{cor:convexbracket_ln}, we know that with probability at least $1-\delta$, $\|h_U - h_L\|_{L_1(P_n)} \leq \epsilon + \epsilon_{n,\delta}$ for all pairs $(h_U, h_L)$ in the bracketing, where $\epsilon_{n,\delta} = sB \sqrt{ \frac{K^{**} (2sB)^{1/2} \log \frac{2}{\delta}}{2 \epsilon^{1/2} n}}$. Corollary~\ref{cor:convexbracket_ln} necessitates $\epsilon \in (0, sB\epsilon_3]$ for some absolute constant $\epsilon_3$; we will verify that this condition holds for large enough $n$ when we set the actual value of $\epsilon$.
For a particular function $h \in \mathcal{G}$, we can construct $\psi_L \equiv \min( |h_U|, |h_L|)$ and $\psi_U \equiv \max( |h_U|, |h_L| )$ so that
\[
\psi_L^2 \leq h^2 \leq \psi_U^2.
\]
We can then bound the $L_1(P)$ norm of $\psi_U^2 - \psi_L^2$ as
\begin{align*}
\int (\psi_U^2(x) - \psi_L^2(x)) p(x)dx  &\leq  \int | h_U^2(x) - h_L^2(x)| p(x) dx \\
   &\leq \int | h_U(x) - h_L(x) | \, |h_U(x) + h_L(x)| p(x) dx \\
   &\leq 2sB \epsilon
\end{align*}

Now we can bound $\| h \|_n^2 - \| h \|_P^2$ as
\begin{align}
\frac{1}{n} \sum_{i=1}^n \psi_L(X_i)^2 - \E \psi_U(X)^2  \leq
    \| h \|_n^2 - \| h \|_P^2 \leq
  \frac{1}{n} \sum_{i=1}^n \psi_U(X_i)^2 - \E \psi_L(X)^2  \label{eqn:hpsi_bound}
\end{align}
Since 
$\psi_L(X_i)^2$ and $\psi_U(X_i)^2$ are bounded random variables with
upper bound $(sB)^2$, Hoeffding's inequality and union bound give that,
with probability at least $1-\delta$,, for all $\psi_L$ (and likewise $\psi_U$)
\[
\left| \frac{1}{n} \sum_{i=1}^n \psi_L(X_i)^2 - 
   \E \psi_L(X)^2 \right| \leq (sB)^2 \sqrt{ \frac{ sK^{**} (sB)^{1/2} \log \frac{2}{\delta}}{ \epsilon^{1/2} n} }
\]

Plugging this into equation~\eqref{eqn:hpsi_bound} above, we have that:
\begin{align*}
& \E \psi_L(X)^2 - \E \psi_U(X)^2 - 
(sB)^2 \sqrt{ \frac{ sK^{**} (sB)^{1/2} \log \frac{2}{\delta}}{ \epsilon^{1/2} n} } \\
 & \leq 
 \| h \|_n^2 - \| h \|_P^2 \leq
\E \psi_U(X)^2 - \E \psi_L(X)^2 + 
(sB)^2 \sqrt{ \frac{ sK^{**} (sB)^{1/2} \log \frac{2}{\delta}}{ \epsilon^{1/2} n} }.
\end{align*}
Using the $L_1(P)$ norm of $\psi_U^2 - \psi_L^2$ result, we have
\begin{align*}
-sB\epsilon - 
(sB)^2 \sqrt{ \frac{ sK^{**} (sB)^{1/2} \log \frac{2}{\delta}}{ \epsilon^{1/2} n} } \leq 
 \| h \|_n^2 - \| h \|_P^2 \leq
sB\epsilon + 
(sB)^2 \sqrt{ \frac{ sK^{**} (sB)^{1/2} \log \frac{2}{\delta}}{ \epsilon^{1/2} n} }
\end{align*}

We balance the terms by choosing $\epsilon = \left( \frac{ (sB)^2 sK^{**} (sB)^{1/2}}{n} \right)^{2/5}$. One can easily verify that $\epsilon \leq sB \epsilon_3$ condition needed by Corollary~\ref{cor:convexbracket_ln} is satisfied when $n \geq c_1 s \sqrt{sB}$ for some absolute constant $c_1$.
We have then that, with probability at least $1-\delta$,
\begin{align*}
\sup_{h \in \mathcal{G}} \big| \| h \|_n^2 - \| h \|_P^2  \big| \leq
  c B^3 \sqrt{ \frac{s^5 \log \frac{2}{\delta}}{n^{4/5}}}
\end{align*}

The theorem follows immediately.

\end{proof}

\begin{lemma}
\label{lem:remove_centering}
Let $f_0$ and $f^*$ be defined as in Section~\ref{sec:false_negative_proof_notations}. Define $\bar{f}^* = \frac{1}{n} \sum_{i=1}^n f^*(X_i)$.
Then, with probability at least $1 - 2\delta$,
\[
\Big | \| f_0 - f^* \|_n^2 - \| f_0 - f^* + \bar{f}^* \|_n^2 \Big| \leq
    c (sB)^2 \frac{1}{n} \log \frac{4}{\delta}
\]
\end{lemma}

\begin{proof}
We decompose the empirical norm as
\begin{align*}
\| f_0 - f^* + \bar{f}^* \|_n^2 &= \| f_0 - f^* \|_n^2 
    + 2 \langle f_0 - f^*, \bar{f}^* \rangle + \bar{f}^{*2} \\
  &= \| f_0 - f^* \|_n^2 + 2 \bar{f}^* \langle f_0 - f^*, \mathbf{1} \rangle_n + 
    \bar{f}^{*2} \\
  &= \| f_0 - f^* \|_n^2 + 2 \bar{f}^* \bar{f}_0 - \bar{f}^{*2}.
\end{align*}
Now
$\bar{f}^* = \frac{1}{n} \sum_{i=1}^n f^*(X_i)$ is the average of $n$ bounded mean-zero random variables and therefore, with probability at least $1-\delta$, $| \bar{f}^* | \leq 4 sB \sqrt{ \frac{1}{n} \log \frac{2}{\delta} }$.
The same reasoning likewise applies to $\bar{f}_0 = \frac{1}{n} \sum_{i=1}^n f_0(X_i)$.

Taking a union bound and we have that, with probability at least $1- 2\delta$, 
\begin{align*}
| \bar{f}^* | | \bar{f}_0 | &\leq c (sB)^2 \frac{1}{n} \log \frac{2}{\delta} \\
\bar{f}^{*2} &\leq c (sB)^2 \frac{1}{n} \log \frac{2}{\delta}
\end{align*}
Therefore, with probability at least $1 - 2\delta$,
\[
\|f_0 - f^*\|_n^2 - c (sB)^2 \frac{1}{n} \log \frac{2}{\delta} \leq
    \| f_0 - f^* + \bar{f}^* \|_n^2 \leq 
\|f_0 - f^*\|_n^2 + c (sB)^2 \frac{1}{n} \log \frac{2}{\delta}
\]
\end{proof}


 
\subsubsection{Supporting Lemma for Theorem~\ref{thm:acdc_faithful}}

\begin{lemma}
\label{lem:acdc_derivative_bound}
Let $f : [0,1]^p \rightarrow \R$ be a twice differentiable function. Suppose $p(\mathbf{x})$ is a density on $[0,1]^p$ such that $\partial_{x_k} p(\mathbf{x}_{-k} \given x_k)$ and $\partial_{x_k}^2 p(\mathbf{x}_{-k} \given x_k)$ are continuous as functions of $x_k$. Let $\phi(\mathbf{x}_{-k})$ be a continuous function not dependent on $x_k$.

Then, $h^*_k(x_k) \equiv \E[ f(X) - \phi(X_{-k}) \given x_k]$ is twice
differentiable and has a second derivative lower 
bounded away from $-\infty$.
\end{lemma}

\begin{proof}
We can write
\[
h^*_k(x_k) = \int_{\mathbf{x}_{-k}} \big(f(\mathbf{x}) - \phi(\mathbf{x}_{-k})\big) p(\mathbf{x}_{-k} \given x_k) d \mathbf{x}_{-k}
\]
The integrand is bounded because it is a sum-product of continuous functions over a compact set. Therefore, we can differentiate under the integral and derive that
\begin{align*}
\partial_{x_k} h^*_k(x_k) &= 
    \int_{\mathbf{x}_{-k}} f'(\mathbf{x}) p(\mathbf{x}_{-k} \given x_k) + (f(\mathbf{x}) - \phi(\mathbf{x}_{-k}) p'(\mathbf{x}_{-k} \given x_k) d \mathbf{x}_{-k}\\
\partial_{x_k}^2 h^*_k(x_k) &= 
    \int_{\mathbf{x}_{-k}} f''(\mathbf{x}) p(\mathbf{x}_{-k} \given x_k) + 2f'(\mathbf{x}) p'(\mathbf{x}_{-k} \given x_k)  + (f(\mathbf{x}) - \phi(\mathbf{x}_{-k}) p''(\mathbf{x}_{-k} \given x_k) d \mathbf{x}_{-k}
\end{align*}
where we have used the shorthand $f'(\mathbf{x}), p'(\mathbf{x}_{-k}
\given x_k)$ to denote $\partial_{x_k} f(\mathbf{x})$, $\partial_{x_k}
p(\mathbf{x}_{-k} \given x_k)$, etc.

This proves that $h^*_k(x_k)$ is twice-differentiable. To see that the second derivative is lower bounded, we note that $f''(\mathbf{x}) p(\mathbf{x}_{-k} \given x_k)$ is non-negative and the other terms in the second-derivative are all continuous functions on a compact set and thus bounded. 
\end{proof}

\begin{lemma}
\label{lem:additive_uniqueness}
Let $p(\mathbf{x})$ be a positive density over $[0,1]^p$. Let $\phi(\mathbf{x}) = \sum_{j=1}^p \phi_j(x_j)$ be an additive function.

Suppose that for all $j$, $\phi_j(x_j)$ is bounded and $\E \phi_j(X_j) = 0$. Suppose for some $j$, $\E \phi_j(X_j)^2 > 0$, then it must be that $\E \phi(X)^2 > 0$. 
\end{lemma}

\begin{proof}

Suppose, for sake of contradiction, that

\begin{align*}
\E \phi(X)^2 = \E \left( \phi_j(X_j) + \phi_{-j} (X_{-j}) \right)^2 = 0.
\end{align*}

Let $A_+ = \{ x_j \,:\, \phi_j(x_j) \geq 0 \}$. Since $\E \phi_j(X_j) = 0$, $\E \phi_j(X_j)^2 > 0$, and $\phi_j$ is bounded, it must be that both $A_+$ has probability greater than 0.

Now, define $B_+ = \{ x_{-j} \,:\, \phi_{-j}(x_{-j}) \geq 0 \}$. $B_+$ then must have probability greater than 0 as well. 

Since $p(\mathbf{x})$ is a positive density, the set $A_+ \times B_+$ must have a positive probability. However, $\phi > 0$ on $A_+ \times B_+ \subset [0,1]^p$ which implies $\E \phi(X)^2 > 0$.

\end{proof}

\subsubsection{Concentration of Measure}

A sub-exponential random is the square of a sub-Gaussian random
variable \cite{vershynin2010introduction}.

\begin{proposition} (Subexponential Concentration
  \cite{vershynin2010introduction})
Let $X_1,...,X_n$ be zero-mean independent subexponential random
variables with subexponential scale $K$. 
Then
\[
P( | \frac{1}{n} \sum_{i=1}^n X_i | \geq \epsilon) \leq
	2 \exp \left[ -c n \min\left( \frac{\epsilon^2}{K^2}, \frac{\epsilon}{K} \right) \right]
\]
where $c > 0$ is an absolute constant.
\end{proposition}

For uncentered subexponential random variables, we can use the following fact. If $X_i$ subexponential with scale $K$, then $X_i - \E[X_i]$ is also subexponential with scale at most $2K$.
Restating, we can set
\[
c \min\left( \frac{\epsilon^2}{K^2}, \frac{\epsilon}{K} \right) = \frac{1}{n} \log \frac{1}{\delta}.
\]
Thus, with probability at least $1-\delta$, the deviation is at most
\[
K \max\left( \sqrt{\frac{1}{cn} \log \frac{C}{\delta}},  \frac{1}{cn} \log \frac{C}{\delta} \right).
\]

\begin{corollary}
Let $W_1,...,W_n$ be $n$ independent sub-Gaussian random variables with sub-Gaussian scale $\sigma$. 
Then, for all $n > n_0$, with probability at least $1- \frac{1}{n}$,
\[
\frac{1}{n} \sum_{i=1}^n W_i^2 \leq c \sigma^2 .
\]
\end{corollary}

\begin{proof}
Using the subexponential concentration inequality, we know that, with probability at least $1-\frac{1}{n}$, 

\[
\left| \frac{1}{n} \sum_{i=1}^n W_i^2 - \E W^2\right| \leq \sigma^2 \max\left( \sqrt{\frac{1}{cn} \log \frac{C}{\delta}}, \frac{1}{cn}\log \frac{C}{\delta} \right).
\]

First, let $\delta = \frac{1}{n}$. Suppose $n$ is large enough such that $ \frac{1}{cn} \log Cn < 1$. Then, we have, with probability at least $1-\frac{1}{n}$,
\begin{align*}
 \frac{1}{n} \sum_{i=1}^n W_i^2 &\leq c\sigma^2 \Big(1+\sqrt{\frac{1}{cn} \log Cn}\Big) \\
		&\leq 2 c \sigma^2.
 \end{align*}
\end{proof}

\subsubsection{Sampling Without Replacement}

\begin{lemma} (\cite{serfling1974probability}) 
Let $x_1,..., x_N$ be a finite list, $\bar{x} = \mu$. Let $X_1,...,X_n$ be sampled from $x$ without replacement. 

Let $b = \max_i x_i$ and $a = \min_i x_i$. Let $r_n = 1- \frac{n-1}{N}$. Let $S_n = \sum_i X_i$.
Then we have that
\[
\P( S_n - n \mu \geq n \epsilon) \leq \exp\left( - 2 n \epsilon^2 \frac{1}{r_n (b-a)^2}\right).
\]
\end{lemma}

\begin{corollary}
\label{cor:serfling}
Suppose $\mu = 0$. 
\[
\P\left( \frac{1}{N} S_n \geq \epsilon\right) \leq \exp\left( -2 N \epsilon^2 \frac{1}{(b-a)^2}\right)
\]
And, by union bound, we have that
\[
\P\left( | \frac{1}{N} S_n| \geq \epsilon\right) \leq 2 \exp\left( -2 N \epsilon^2 \frac{1}{(b-a)^2}\right)
\]

\end{corollary}

A simple restatement is that with probability at least $1- \delta$, the deviation $| \frac{1}{N} S_n|$ is at most $ (b-a) \sqrt{ \frac{1}{2N} \log \frac{2}{\delta}}$.

\begin{proof}
\[
\P\left( \frac{1}{N} S_n \geq \epsilon\right) = \P\left( S_n \geq \frac{N}{n} n \epsilon\right) \leq \exp\left( - 2 n \frac{N^2}{n^2} \epsilon^2 \frac{1}{r_n (b-a)^2} \right).
\]
We note that $r_n \leq 1$ always, and $n \leq N$ always.   Thus,
\[
\exp\left( - 2 n \frac{N^2}{n^2} \epsilon^2 \frac{1}{r_n (b-a)^2} \right)  \leq \exp\left( - 2 N \epsilon^2 \frac{1}{(b-a)^2}\right)
\]
completing the proof.
\end{proof}

\subsubsection{Bracketing Numbers for Convex Functions}

\begin{definition}
Let $\mathcal{C}$ be a set of functions. For a given $\epsilon$ and metric $\rho$ (which we take to be $L_2$ or $L_2(P)$), we define a \textit{bracketing} of $\mathcal{C}$ to be a set of pairs of functions $\{ (f_L, f_U) \}$ satisfying (1) $\rho( f_L, f_U) \leq \epsilon$ and (2) for any $f \in \mathcal{C}$, there exist a pair $(f_L, f_U)$ where $f^U \geq f \geq f^L$. 

We let $N_{[]}(\epsilon, \mathbf{C}, \rho)$ denote the size of the smallest bracketing of $\mathcal{C}$
\end{definition}

\begin{proposition} (Proposition 16 in \cite{kim2014global})
\label{prop:convexbracket}
Let $\mathcal{C}$ be the set of convex functions supported on $[-1, 1]^d$ and uniformly bounded by $B$. Then there exist constants $\epsilon_3$ and $K^{**}$, dependent on $d$, such that
\[
\log N_{[]} (2\epsilon, \mathcal{C}, L_2) \leq K^{**} \left( \frac{2B}{\epsilon} \right)^{d/2}
\]
for all $\epsilon \in (0, B \epsilon_3]$.
\end{proposition}

It is trivial to extend Kim and Samworth's result to the $L_2(P)$ norm for an absolutely continuous distribution $P$.

\begin{proposition}
\label{prop:convexbracket_lp}
Let $P$ be a distribution with a density $p$. Let $\mathcal{C}, B, \epsilon_3, K^{**}$ be defined as in Proposition~\ref{prop:convexbracket}. Then,
\[
\log N_{[]} (2\epsilon, \mathcal{C}, L_1(P)) \leq K^{**} \left( \frac{2B}{\epsilon} \right)^{d/2}
\]
for all $\epsilon \in (0, B\epsilon_3]$.
\end{proposition}

\begin{proof}
Let $\mathcal{C}_\epsilon$ be the bracketing the satisfies the size bound in Proposition~\ref{prop:convexbracket_lp}. 

Let $(f_L, f_U) \in \mathcal{C}_\epsilon$. Then we have that:
\begin{align*}
\| f_L - f_U \|_{L_1(P)} &= \int | f_L(x) - f_U(x)| p(x) dx \\
   &\leq \left( \int | f_L(x) - f_U(x) |^2 dx \right)^{1/2}
      \left( \int p(x)^2 dx \right)^{1/2} \\
  &\leq \left( \int | f_L(x) - f_U(x)|^2 dx \right)^{1/2}\\
 &\leq \| f_L - f_U \|_{L_2} \leq \epsilon.
\end{align*}
On the third line, we used the fact that $\int p(x)^2 dx \leq \left( \int p(x) dx \right)^2 \leq 1$.
\end{proof}

It is also simple to extend the bracketing number result to additive convex functions. As before, let $\mathcal{C}^s$ be the set of additive convex functions with $s$ components.

\begin{corollary}
\label{cor:convexadditive_lp}
Let $P$ be a distribution with a density $p$. Let $B, \epsilon_3, K^{**}$ be defined as in Proposition~\ref{prop:convexbracket}. Then,
\[
\log N_{[]}(2\epsilon, \mathcal{C}^s, L_1(P)) \leq s K^{**} 
    \left( \frac{2sB}{\epsilon} \right)^{1/2}
\]
for all $\epsilon \in (0, s B \epsilon_3]$.
\end{corollary}

\begin{proof}
Let $f \in \mathcal{C}^s$. We can construct an $\epsilon$-bracketing for $f$ through $\epsilon/s$-bracketings for each of the components $\{ f_k \}_{k=1,...,s}$:
\[f_U = \sum_{k=1}^s f_{Uk}  \qquad f_L = \sum_{k=1}^s f_{Lk} \]
It is clear that $f_U \geq f \geq f_L$. It is also clear that $\| f_U - f_L \|_{L_1(P)} \leq \sum_{k=1}^s \| f_{Uk} - f_{Lk} \|_{L_1(P)} \leq \epsilon$.
\end{proof}

The following result follows from Corollary~\ref{prop:convexbracket_lp} directly by a union bound. 

\begin{corollary}
\label{cor:convexbracket_ln}
Let $X_1,...,X_n$ be random samples from a distribution $P$. Let $1 > \delta > 0$. Let $\mathcal{C}^s_\epsilon$ be an $\epsilon$-bracketing of $\mathcal{C}^s$ with respect to the $L_1(P)$-norm whose size is at most $N_{[]}( 2\epsilon, \mathcal{C}^s, L_1(P))$. Let $\epsilon \in (0, s B \epsilon_3]$.

Then, with probability at least $1-\delta$, for all pairs $(f_L, f_U) \in \mathcal{C}^s_\epsilon$, we have that
\begin{align*}
\frac{1}{n} \sum_{i=1}^n |f_L(X_i) - f_U(X_i)| &\leq \epsilon + \epsilon_{n, \delta}
\end{align*}
where 
$$\epsilon_{n,\delta} \equiv 
sB \sqrt{ \frac{ \log N_{[]}(2\epsilon, \mathcal{C}^s, L_2(P)) + \log \frac{1}{\delta}}{2n}} 
= \sqrt{ \frac{ sK^{**}(sB)^{1/2}}{2\epsilon^{1/2}n} + \frac{1}{2n} \log \frac{1}{\delta}}.$$
\end{corollary}

\begin{proof}
Noting that $|f_L(X_i) - f_U(X_i)|$ is at most $sB$ and there are
$N_{[]}(2\epsilon, \mathcal{C}^s, L_1(P))$ pairs $(f_L, f_U)$, the
inequality follows from a direct application of a union bound and Hoeffding's Inequality.
\end{proof}

To make the expression in this corollary easier to work with, we derive an upper bound for $\epsilon_{n, \delta}$. Suppose 
\begin{align}
\epsilon^{1/2} \leq 2sK^{**} (sB)^{1/2} \quad \trm{and} \quad \log \frac{1}{\delta} \geq 2 \label{cond:simplify_covering_number}.
\end{align}
Then we have that
\begin{align*}
\epsilon_n \leq sB \sqrt{ \frac{ sK^{**} (sB)^{1/2} \log \frac{1}{\delta}}{\epsilon^{1/2}n}}.
\end{align*}

\section{Gaussian Example}
\label{sec:gaussian_example}

Let $H$ be a positive definite matrix and let $f(x_1, x_2) = H_{11} x_1^2 + 2H_{12} x_1x_2 + H_{22} x_2^2 + c$ be a quadratic form where $c$ is a constant such that $\E[f(X)] = 0$. Let $X \sim N(0, \Sigma)$ be a random bivariate Gaussian vector with covariance $\Sigma = [1, \alpha; \alpha, 1]$ 

\begin{proposition}
Let $f^*_1(x_1) + f^*_2(x_2)$ be the additive projection of $f$ under the bivariate Gaussian distribution. That is,
\begin{align*} 
f^*_1, f^*_2 \equiv \argmin_{f_1, f_2} \Big\{ \E \left(f(X) - f_1(X_1) - f_2(X_2)\right)^2 \,:\, \E[f_1(X_1)] = \E[f_2(X_2)] = 0 \Big\}
\end{align*}

Then, we have that 
\begin{align*}
f^*_1(x_1) &= \left( \frac{T_1 - T_2 \alpha^2}{1 - \alpha^4} \right) x_1^2 + c_1 \\
f^*_2(x_2) &= \left( \frac{T_2 - T_1\alpha^2}{1 - \alpha^4} \right) x_2^2 + c_2
\end{align*}
where $T_1 = H_{11} + 2H_{12} \alpha + H_{22} \alpha^2$ and $T_2 = H_{22} + 2H_{12} \alpha + H_{11} \alpha^2$ and $c_1,c_2$ are constants such that $\E[f_1^*(X_1)] = \E[f_2^*(X_2)] = 0$.
\end{proposition}

\begin{proof}
By Lemma~\ref{lem:general_int_reduction}, we need only verify that $f^*_1, f^*_2$ satisfy
\begin{align*}
f^*_1(x_1) &= \E[ f(X) - f_2^*(X_2) \given x_1] \\
f^*_2(x_2) &= \E[ f(X) - f_1^*(X_1) \given x_2 ] .
\end{align*}
Let us guess that $f^*_1, f^*_2$ are quadratic forms $f^*_1(x_1) = a_1 x_1^2 + c_1$, $f^*_2(x_2) = a_2 x_2^2 + c_2$ and verify that there exist $a_1, a_2$ to satisfy the above equations. Since we are not interested in constants, we define $\simeq$ to be equality up to a constant. 
Then, 
\begin{align*}
&\E[ f(X) - f_2^*(X_2) \given x_1]\\
& \simeq \E[ H_{11} X_1^2 + 2H_{12}X_1 X_2 + H_{22}X_2^2 - a_2 X_2^2 \given x_1 ] \\ 
    &\simeq H_{11} x_1^2 + 2 H_{12} x_1 \E[X_2 \given x_1] + H_{22} \E[X_2^2 \given x_1] - a_2\E[X_2^2 \given x_1] \\
   &\simeq H_{11} x_1^2 + 2H_{12} \alpha x_1^2 + H_{22} \alpha^2 x_1^2 - a_2 \alpha^2 x_1^2\\
   &\simeq (H_{11} + 2 H_{12} \alpha + H_{22} \alpha^2 - a_2 \alpha^2) x_1^2.
\end{align*}
Likewise, we have that
\[
\E[ f(X) - f_1^2(X_1) \given x_2] \simeq (H_{22} + 2H_{12}\alpha + H_{22}\alpha^2 - a_1 \alpha^2) x_2^2.
\]
Thus, $a_1, a_2$ need only satisfy the linear system
\begin{align*}
T_1 - a_2 \alpha^2 &= a_1 \\
T_2 - a_1 \alpha^2 &= a_2 
\end{align*}
where $T_1 = H_{11} + 2H_{12} \alpha + H_{22} \alpha^2$ and $T_2 = H_{22} + 2H_{12} \alpha + H_{11} \alpha^2$.
It is then simple to solve the system and verify that $a_1, a_2$ are as specified.
\end{proof}

\end{document}